\documentclass[11pt]{article}
\usepackage{indentfirst}

\usepackage{graphicx}
\usepackage{color,cite} 
\usepackage[titletoc,title]{appendix}
\usepackage{amsmath,amsfonts,amssymb,amsthm}
\def\la{\big\langle}
\def\ra{\big\rangle}

\def\ds{\displaystyle}

\def\L{{\bf L}}

\def\bfv{{\bf v}}

\def\bfw{{\bf w}}
\def\bfu{{\bf u}}
\def\bfb{{\bf b}}
\def\bfc{{\bf c}}

\def\bfn{{\bf n}}
\def\bfe{{\bf e}}

\def\bfp{{\bf p}}
\def\bfq{{\bf q}}

\def\bfA{{\bf A}}

\def\bpm{\begin{pmatrix}}
\def\epm{\end{pmatrix}}

\def\S{{\cal S}}
\def\F{{\cal F}}

\def\M{{\cal M}}

\def\N{{\cal N}}

\def\R{{\mathbb R}}

\def\rank{\hbox{\rm Rank}}

\def\vp{\varphi}

\def\P{{\cal P}}

\def\vs{\vskip 2em}

\def\v{\vskip 1em}
\def\D{{\cal D}}
\def\O{{\cal O}}

\def\begi{\begin{itemize}}
\def\endi{\end{itemize}}
\def\C{{\cal C}}
\def\div{\hbox{div}}
\def\ov{\overline}
\def\Tilde{\widetilde}

\def\bega{\begin{array}}
\def\enda{\end{array}}

\def\bel{\begin{equation}\label}
\def\eeq{\end{equation}}
\def\sqr#1#2{\vbox{\hrule height .#2pt
\hbox{\vrule width .#2pt height #1pt \kern #1pt
\vrule width .#2pt}\hrule height .#2pt }}
\def\square{\sqr74}
\def\endproof{\hphantom{MM}\hfill\llap{$\square$}\goodbreak}

\newtheorem{theorem}{Theorem}[section]

\newtheorem{lemma}{Lemma}[section]

\newtheorem{remark}{Remark}[section]

\begin{document}
\title{\bf Development of singular surfaces and formation of generic singularities in 3D pressureless flows}
\vs

\author{Geng Chen$^{*}$, Faris El-Katri$^{*}$, and Shoujun Huang$^{\dag}$\\
\, \\
$^{*}$Department of Mathematics, University of Kansas,\\
 Lawrence, KS 66045, USA.\\
$^\dag$College of Mathematical Medicine, Zhejiang Normal University, \\
Jinhua 321004, P. R. China.\\
\, \\
E-mails: gengchen@ku.edu,~elkatri@ku.edu,~sjhuang@zjnu.edu.cn.}
\maketitle
\v
{\bf Abstract:} In this paper,  the Cauchy problem for the equations of pressureless flows in three space dimensions is considered. It is well known that for smooth initial data, the density and velocity gradient often become unbounded at a finite time $t_0$. This type of singularity is associated with sticky-particle dynamics. For an open dense set of smooth initial data, the formation and structure of singularities are investigated using the transversality theorem, with the aid of constructive methods for initial velocity fields. In contrast to the two‑dimensional case studied in \cite{BCH2}, for $t>t_0$ a singular hypersurface will be formed, whose governing system of equations is not hyperbolic. By assuming the initial data are analytic, we are able to precisely analyze the singular behavior of velocity and density. Based on this, the system of equations for singular hypersurfaces can be recast into a first-order system of Briot-Bouquet type, and the local analytic solution can be constructed using a version of the Cauchy-Kovalevskaya theorem.


\section{Introduction}\label{s:1}
\setcounter{equation}{0}

In this paper, we shall consider the initial value problem for the equations of pressureless flows
in three space dimensions (3D):
\begin{equation}\label{SP}
\left\{\begin{array}{rl}
\!\!\partial_t\rho+\div( \rho \bfv)&=0, \cr
\!\!\partial_t(\rho \bfv)+\div\big( \rho \bfv\otimes \bfv\big)&=0,\enda
\right.  \qquad\quad  t\in  ]0,T[\,,\quad x\in\R^3,\eeq
\bel{ic}\rho(0,x)~=~\bar \rho(x), \qquad \bfv(0,x)~=~\bfw(x).\eeq

It is well-known that as long as the solution remains smooth, it can be directly computed by the method
of characteristics.  Indeed, we have
\begin{align}
	\mathbf{v}(x+t\mathbf{w}(x))&=\mathbf{w}(x), 
	\label{Eqn:velocityTransport}\end{align}
and 
	\begin{align}
	\rho(x+t\mathbf{w}(x))&=\frac{\bar\rho(x)}
	{\det(I+tD\mathbf{w}(x))}. 
	\label{Eqn:densityTransport}
\end{align}
The above formulas are valid up to the first time $\tau(x)$
where a singularity occurs, which happens as soon as the matrix
$I+ tD\bfw(x)$ is no longer invertible.
In other words, the blowup time $\tau(x)$ is characterized by the smallest time $t>0$
such that $-1/t$ is an eigenvalue of $D\bfw(x)$, for some $x\in\R^n$.

The equations of pressureless flows have extensive and vital applications in research fields including the theory of large-scale cosmic structure formation, the dynamics of sticky particle models, and flocking dynamics.
E et al. \cite{ERS} constructed global weak solutions for two adhesion-particle conservation law systems via a generalized variational principle. Solutions lack uniqueness under standard entropy conditions, while the modified gravitational model yields unique weak solutions. For continuous nowhere-differentiable random initial data, the density becomes a pure point measure with dense shock locations for any $t>0$.  In cosmology, the equations of  pressureless flows correspond to the approximation proposed by Zeldovich \cite{Z}, later known as the Zeldovich approximation \cite{GSS}. Under weak gravitational perturbations, dark matter particles move freely along their respective initial velocity fields until the trajectories of neighboring particles intersect. Sheet-like and other density-concentrated structures form at the focal points of characteristics, and this approximation facilitates the understanding of large-scale cosmic structures \cite{GSS}.
In the dynamics of sticky particles \cite{Brenier}, particles stick together and move at the same velocity upon collision without interpenetrating each other. Meanwhile, the connection between discrete particle systems and scalar conservation laws, as well as measure-valued solutions, existence and convergence of solutions, have been thoroughly investigated, laying a rigorous mathematical foundation for the sticky particle model and the evolution of large-scale cosmic structures.

In one space dimension (1D), Danilov \cite{D} showed that the solution of (\ref{SP}) will generically contain a finite number of point masses, which can
eventually collide with each other. Huang and Wang \cite{hw} considered the well-posedness of pressureless flows in the 1D case. They proved the uniqueness of the weak solution when the initial data is a Radon measure, where the energy is required to be initially weakly continuous besides the Oleinik entropy
condition.

For the case of higher space dimensions, Sever \cite{Sever} presented an existence of a weak solution of the equations of pressureless flows (\ref{SP}), when both the  mass and kinetic energy are finite initially.  However, Bressan and Nguyen \cite{BN} pointed out that  the existence and the uniqueness of solutions to the Cauchy problem (\ref{SP})-(\ref{ic}) can fail, even for $\L^\infty$ initial data. Counterexamples were also introduced in \cite{BN}. Recently, Bianchini and Daneri \cite{BD} found that for ``almost all" measure-valued initial data,
the particles to (\ref{SP})-(\ref{ic}) never hit each other and
their motion is rectilinear, with constant speed.

Most recently, Bressan, Chen and Huang  \cite{BCH, BCH2} gave a new framework to study the generic singularities of pressureless flows (\ref{SP}) in the two-dimensional (2D) case. It was pointed out that Eq. (\ref{SP}) shall give rise to singular curves, where the mass has positive density w.r.t. one-dimensional Hausdorff measure. One of the important features  of
the governing equations for these curves is that  the corresponding Cauchy problem is ill posed and thus the local solutions have to be constructed only within a class of analytic data. Under generic assumptions on initial data, the local analytical solutions can be determined by
using a version of the Cauchy–Kovalevskaya theorem \cite{BCH2}. 

Based on the framework in \cite{BCH, BCH2}, in this paper our goal is to understand the formation
and the time evolution of the singular hypersurfaces, which arise from  the pressureless flows (\ref{SP})-(\ref{ic}) in the 3D case.

Now we briefly introduce our main result.
In this study, the space of three times continuously differentiable vector fields
$\bfw=(w_1,w_2, w_3)$ on $\R^3$  is denoted by
$$\C^3(\R^3;\R^3)~\doteq~\left\{ \bfw:\R^3\mapsto\R^3\,;~~\|\bfw\|_{\C^3}\doteq
\sup_{x\in \R^3} \sup_{|\alpha|\leq 3, ~i=1,2,3} |D^\alpha w_i(x)|~<~+\infty\right\}.$$
We consider the generic initial data within the subspace of $\C^3$
vector fields whose Jacobian
matrix $D\bfw$ vanishes as $|x|\to +\infty$, which  is denoted by
\bel{C30}\C^3_0(\R^3;\R^3)~\doteq~\left\{ \bfw\in \C^3(\R^3;\R^3),~~\lim_{|x|\to +\infty} D\bfw(x)~=~ {\bf 0}
\right\}.\eeq
It is obvious  that $\C^3_0$ is a closed subspace of $\C^3$ and is a Banach space as well.
As a consequence, for any given
$\bfw\in \C^3_0$ and $T>0$, we can find $R>0$ large enough so that
for all $t\in [0,T]$ one has
\bel{wreg}\det\bigl(I+tD\bfw(x)\bigr)~>~0\qquad\hbox{for all}~~t\in [0,T],~~|x|\geq R.\eeq
This implies that for $t\in [0,T]$, the only singularities of the solution to
(\ref{SP})-(\ref{ic}) can occur within a bounded domain.

Our first main result is on the generic singularity of (\ref{SP}). 
More precisely, assume that the initial density $\bar\rho$ and  velocity field $\mathbf{w}$ are analytic on $\D$, an open set in $\mathbb{R}^3$. We make the following assumptions. 

({\bf A1})	For every $x\in\mathcal{D}$, the matrix $D_x\mathbf{w}$ has three distinct real eigenvalues $\lambda_1(x)<\lambda_2(x)<\lambda_3(x). $
Moreover, the initial density $\bar\rho(x)>0$. 
\\\indent
({\bf A2}) The eigenvalue $\lambda_1(x)$ attains a strictly negative minimum value $\lambda_1(x_0)<0$ at some point $x_0\in\mathcal{D}$. Moreover, its Hessian $D^2\lambda_1(x_0)$ is strictly positive definite. 

Actually, the assumption ({\bf A2}) on the Hessian matrix  is indeed satisfied for generic initial data $\bfw(x)$, as in the 2D case.

By selecting proper perturbation of velocity fields and using the Thom's Transversality Theorem, we can describe the structure of generic singularities of solutions for pressureless flow (\ref{SP}), starting from an open dense set of initial data. The precise statement of the result will be given in Theorem \ref{l:83}.
The proof of this result is based on the one for the 2D solution in \cite{BCH}.
But because of three space dimensions, this argument becomes relatively more complicated than that for 2D case. And we need several new treatments for the coupling present due to the three dimensional setting. Here the use of Thom's transversality theorem is a successful approach to studying the generic structure of solutions, first used by Bressan and the first author for the variational wave equation, a 1D model for the liquid crystals, in \cite{BC}. This method has been recently applied to the compressible Euler equations in \cite{CHS}.

The next goal is to study the detailed behavior of solutions near the point where the singularity of (\ref{SP})-(\ref{ic}) forms and then to show the existence of local solutions after the formation of singularities, i.e. the solutions including singular analytical hypersurfaces. 

To do that, based on the conservation of mass and momentum of fluids, we derive a system of equations for the singular hypersurfaces. This is one of the most innovative contributions of this paper. Under the map $x\mapsto (\tau(x), x+\tau(x)\bfw(x))$, the image of the surface $\Gamma^*$ defined by (\ref{g*}) is also a surface in four-dimensional spacetime $\mathbb{R}^+\times\mathbb{R}^3$, which is represented by $\Gamma^{\sharp}$.  When $t>t_0=\tau(x_0)$, the singular hypersurface $\mathcal{S}$ will be formed and begin to develop, whose boundary should lie on the surface $\Gamma^{\sharp}$ (see Figure 1). By carefully making use of the properties of eigenspaces of $D\bfw(x)$ and their dual spaces, we succeed in analyzing the asymptotic singular behavior of velocity and density as $t$ tends to $\tau(x_0)$, which is essentially similar  to the 2D case, but needs different ideas. Afterwards, we are able to rewrite the equations for $\mathcal{S}$ in a compact form of Fuchs type or Briot–Bouquet type. Finally, we can obtain the local analytical solution for $\mathcal{S}$ using a version of the Cauchy-Kovalevskaya theorem. It also implies that for 3D pressureless flows, we can construct local piecewise smooth solutions, containing analytical singular surfaces for generic initial data. Our main result in this part will be given in Theorems \ref{Thm:ZpmForm} and \ref{Theorem:vpm} for structure of solution near the point where singularity forms and Theorem \ref{thm5.1} for local existence of solutions beyond singularity formation.

There have been several works on  the singularities of solutions to nonlinear PDEs for generic data by Damon \cite{Damon} \& Guckenheimer \cite{Guck}. Generic singularities of solutions to a class of second-order variational wave equation were studied by Bressan, Chen and their collaborators \cite{BC,BHY,CCD}. For one-dimensional scalar conservation laws, the properties of generic solutions had been studied extensively by Golubitsky \& Schaeffer \cite{GS}, Guckenheimer \cite{Gk2}, Schaeffer \cite{S} and Chen, Hu \& Shen \cite{CHS}. Kong \cite{Kong} provided an asymptotic description of this blowup, and also of the initial stages of the new shock, in a generic setting in one space dimension. A study of shock formation for multidimensional scalar conservation laws can be found in Danilov \& Mitrovic \cite{DM}. For systems of conservation laws, Alinhac \& Christodoulou described the structure and development of shocks in the multidimensional setting \cite{A,Al,Ali,Ch1}. Later work by Buckmaster, Shkoller \& Vicol studied the 3-dimensional equations of isentropic gas dynamics and the blowup asymptotics leading to shock formation in \cite{BSV,Vicol0,Vicol1,Vicol3} and many other papers. In radial symmetry, Chen, El-Katri, Hu \& Shen established global existence of classical solutions to the isentropic compressible Euler equations for supersonic rarefactive flows and singularity formation for strongly compressive flows \cite{CEHS}; see also \cite{CCW}.

This paper is organized as follows.  In Section 2, the formation of generic singularities is discussed in detail.  In Section 3, we derive the equations of singular hypersurfaces for pressureless flows.  Section 4 is devoted to studying the detailed asymptotic expansion of the velocity $\bfv^{\pm}$ and of the density $\rho^{\pm}$ of the
pressureless flows in front of and behind the singular hypersurface.  Finally, in Section 5 we are able to construct the singular surface during its initial stages,
valid for analytic data. For readers' convenience, some details of computation have been gathered in the Appendix.

\section{Formation of generic singularities in 3D}
\label{s:3}
\setcounter{equation}{0}
For readers' convenience, we review some basic
definitions and results in transversality theory.
For the proofs, one can refer to  \cite{Bloom, GG}.

{\bf Definition}. {\it Let $f:\;X\mapsto Y $ be a  smooth map of
manifolds and let $W$ be a submanifold of $Y$. We say that $f$ is
transversal to $W$ at a point $p\in X$, and write $f\pitchfork_p W$,
if
\begin{itemize}\item either $f(p)\notin W$,
\item {\text or else} $f(p)\in W$ and $T_{f(p)}Y=(df)_p(T_pX)+T_{f(p)}W.$\end{itemize}We say that $f$ is transverse to $W$, and
write $f\pitchfork W$, if $f\pitchfork_p W$ for every point $p\in
X$.}

 Here $T_pX$ denotes the tangent space at the
point $p\in X$, while $T_qY$ and $T_qW$ denote, respectively, the
tangent spaces to $Y$ and to $W$ at the point $q\in W\subset Y$.
Finally, $(df)_p:\;T_pX\mapsto T_{f(p)}Y$ denotes the differential
of the map $f$ at $p$.

{\bf Transversality Theorem}. {\it  Let $X, \Theta$ and $Y$ be smooth
manifolds, $W$ a submanifold of $Y$. Let $\theta\mapsto
\phi^{\theta}$ be a smooth map which to each $\theta\in \Theta$
associates a function $\phi^{\theta}\in \C^{\infty}(X,Y)$, and
define $\Phi: X\times\Theta\mapsto Y$ by setting
$\Phi(x,\theta)=\phi^{\theta}(x)$. \\
If $\Phi\pitchfork W$, then the
set $\{\theta\in\Theta:\;\phi^{\theta}\pitchfork W\}$ is dense in
$\Theta$.}

In this section, our aim is  to understand the formation of new singularities for a solution to
(\ref{SP}), provided with a generic set of
smooth initial data.
Given $T>0$ and a vector field $\bfw\in \C^2_0(\R^3;\R^3)$, for $r=0,1,2$
we consider the
sets
\bel{Sdef}
\S_r~\doteq~\Big\{(t,x)\,;~~0<t<T,~~\rank (I+ tD\bfw(x))=r\Big\}~\subset~
 \R\times\R^3.\eeq
By Proposition~3.2.6 in \cite{Bloom} or Lemma 5.2 in \cite{GG}, we have 
\begin{lemma}\label{l:corank}
Within the space  of all $3\times 3$ matrices, 
for every $r=0,1,2$, the set
 \bel{MR}\M_r~\doteq~\{A\in \R^{3\times 3}\,;~~\hbox{\rm rank}(A)=r\}\eeq
 is an embedded  submanifold of dimension
$r(6-r)$.
\end{lemma}

Now we have the following results.
\begin{lemma}\label{lemma2}  There exists an open dense set of
vector fields $\F_1\subset \C^3_0(\R^3;\R^3)$ such that for every ${\bf w}\in\F_1$,
one has
\begin{equation}\label{emp3}\S_0~\doteq~\Big\{(t,x)\in [0, T]\times\R^3\,,~~
I+tD_x{\bf w}(x)={\bf 0}\Big\}
~=~\emptyset\end{equation}
and the set
\begin{equation}\label{emp4}\S_1~\doteq~\Big\{
(t,x)\in [0, T]\times\R^3\,,~~\rank(I+tD_x{\bf w}(x))=1\Big\}\end{equation}
consists of finitely many points.
\end{lemma}

 {\bf Proof.} The proof for (\ref{emp3}) relies on the same argument used in
 Lemma~5.3 in \cite{BCH}. Thus, the details are omitted here.
 In the following, we shall focus on the second statement.

 {\bf 1.} Let $\bfw$ be a vector field for which (\ref{emp3}) holds.
  Consider a point $(t_0, x_0)\in \S_1$. Then we can choose a system of independent and dependent coordinates $x=(x_1, x_2, x_3)$
and $y=(y_1, y_2, y_3)$ such that, in these coordinates, the
Jacobian matrix takes the form
 \bel{ran1} I+t_0D\bfw(x_0)~=~
\left[\bega{cc} a & \bfb\cr \bfc^\top & D\enda \right].\eeq
Here $a$ is a non-zero constant, $\bfb, \bfc\in \R^{2}$ are row vectors, while $D$ is a $2\times2$ matrix.
The condition $\rank(I+t_0D_x{\bf w}(x_0))=1$ can be locally
represented in terms of the following equations
\bel{ran2}\phi_{ij}~\doteq~\left(D-\bfc^\top a^{-1}\bfb\right)_{ij}~=~0,\qquad\qquad i,j\in \{1,2\}.\eeq

{\bf 2.}  Consider the family of perturbed vector fields
\bel{ran4}\bfv(\theta)(x)~=~\bfw(x)+
\sum_{\alpha,\beta=2}^3\theta_{\alpha\beta}\bfp_{\alpha\beta}(x) \eeq
where, for $\alpha,\beta=2,3$,
\bel{ran5}\bfp_{\alpha\beta}(x)~\doteq~\vp(x-x_0)\cdot  [ D^{(\alpha\beta)}(x-x_0)]\,.\eeq
Here $D^{(\alpha\beta)}$ denotes the $3\times 3$
matrix whose entries are all zero, except for the element in the
$\alpha$-th row and $\beta$-th column, which is 1. In addition, $\varphi\in\C_0^3(\R^n)$ is defined by
\bel{test}\left\{\bega{cl} \vp(x)\,=\,1\quad&\hbox{if} ~~|x|\leq 1,\\[3mm]
 \vp(x)\,\in\,[0,1]\quad &\hbox{if} ~~1\leq |x|\leq 2 ,\\[3mm]
  \vp(x)\,=\,0\quad&\hbox{if} ~~|x|\geq 2,\enda\right.\eeq 
The perturbed Jacobian  matrix can now be written as
\begin{equation}\label{ma}I+t[D_x{\bf v}(\theta)](x)   ~=~
\left[\bega{cc} a & {\bfb}\cr ({\bfc}+\Tilde {\bfc})^\top & D+\Tilde D\enda \right],
\end{equation}
with
\begin{equation}\label{ran6}\Tilde{\bfc}~=~t\left(\begin{array}{c}\sum_{j=2}^3\theta_{2
j}(x_j-x_{0j})\varphi_{x_{1}}\vspace{3mm}\\\sum_{j=2}^3\theta_{3
j}(x_j-x_{0j})\varphi_{x_{1}}\end{array}\right),\end{equation}\begin{equation}\label{ran7}
\Tilde D~=~t\left(\sum_{j=2}^3\theta_{\alpha
j}\Big((x_j-x_{0j})\varphi\Big)_{x_{\beta}}\right)_{\alpha,\beta=2,3}.\end{equation}

{\bf 3.}
Introducing the functions $$\Phi_{ij}(t,x,\theta)~\doteq~\Big(D+\Tilde D-(\bfc+\Tilde \bfc)^\top a^{-1}\bfb\Big)_{ij}\,,\qquad i,j\in \{2,3\},$$
we claim that the map $(t,x,\theta)\mapsto \Phi(t,x,\theta)\doteq(\Phi_{ij})_{i,j=2,3}$ is transversal to the zero-dimensional manifold $\{{\bf 0}\}\subset\R^{4}$
at the point $(t_0,x_0)$.
By the definition of transversality, it suffices
to show that the $4\times4$ Jacobian matrix of partial derivatives
of $\Phi$ w.r.t.~$\theta=(\theta_{\alpha\beta})$ has rank $4$.  Indeed, at
the point $(t_0, x_0)$, an explicit computation
yields $\Tilde \bfc^\top=0$ and
$${\partial \Phi_{ij}\over\partial \theta_{\alpha\beta}}(t_0, x_0, \theta)
~=\frac{\partial
(\Tilde D)_{ij}}{\partial\theta_{\alpha\beta}}(t_0,x_0,\theta)~=~t_0\cdot
\delta_{i\alpha}\,\delta_{j\beta}~=~\left\{ \bega{cl} t_0\quad
&\hbox{if}\quad i=\alpha
~~\hbox{and}~~j=\beta,\\[3mm]
0\quad &\hbox{otherwise.}\enda\right.$$
We thus conclude that the map $\Phi(t,x,\theta)$ is transversal to the
manifold $\{\bf 0\}$ for $(t,x)$ in a neighborhood of $(t_0, x_0)$ and $\theta$ sufficiently close to zero.

{\bf 4.}
By (\ref{emp3}) it follows  $$\S_1=\Big\{(t,x);\;\rank(I+tD_x{\bfw}(x))=1\Big\}=\Big\{(t,x);\;\rank(I+tD_x{\bfw}(x))\leq1\Big\},$$
hence the set $\S_1$ in (\ref{emp4}) is closed.
Since $\bfw\in \C^3_0$, the  decay of $\nabla \bfw(x)$ as $|x|\to \infty$ implies that
the set $\S_1$ is compact.

Similar to Step 4 in the proof of Lemma~5.3 in \cite{BCH},
we can cover the compact set $\S_1$
with finitely many neighborhoods $\N_{(t_i, x_i)}$, $i=1,\ldots,m$.
Combining all these perturbations, we construct a family of perturbations
depending on $4m$ parameters, such that the map $\Phi(t,x,\theta)$ is transversal to the zero manifold $\{\bf 0\}\subset \R^4$ at every point in $\S_1$ and $\theta\in\R^{4m}$ sufficiently close to zero. Trivially, this map is also transversal at points
$(t,x)\notin \cup_i \N_{(t_i, x_i)}$ where the rank is greater than $ 2$.

By the Transversality Theorem,  for almost every $\theta$ in a neighborhood of
${\bf 0}\in \R^{4m}$, the map $(t,x)\mapsto \Phi(t,x,\theta)\doteq(\Phi_{ij})$
 is transversal to the zero manifold $\mathcal{M}=\{\bf 0\}$.
  Therefore, whenever
 \bel{rank2}\rank(I+t D_x[\bfv(\theta)](x)) ~=~ 1,\eeq
 the $4\times 4$ Jacobian matrix of partial derivatives
$D_{(t,x)}\Phi(t,x)$ (w.r.t.~the variables $t, x_1, x_2,x_3$) has full rank $4$.
Hence, for the vector field $\bfv(\theta)$, the set
of points $(t,x) $ where (\ref{rank2}) holds consists of isolated points.

The previous arguments show that there exists a dense set
$\F_1\subset\C_0^3(\mathbb{R}^3;\R^3)$ such that, for every ${\bf w}\in\F_1$,
the corresponding set $\S_1$ in (\ref{emp4}) contains  at most finitely many points.

{\bf 5.} It remains to prove that the set $\F_1$ is open.
Let $\bfw\in \F_1$.   Then by assumption there exist at most finitely many
points $(t_k, x_k)\in [0,T]\times\R^3$, $k=1,\ldots,N$, where the matrix $I+t_k D_x \bfw(x_k)$ has rank 1.

For each one of these points we can choose a neighborhood $\N_{(t_k, x_k)}$
and a local system of coordinates such that
the condition $\rank(I+t D_x \bfw(x)) = 1$ can be written as
\bel{PIJ}\Phi_{ij}(t,x)=0\qquad i,j\in \{1,2\}.\eeq
Moreover, the $4\times 4$ Jacobian matrix of partial derivatives
$D_{(t,x)} \Phi_{ij}(t,x)$ has full rank.

By the implicit function theorem, for every vector field $\Tilde\bfw$ with
$\|\Tilde\bfw-\bfw\|_{\C^3}$ sufficiently small, there exists a unique point $(\tilde t_k, \tilde x_k)\in \N_{(t_k, x_k)}$ such that
 $\rank(I+\tilde t_k D_x \Tilde \bfw(\tilde x_k)) = 1$.
On the other hand, since the rank of a matrix is lower semicontinuous,
for all $(t,x)\notin \cup_k\N_{(t_k, x_k)}$ we have $\rank(I+t D_x \bfw(x)) \geq 2$.

This shows that $\Tilde\bfw\in \F_1$ as well, proving that the set $\F_1$ is open. The proof is completed. 
 \endproof

As mentioned before, our main interest in this section  is to understand the formation of a singularity in a solution to (\ref{SP}) with  generic initial data (\ref{ic}) in $\R^3$.
Given a vector field $\bfw\in \C^3_0(\R^3;\,\R^3)$, we consider the compact set
$$K_0~\doteq~\Big\{(t,x)\in [0,T]\times\R^n\,;~~\det(I+t D_x\bfw(x)) = 0\Big\},$$
which defines a multifunction with closed and compact values
\bel{taud}\bega{rl}\tau(x)&=~\Bigl\{t\in [0,T]\,;~~\det(I+t D_x\bfw(x)) = 0\Bigr\}\\[4mm]
&\ds=~
\left\{- {1\over\lambda}\,;~~\hbox{$\lambda$
is a real  eigenvalue of the matrix $D_x\bfw(x)$, with  $\lambda\leq  -{1\over T}$}\right\}.
\enda\eeq
Clearly, $\tau(x)$ can have at most $3$ distinct values.

In the following, $\S_1\subset \R^{1+3}$ is the set defined at (\ref{emp4}). Now we have 
\begin{lemma}\label{lemma3}   There exists an open dense set of vector fields
$\F_2\subset\C^3_0(\R^3;\R^3)$ with the following property.
For every $\bfw\in \F_2$,
 the multifunction $\tau(\cdot)$ cannot have a local minimum
at any  point $(t_0, x_0)$.
\end{lemma}

{\bf Proof.}
We need to show that, for a generic vector field $\bfw\in \C^3_0$,
if $(t_0, x_0)\in \S_1$, then there exist points $(t,x)$ arbitrarily close to
$(t_0, x_0)$, such that
\bel{bmin}t\,<\,t_0,\qquad\qquad \det(I+t D_x\bfw(x)) ~=~ 0.\eeq

{\bf 1.} To understand what is involved here, we assume $(t_0, x_0)\in \mathcal{S}_1$.
Let
\begin{equation}
\label{lam12}    
-{1\over t_0}~=~\lambda_1~=~\lambda_2~<~\lambda_3~<~0
\end{equation}
be the eigenvalues of the matrix $D\bfw(x_0)$.  By assumption,
$$\rank(D\bfw(x_0) -\lambda_1I) ~=~1.$$
For $x$ close to $x_0$, the matrix $D\bfw(x)$ will have an isolated
eigenvalue close to $\lambda_3$.   Call $\ell_3(x)$ a left eigenvector
corresponding to this eigenvalue, and let $E(x)=\ell_3^\perp$ be the
2-dimensional orthogonal subspace.
We can now consider the map
\bel{Mx}x~\mapsto~ M(x)~=~D\bfw(x)\bigg|_{E(x)}\,.\eeq
For each $x\in \R^3$, the map (\ref{Mx}) associates a $2\times 2$ matrix
obtained by restricting the Jacobian $D_x\bfw$ to the 2-dimensional
invariant subspace $E(x)$.

Notice that, for $x=x_0$, by assumption we have
$$M(x_0)~=~\begin{pmatrix} \lambda_1&0\cr 0&\lambda_1\end{pmatrix}.$$
We now study the eigenvalues of this matrix for $x$ close to $x_0$.
By writing a Taylor approximation for $y\approx 0\in \R^3$, we have
\bel{M2}
M(x_0+y)~\approx~\begin{pmatrix} \lambda_1+ \bfb_{11}\cdot y&\bfb_{12}\cdot y\cr \bfb_{21}\cdot y&\lambda_1+\bfb_{22}\cdot y\end{pmatrix}.\eeq
Here $\bfb_{ij}$ are row vectors in $\R^3$. 

We compute the eigenvalues of $M(x_0+y)$ in (\ref{M2}).
The sum of these two eigenvalues is the trace, hence
\bel{M3}\lambda_1(x_0+y)+\lambda_2(x_0+y)~=~\lambda_1+\lambda_2+ (\bfb_{11}+\bfb_{22})\cdot y.\eeq
First, we need to find a perturbation
$y$ such that $(\bfb_{11}+\bfb_{22})\cdot y<0$, (note $\bfb_{11}+\bfb_{22}\not= 0\in\R^3$ is clearly a generically true assumption), so by \eqref{lam12}, 
$$\lambda_1(x_0+y)~<~\lambda_1\,.$$
On the other hand, we must make sure that the eigenvalues
$\lambda_1,\lambda_2$ of $M(x_0+y)$ are real.
Thus, we need to study the equation
\bel{char}\lambda^2 -\Big[ 2\lambda_1 +(\bfb_{11}+\bfb_{22})\cdot y\Big]\lambda +
\Big[ (\lambda_1+ \bfb_{11}\cdot y)(\lambda_1+\bfb_{22}\cdot y) - (\bfb_{12}\cdot y)
(\bfb_{21}\cdot y)\Big]~=~0,\eeq
and check the sign of the discriminant:
\bel{discr}\bega{rl}
\Delta&\ds=~\Big[ 2\lambda_1 +(\bfb_{11}+\bfb_{22})\cdot y\Big]^2 - 4 \Big[ (\lambda_1+ \bfb_{11}\cdot y)(\lambda_1+\bfb_{22}\cdot y) - (\bfb_{12}\cdot y)
(\bfb_{21}\cdot y)\Big]
\\[4mm]
&\ds =~\Big[ (\bfb_{11}\cdot y) +(\bfb_{22}\cdot y)\Big]^2 - 4 (\bfb_{11}\cdot y)(\bfb_{22}\cdot y)
+4(\bfb_{12}\cdot y)
(\bfb_{21}\cdot y) \\[4mm]
&\ds=~\Big[ (\bfb_{11}\cdot y) -(\bfb_{22}\cdot y)\Big]^2 +4(\bfb_{12}\cdot y)
(\bfb_{21}\cdot y)
.\enda
\eeq
In the end, we need to find a vector $y\in \R^3$ such that
\bel{M4}
(\bfb_{12}\cdot y)(\bfb_{21}\cdot y)~>~0,\qquad \qquad (\bfb_{11}+\bfb_{22})\cdot y~
< ~0.\eeq

To show such a vector $y$ exists under the generic conditions on the vectors $\bfb_{ij}$, we only need to prove that $\bfb_{12}$, $\bfb_{21}$, $(\bfb_{11}+\bfb_{22})$ are
linearly independent. In fact, assume $\bfb_{12}$, $\bfb_{21}$, $(\bfb_{11}+\bfb_{22})$ are
linearly independent, then there exists a vector $z$ on the plane spanned by $\bfb_{12}$ and $\bfb_{21}$, such that 
$\bfb_{12}\cdot z$ and $\bfb_{21}\cdot z$ are both nonzero and have the same sign. Then still using $\bfb_{12}$, $\bfb_{21}$, $(\bfb_{11}+\bfb_{22})$ are
linearly independent, and either choosing $y=z$ or $y=-z$, we have 
$$(\bfb_{11}+\bfb_{22})\cdot y\,<\,0.$$
This would yield the desired conclusion.

In the next part, we will include a proof showing that $\bfb_{12}$, $\bfb_{21}$, $(\bfb_{11}+\bfb_{22})$ are
linearly independent, under the generic conditions on the vectors $\bfb_{ij}$.

{\bf 2.} By Lemma~\ref{lemma2}, there is an open dense set $\F_1\subset\C^3_0$
of vector fields for which the set $\S_1$ consists of finitely many isolated points.
By the previous proof, we can also assume the full-rank condition:
\begi
\item[{\bf (FR)}]  {\it For each $(t_0, x_0)\in\S_1$, after a suitable change of coordinates, the condition $\rank(I+t D_x \bfw(x)) = 1$ can be written
in terms of the 4 scalar equations
(\ref{PIJ}).
Moreover, at the point
$(t_0, x_0)$, the $4\times 4$ Jacobian matrix of partial derivatives
$D_{(t,x)} \Phi_{ij}(t,x)$ has full rank.}
\endi

Based on the analysis in step {\bf 1}, we introduce an additional  condition
of linear independence:
\begi
\item[{\bf (LI)}] {\it At every point $(t_0, x_0)\in\S_1$
one can find a local system of coordinates such that
the corresponding vectors $\bfb_{12}$, $\bfb_{21}$,
$(\bfb_{11}+\bfb_{22})$  defined at (\ref{M2}) are linearly independent.}
\endi

We now define $\F_2$ to be the family of all vector fields $\bfw\in \F_1$
which satisfy the additional conditions {\bf (FR)} and {\bf (LI)}.
The openness of $\F_2\subseteq\C^3_0$ is clear.  Indeed,
for a matrix, the property of being full rank is preserved under small perturbations.
Similarly, for the three vectors $\bfb_{12}$, $\bfb_{21}$,
$(\bfb_{11}+\bfb_{22})$  introduced at (\ref{M2}), the property of being linearly independent is also preserved under smooth perturbations.

{\bf 3.}  Having proved openness, it remains to prove density.
By the proof of Lemma~\ref{lemma2}, we already know that
the set of vector fields that satisfy
{\bf (FR)} is dense. We claim that this remains true after adding the
condition {\bf (LI)}.  Indeed, let $\bfw\in \F_1$ be a vector field satisfying
{\bf (FR)}.    Let $(t_0, x_0)\in \S_1$.  At this stage it suffices to construct
a one-parameter family of vector fields $\theta\mapsto \bfv(\theta)$
such that $\bfv(0)=\bfw$, while $\bfv(\theta)$ satisfies {\bf (LI)}
for every $\theta>0$ small.

{\bf 4.} To achieve this goal, we first derive the formulas for $\bfb_{ij}$. Choose a basis of $\R^3$ of the form $\{\bfe_1,\bfe_2,\bfe_3\}$,
where $\bfe_3$ lies in the kernel of the linear map $ D_x \bfw(x_0)-\lambda_3 \,I$,
while $\hbox{span}\{\bfe_1,\bfe_2\}$
is an invariant subspace.
Let $\{\ell_1,\ell_2,\ell_3\}$ be the dual basis, so that
\begin{equation}\label{orth}\ell_i\cdot \bfe_j~=~\delta_{ij}\,,\qquad\quad i,j\in \{1,2,3\}.\end{equation}
For $x$ close to $x_0$,  the matrix $D\bfw(x)$ will have an isolated
eigenvalue close to $\lambda_3$.   Call $\ell_3(x)$ a left eigenvector
corresponding to this eigenvalue, and let $E(x)=\ell^\perp_3(x)$ be the
2-dimensional orthogonal subspace. The basis of $\R^3$ can be chosen as $\{\bfe_1(x),\bfe_2(x),\bfe_3(x)\}$, which coincides
with $\{\bfe_1,\bfe_2,\bfe_3\}$ at the point $x_0$. We still denote the dual basis by $\{\ell_1(x),\ell_2(x),\ell_3(x)\}$, so that
$$\ell_i(x)\cdot \bfe_j(x)~=~\delta_{ij}\,,\qquad\quad i,j\in \{1,2,3\}.$$
Note that $\{\ell_1(x),\ell_2(x),\ell_3(x)\}$ reduces to
$\{\ell_1,\ell_2,\ell_3\}$ when $x=x_0$.

By construction, we have
\begin{eqnarray}\label{v1}D{\bf w}(x)\,\bfe_1(x)&=&a_{11}\,\bfe_1(x)+a_{21}\,\bfe_2(x),\\
\label{v2}D{\bf w}(x)\,\bfe_2(x)&=&a_{12}\,\bfe_1(x)+a_{22}\,\bfe_2(x),\\
\label{v3}D{\bf w}(x)\,\bfe_3(x)&=&a_{13}\,\bfe_1(x)+a_{23}\,\bfe_2(x)+a_{33}\,\bfe_3(x), \end{eqnarray}
where $a_{ij} \;(i,j=1,2)$ and $a_{i3}\;(i=1,2,3)$ depend on $x$.
Thus, it should hold that  $$M(x)=\left(\begin{array}{cc}a_{11}&a_{12}\\
a_{21}&a_{22}\end{array}\right).$$
By (\ref{v1}), (\ref{v2}) and (\ref{v3}), we have the following relationships
\begin{equation}\label{aij}a_{ij}=\ell_i(x)D{\bf w}(x){\bfe}_j(x),\;\;i,j=1,2; \quad
a_{i3}(x)=\ell_i(x)D{\bf w}(x){\bfe}_3(x),\;i=1,2,3.\end{equation}
When returning to the point $x_0$, since $\rank(I+t_0D{\bf w}(x_0))=1$ and $\ell_3$ is perpendicular to the right eigenvectors of $\lambda_1, \lambda_2$, we know that the right eigenvectors of $\lambda_1, \lambda_2$ should lie in the subspace $E(x_0)$ and are linearly independent. Then, without loss of generality,  $\bfe_1(x_0)=\bfe_1$ and $\bfe_2(x_0)=\bfe_2$ can be selected to be the right eigenvectors of $\lambda_1$ and $\lambda_2$, respectively.

We now study the eigenvalues of the matrix $M(x)$ for $x$ close to $x_0$. By (\ref{M2}), in order to derive the formulas for ${\bf b}_{ij}$, it suffices to study  $a_{ij}(x_0+y) \,(i,j=1,2)$ when $y\approx0$. Indeed, for $y\approx0$ we have
\begin{eqnarray}a_{ij}(x_0+y)&=&\ell_i(x_0+y)D{\bf w}(x_0+y){\bfe}_j(x_0+y)\nonumber\\
&\approx&\Big(\ell_i(x_0)+y^\top D\ell_i(x_0)\Big)\Big(D{\bf w}(x_0)+D^2{\bf w}(x_0)y\Big)\Big(\bfe_j(x_0)+D\bfe_j(x_0)y\Big)\nonumber\\
&\approx&\ell_iD{\bf w}(x_0){\bfe}_j+\ell_iD{\bf w}(x_0)D\bfe_j\,y+\ell_iD^2{\bf w}(x_0)\,y\, \bfe_j+y^\top D\ell_i\, D{\bf w}(x_0)\bfe_j\nonumber\\
\label{aijn}&=&\lambda_j\delta_{ij}+\ell_iD^2{\bf w}(x_0)\,y\, \bfe_j+\ell_iD{\bf w}(x_0)D\bfe_j\,y+y^\top D\ell_i\,D{\bf w}(x_0)\bfe_j.\end{eqnarray}
The last two terms in (\ref{aijn}) contain only the first derivatives of the vector field ${\bf w}(x)$, which will disappear in the perturbation ${\bf v}(\theta)(x)=\bfw+\theta \,\bfp(x)$ with $\bfp$ being quadratic in $x-x_0$. Thus, without any ambiguity we neglect these two terms and simply write
\begin{equation}\label{a}a_{ij}(x_0+y)\approx\lambda_j\delta_{ij}+\ell_iD^2{\bf w}(x_0)\,y\, \bfe_j,\quad\text{for}\quad y\approx0,\end{equation}
which further implies that $$\ell_iD^2{\bf w}(x_0)\,y\, \bfe_j={\bf b}_{ij}\cdot y.$$

Calling
\begin{equation}\label{mk}M_k\doteq\frac{\partial}{\partial x_k}\Big(D{\bf w}(x_0)\Big),\end{equation}
 it follows from (\ref{a}) that
\begin{equation}\label{b1}{\bf b}_{ij}=\left(\begin{array}{c}\ell_i\,M_1\,{\bf e}_j\\
\ell_i\,M_2\,{\bf e}_j\\\ell_i\,M_3\,{\bf e}_j\end{array}\right), \quad i,\;j=1,\,2.\end{equation}

{\bf 5.}
Introduce
\begin{equation}\label{nm}Q\doteq\Big({\bf b}_{12}, {\bf b}_{21}, {\bf b}_{11}+{\bf b}_{22}\Big),\end{equation}
which depends on the vector field $\bfw$ and is evaluated at $x$. Sometimes we write $Q=Q(\bfw)$.

Assume $Q(\bfw)$ does not have full rank at some point $x_0$. Then, we consider the following perturbation
\bel{onep}\bfv(\theta)(x)=\bfw(x)+\theta\,\bfp(x).\eeq
Here $$\bfp(x)=\varphi(x-x_0)\sum_{i,j,k=1}^3a_{ijk}(x_i-x_{0i})(x_j-x_{0j})\bfe_k,$$
where $\varphi$ is defined by (\ref{test}) and the constants $a_{ijk}$ can be chosen arbitrarily.
Then, by (\ref{b1}) and (\ref{nm}) we have the perturbed matrix
\bel{perq}Q(\bfv)=Q(\bfw)+\theta \,Q(\bfp),\eeq
where  $Q(\bfp)$ is given by
$$Q(\bfp)=\Big(A_1\,\bfe_2, A_2\,\bfe_1, A_1\,\bfe_1+A_2\,\bfe_2\Big),$$
in which
$$A_i=\left(\begin{array}{ccc}a_{11i}+a_{11i}&a_{21i}+a_{12i}&a_{31i}+a_{13i}\\[2mm]
a_{12i}+a_{21i}&a_{22i}+a_{22i}&a_{32i}+a_{23i}\\[2mm]
a_{13i}+a_{31i}&a_{23i}+a_{32i}&a_{33i}+a_{33i}\end{array}\right).$$

Denote  $$A_1=(a_1,\; a_2,\; a_3)^\top\quad A_2=(b_1,\; b_2,\; b_3)^\top.$$
Since $a_{ijk}$ are arbitrary, we can select
\bel{ch}a_1=-\ell_2,\;\;a_2=\ell_3,\;\;a_3=-\ell_1,\;\;b_1=\ell_3,\;\;b_2=-\ell_1,\;\;b_3=\ell_3,\eeq
where $\ell_i\,(i=1,2,3)$ are the dual basis of $\bfe_j\,(j=1,2,3)$, by (\ref{orth}) we obtain
$Q(\bfp)=-I$. Thus, it follows from (\ref{perq}) that
$$Q(\bfv)=Q(\bfw)-\theta\,I.$$
It is easy to see that the above matrix has full rank for $\theta\in(0, \lambda_*)$, where $\lambda_*$ is the smallest one among the absolute values of nonzero eigenvalues of $Q(\bfw)$.  In other words,  the vectors $\bfb_{12}, \bfb_{21}$ and $\bfb_{11}+\bfb_{22}$ are linearly independent for $x$ in a neighborhood of  $x_0$ and for $\theta>0$ sufficiently small.  Hence, we prove that the vector field, denoted by $\F_2\subset\F_1\subset\C_0^3$, satisfying both  the conditions {\bf (FR)} and  {\bf (LI)} is dense in $\C_0^3$. The proof is completed. 
\endproof

By the previous lemma, for a generic velocity field
$\bfw\in \C^3_0(\R^3;\R^3)$, new singularities
are NOT formed at the isolated points $(t_0, x_0)$ where  rank$(I+ t_0 D_x\bfw(x_0))=1$.
The next lemma describes the local minima
of the function $\tau$ introduced at (\ref{taud}), where these singularities can actually be formed.

\begin{theorem}\label{l:83}  There exists an open dense set $\F_3\subset
\C^3_0(\R^3;\R^3)$ of vector fields with the following property.
For every $\bfw\in \F_3$, the  function $\tau(\cdot)$, defined in \eqref{taud},
 has a finite number of local minima
with $(\bar t, \bar x)\in
[0,T]\times\R^3$. At every such point, one has
 Rank$(I+ \bar t D_x\bfw(\bar x))=2$. Moreover,
the Hessian matrix $H(\bar x)= {\partial^2\tau\over\partial x_i\partial x_j}$
is strictly positive definite.
\end{theorem}

{\bf Proof.}
{\bf 1.} Based on the results in Lemmas~\ref{lemma2} and \ref{lemma3}, by using the same argument in Lemma~5.3 in \cite{BCH}, we can assume  that there exists an open dense set $\F\subset\C_0^3(\R^3;\R^3)$, such that the set
\bel{rank3}\S_2=\big\{(t,x)\in[0,T[\times\R^3;\;\;\rank(I+tD_x\bfw(x))=2\big\}\eeq
is an embedded manifold of dimension 3.  Moreover, the set
\bel{cps2}\bega{rl}\ov \S_2&=~\Big\{(t,x)\in [0,T]\times \R^3\,;~~\text{rank}(I+tD_x{\bf w}(x))=2\Big\}\\[3mm]
&=~
\Big\{(t,x)\in [0,T]\times\R^3\,;~~\text{rank}(I+tD_x{\bf w}(x))\leq 2\Big\}\\[3mm]
&~~\left\backslash
\Big\{\text{Finitely many points in}\; [0,T]\times\R^3;\;\rank(I+tD_x\bfw(x))=1\Big\}\right.
\enda\end{equation}
is compact.

Next we need to show there exists an open subset $\F_3\subset\F\subset\C_0^3$ such that for $\bfw\in\F_3$,  the system of four scalar equations
\bel{scal}\Psi(t,x)=(0,0,0,0),\eeq
\bel{scal1}\Psi\doteq\Big(\det(I+tD_x\bfw(x)),\quad\nabla_x\det(I+tD_x\bfw(x))\Big),\eeq
has finitely many solutions on $[0,T]\times\R^3$.

{\bf 2.} For any $\bfw\in\F$, consider the point $(t_0,x_0)\in\ov\S_2$. As in  the proof of
 Lemma~5.3 in \cite{BCH}, we can choose a local system of coordinates such that
\bel{rank4} I+t_0D\bfw(x_0)~=~
\left[\bega{cc} A & \bfb\cr \bfc^\top & d\enda \right],\eeq
where $A$ is a $2\times2$ invertible matrix, while $d-\bfc^\top A^{-1}\bfb=0$.

In a neighborhood of the point $(t_0,x_0)$, the condition $\rank(I+tD_x\bfw(x))=2$ can be equivalently written as
\bel{ceq}\phi(t,x)\doteq d-\bfc^\top A^{-1}\bfb=0.\eeq

We shall consider the system of 4 scalar equations in 4 variables
\bel{sle}\phi(t,x)=\phi_{x_1}(t,x)=\phi_{x_2}(t,x)=\phi_{x_3}(t,x)=0.\eeq

{\bf 3.} We construct a family of perturbations of the vector field $\bfw$ depending on 4 parameters
$\theta_0, \theta_1, \theta_2$ and $\theta_3$, namely,
\bel{pb}\bfv(\theta)(x)=\bfw(x)+\bfp(x,\theta).\eeq
Here the vector field $\bfp=(p_1,p_2,p_3)^\top$ is defined as
\bel{pdef}p_1(0,\theta)=p_2(0,\theta)=0, \;\;p_3=\theta_0(x_3-x_{03})\varphi+\sum_{k=1}^3\theta_k(x_k-x_{0k})(x_3-x_{03})\varphi,\eeq
where $\varphi\in\C_0^{\infty}(\R^n)$ is defined in (\ref{test}).
Then, the corresponding matrix $$I+tD_x[\bfv(\theta)](x)=I+tD_x\bfw(x)+tD_x\bfp(x,\theta)$$
becomes
\bel{rank5}I+tD_x[\bfv(\theta)](x)=\left[\bega{cc} A & \bfb\cr \bfc^\top & d\enda \right]+
\left[\bega{cc} \bf0 & \bf0\cr \bfc_1^\top & d_1\enda \right]=\left[\bega{cc} A & \bfb\cr \bfc^\top+\bfc_1^\top & d+d_1\enda \right],\eeq
where $$\bfc_1^\top=t\left(\frac{\partial p_3}{\partial x_1},\;\frac{\partial p_3}{\partial x_2}\right),\quad d_1=t\,\frac{\partial p_3}{\partial x_3}.$$
The corresponding scalar equation (\ref{ceq}) then takes the form
\bel{neq}\widetilde{\phi}(t,x)\doteq(d+d_1)-(\bfc^\top+\bfc_1^\top)A^{-1}\bfb=0.\eeq

We claim that the map
\bel{nmap}(t,x,\theta)\mapsto \Phi(t,x,\theta)\doteq \big(\widetilde{\phi},\;\widetilde{\phi}_{x_1},\;\widetilde{\phi}_{x_2},\;\widetilde{\phi}_{x_3}\big)\eeq
is transversal to the zero-dimensional manifold $\{\bf0\}\in\R^4$ in a neighborhood of $(t_0,x_0,0)$.
Indeed, at the point $(t,x,\theta)=(t_0,x_0,0)$ direct computation shows that
$$D_{\theta}\Phi=\left(\begin{array}{cccc}\widetilde{\phi}_{\theta_0}&\widetilde{\phi}_{\theta_1}&\widetilde{\phi}_{\theta_2}
&\widetilde{\phi}_{\theta_3}\\[2mm]
\widetilde{\phi}_{x_1\theta_0}&\widetilde{\phi}_{x_1\theta_1}&\widetilde{\phi}_{x_1\theta_2}&\widetilde{\phi}_{x_1\theta_3}\\
[2mm]\widetilde{\phi}_{x_2\theta_0}&\widetilde{\phi}_{x_2\theta_1}&\widetilde{\phi}_{x_2\theta_2}&\widetilde{\phi}_{x_2\theta_3}\\
[2mm]\widetilde{\phi}_{x_3\theta_0}&\widetilde{\phi}_{x_3\theta_1}&\widetilde{\phi}_{x_3\theta_2}&\widetilde{\phi}_{x_3\theta_3}\end{array}\right)
=t_0\left(\begin{array}{cccc}1&0&0&0\\
[2mm]*&1&0&0\\[2mm]
*&*&1&0\\[2mm]*&*&*&2\end{array}\right),$$
which has full rank. Thus, we conclude that the map $\Phi$ is transversal to the zero manifold $\{\bf0\}\in\R^4$, for $(t,x)$ in a neighborhood $\N_{(t_0,x_0)}$ of $(t_0,x_0)$ and $\theta$ in a neighborhood of $0$.

{\bf 4.} Consider the  set of all $(t,x)\in [0,T]\times\R^3$
for which the four identities in (\ref{scal}) hold.   By compactness, it can be covered
by finitely many neighborhoods $\N_{(t_i, x_i)}$, $i=1,\ldots,m$.
For each $i$, we can construct a family of perturbations $\bfv^i(\theta)$, as in
(\ref{pb}), such  that the map (\ref{nmap}) is transversal to the zero manifold
$\{\bf 0\}\subset\R^4$, restricted to $\N_{(t_i, x_i)}$.
Gathering together all these
perturbations we obtain a family of vector fields $\bfv(\theta)$ depending on
$4m$ parameters, such that the corresponding map (\ref{nmap}) is transversal
to the zero manifold on the whole domain $[0,T]\times\R^3$.
By the similar discussion in the proof of step {\bf 3} in Theorem~5.4 in \cite{BCH} and the Transversality Theorem, we conclude that for almost all $\theta\in\R^{4m}$, the map is transversal to the zero-dimensional manifold  $\{\bf0\}\in\R^4$.

{\bf 5.} Now we define the subset $\F_3\subset\F$, which satisfies the following property
\begin{itemize}
\item[{(\bf PR)}]{\it At every point $(t,x)\in[0,T]\times\R^3$ where (\ref{scal}) holds, the Jacobian matrix $D_{(t,x)}\Psi(t,x)$ has the full rank 4.}
\end{itemize}
By the preceding discussion, we find $\F_3$ is dense in $\C_0^3(\R^3;\R^3)$. In addition, for every $\bf w\in \F_3$, the system (\ref{scal})
has at most finitely many solutions on $[0,T]\times\R^3$.

Next we show the openness of $\F_3$ by contradiction.  Assume there exists $\bfw\in\F_3$ and a sequence of vector fields $\bfw_{\nu}\notin\F_3$
converging to $\bfw$ in the $\C^3$ norm.  Then there exist points $(t_{\nu}, x_{\nu})\in[0,T]\times\R^3$ such that the functions $\Psi_{\nu}$ in
(\ref{scal}) satisfy
$$\Psi_{\nu}(t_{\nu},x_{\nu})=(0, 0, 0, 0),\quad \rank\big[D_{(t,x)}\Psi_{\nu}(t_{\nu},x_{\nu})\big]\leq3$$
for all $\nu$. By possibly choosing a subsequence, we can obtain the convergence $(t_{\nu}, x_{\nu})\rightarrow (\ov t,\ov x)$.
Then, by noting the convergence $\bfw_{\nu}\rightarrow \bfw$ in $\C^3$, we have
$$\Psi(\ov t, \ov x)=(0, 0, 0, 0),\quad \rank\big[D_{(t,x)}\Psi (\ov t, \ov x\big]\leq3,$$
which contradicts our assumption.

{\bf 6. }Finally, call $\psi(t,x)=\det[I+tD_x\bfw(x)]$. Let $\tau=\tau(x)$ be the function implicitly defined by $\psi(\tau(x),x)=0$ as in
(\ref{taud}). At any point where
\bel{tan}\psi(t,x)=\psi_{x_1}(t,x)=\psi_{x_2}(t,x)=\psi_{x_3}(t,x)=0,\eeq
we should have $\psi_t(t,x)\neq0$. Then, an explicit computation shows
\bel{tau2}\tau_{x_ix_j}=-\frac{\psi_{x_ix_j}}{\psi_t},\;\;i,\,j=1,2,3.\eeq
By construction, every function $\bfw\in\F_3$ satisfies the property ({\bf PR}). Therefore,
at every point $(t,x)\in[0,T]\times\R^3$ where (\ref{tan}) holds, we have
\bel{gma}D_{(t,x)}\Psi=\left[\begin{array}{cccc}\psi_t&\psi_{x_1}&\psi_{x_2}&\psi_{x_3}\\[2mm]
\psi_{x_1t}&\psi_{x_1x_1}&\psi_{x_1x_2}&\psi_{x_1x_3}\\[2mm]
\psi_{x_2t}&\psi_{x_2x_1}&\psi_{x_2x_2}&\psi_{x_2x_3}\\[2mm]
\psi_{x_3t}&\psi_{x_3x_1}&\psi_{x_3x_2}&\psi_{x_3x_3}\end{array}\right]=
\left[\begin{array}{cccc}\psi_t&0&0&0\\[2mm]
\psi_{x_1t}&\psi_{x_1x_1}&\psi_{x_1x_2}&\psi_{x_1x_3}\\[2mm]
\psi_{x_2t}&\psi_{x_2x_1}&\psi_{x_2x_2}&\psi_{x_2x_3}\\[2mm]
\psi_{x_3t}&\psi_{x_3x_1}&\psi_{x_3x_2}&\psi_{x_3x_3}\end{array}\right]\eeq
has rank 4. As a consequence, the Hessian matrix
$$\left[\begin{array}{ccc}\tau_{x_1x_1}&\tau_{x_1x_2}&\tau_{x_1x_3}\\[2mm]
\tau_{x_2x_1}&\tau_{x_2x_2}&\tau_{x_2x_3}\\[2mm]
\tau_{x_3x_1}&\tau_{x_3x_2}&\tau_{x_3x_3}\end{array}\right]=-\frac{1}{\psi_t}
\left[\begin{array}{ccc}\psi_{x_1x_1}&\psi_{x_1x_2}&\psi_{x_1x_3}\\[2mm]
\psi_{x_2x_1}&\psi_{x_2x_2}&\psi_{x_2x_3}\\[2mm]
\psi_{x_3x_1}&\psi_{x_3x_2}&\psi_{x_3x_3}\end{array}\right]$$
has rank 3. Thus, the proof is completed. 
\endproof

\section{Evolution of singular hypersurfaces}
\label{s:2}
\setcounter{equation}{0}
Now we consider again the Cauchy problem (\ref{SP})-(\ref{ic}).
Note that outside the singularities, the density $\rho$ and the velocity $\bfv$  of the pressureless gas are
determined by free flow, as in (\ref{Eqn:velocityTransport})-(\ref{Eqn:densityTransport}).

Following \cite{BCH}, we  consider a singular surface $\Sigma^\top\subset\R^3$.  For a fixed time $t$,
points on this surface will be denoted by
$$y(\zeta,t)~=~y(\zeta_1,\zeta_2,t)~\in ~\Sigma^\top.$$
We think of $\zeta\in \R^2$ as a Lagrangian parameter, labeling a particle
which moves on this surface. As time increases, $t\mapsto y(\zeta,t)$ denotes the location of a moving particle. Let $\bfv(\zeta,t)={d\over dt} y(\zeta,t)$ be the velocity of this particle and
$\eta(\zeta,t)$ be the mass density at the point $y(\zeta,t)$.

For any Borel set $A\subset\R^{2}$, consider the portion of the surface
$\Sigma^{A,t}=\{y(t,\zeta)\,;~\zeta\in A\}\subset\Sigma^\top$.
The total area of $\Sigma^{A,t}$ is then computed by
$$\hbox{Area}(\Sigma^{A,t})~=~\int_A \det\left(\bfn, {\partial y\over\partial \zeta_1}, {\partial y\over\partial \zeta_{2}}\right)\, d\zeta,$$where $\bfn$ is the unit normal to the surface $\Sigma^\top$ at 
the point $y(\zeta,t)$.
The total mass contained in $\Sigma^{A,t}$ is
$$\int_A \eta(\zeta,t)\, d\zeta.$$
Introduce the vector
\bel{etx}
\bfe(\zeta,t)~\doteq~\det \left(\bfn, {\partial y\over\partial \zeta_1}, {\partial y\over\partial \zeta_{2}}\right)\, \bfn(\zeta,t)={dy\over d\zeta_1}\times {dy\over d\zeta_2}.\eeq

Let
$\bfv^+(t,x), \bfv^-(t,x), \rho^+(t,x), \rho^-(t,x),$
be respectively the velocity and the density (w.r.t.~3-dimensional Lebesgue measure)
of the flow, in front of and behind the hypersurface.
Then,
the conservation of mass yields
\bel{o2} {d\over dt} \eta~=~\la \bfv-\bfv^+, \bfe\ra \,\rho^+
+ \la \bfv^--\bfv, \bfe\ra\, \rho^- .\eeq
Notice that the  right-hand side of (\ref{o2}) measures the rate
at which additional mass sticks to the surface $\Sigma^\top$ from both sides.
Here and in the sequel, it is understood that
all the functions $\rho^\pm$ and $\bfv^\pm$ are evaluated at the point
$(t, y(\zeta,t))$.

In addition, the conservation of momentum implies
\bel{o3}{d\over dt} (\eta\,\bfv)~=~\la \bfv-\bfv^+, \bfe\ra \,\rho^+\bfv^+
+ \la \bfv^--\bfv, \bfe\ra \rho^-\bfv^-.\eeq
Together with (\ref{o2}), this yields 

\bel{o4}
{d\over dt} \bfv~=~{1\over \eta}\cdot
\Big[\langle \bfv-\bfv^+, \bfe\rangle \,\rho^+(\bfv^+-\bfv)
+ \langle \bfv^--\bfv, \bfe\rangle \rho^-(\bfv^--\bfv)\Big].
\eeq

Setting ${\bfp}=\frac{\partial y}{\partial\zeta_1}, {\bf q}=\frac{\partial y}{\partial\zeta_2}$ and observing that $\frac{\partial^2 y}{\partial\zeta_1\partial t}=\frac{\partial \bfv}{\partial\zeta_1}, \frac{\partial^2 y}{\partial\zeta_2\partial t}=\frac{\partial \bfv}{\partial\zeta_2}
$, we can derive the following system 
\bel{ssy}
\left\{\bega{rl}y_t&=~\bfv,\\[3mm]
\eta_t&=~\la\bfv-\bfv^+, {\bfp}\times{\bf q}\ra\rho^+-\la\bfv-\bfv^-, {\bfp}\times{\bfq}\ra\rho^-,\\[3mm]
\bfv_t&=~\ds \la \bfv-\bfv^+, {\bfp}\times{\bfq}\ra\rho^+\, {\bfv^+-\bfv\over\eta}-\la\bfv-\bfv^-,{\bfp}\times{\bfq}\ra\rho^-\, {\bfv^--\bfv\over\eta}\,,\\[3mm]
\bfp_t - \bfv_{\zeta_1}&= ~0,\\[3mm]
\bfq_t - \bfv_{\zeta_2}&= ~0.
\enda
\right.\eeq
We found that  the semilinear
system (\ref{ssy}) is not hyperbolic, as in the 2D case. As a consequence,
in general we expect that  the Cauchy problem describing the evolution of a
singular hypersurface  will be ill-posed.
However, existence and uniqueness of local solutions can still be achieved, but only for analytic initial data.

\section{Structure of solutions near the point of singularity formation}
\label{s:4}
\setcounter{equation}{0}

In this section, we shall study the formation of new singularities  in three spatial dimensions. Assuming that the initial data $\bar{\rho}, \bfw$ in (1.2) are analytic, we wish to construct
a solution of (\ref{SP}) beyond the time $t_0$ when the singularity first appears. More precisely, we are going to show that the solution is smooth outside a singular hypersurface $\mathcal{S}$, where a positive amount of mass is concentrated; see Figure 1. Outside this surface $\mathcal{S}$, the velocity and the density
of the pressureless flows are still computed by (1.3) and (1.4) respectively. To completely describe this solution, it thus remains to determine the position of the singular hypersurface $\mathcal{S}$ and the velocity and density of particles along $\mathcal{S}$. 

As in \cite{BCH2}, we seek to construct a local map of the following form
$
(t,\zeta)\longmapsto
\bigl(\eta(t,\zeta),y(t,\zeta)\bigr),
$
where $y(t,\zeta)$ denotes the location of a particle along the singular surface $\mathcal{S}$, and $\eta(t,\zeta)$ denotes its surface density.



Recall that we assume ({\bf A1}) and ({\bf A2}), which can be found in the introduction. Define  
\begin{align}
    \label{tau}\tau(x) := -\frac{1}{\lambda_1(x)}.
\end{align}
By ({\bf A2}), it follows that
$\tau$ has a positive minimum at $x_0$, and the Hessian matrix $D^2\tau(x_0)$ is also strictly positive
definite. By Lemma 3.4, the assumption ({\bf A2}) on the Hessian matrix
is indeed satisfied for generic initial data in $\mathcal{C}^3$.

Let $\xi(t,x):=x+t\mathbf w(x)$ 
be the Lagrangian flow map. By Eq. (\ref{Eqn:densityTransport}), we find that 
$$
    \lim_{t\uparrow\tau(x_0)}\rho(t,\xi(t,x_0)) 
    =\infty,
$$
since at $t=\tau(x)$, the derivative of the flow map $ D_x\xi=I+tD_x\bfw(x) $
becomes singular and ceases to be locally invertible. Subsequently, the inverse of the flow map bifurcates into three branches, 
\begin{equation}
    X_1=X_1(t,y),\qquad x=x(t,y),
   \qquad X_2=X_2(t,y)
    \label{Eqn:FlowMapBifurcation}
\end{equation}
satisfying $$
    \xi(t,X_1)=\xi(t,x)=\xi(t,X_2)=y, 
$$
and 
$$t>\tau(x), \qquad t<\tau(X_1), \qquad t<\tau(X_2).  
$$

Under the assumptions ({\bfA1}) and ({\bfA2}), the time of first singularity formation is given by
\begin{align}
    t_0:= -\frac{1}{\lambda_1(x_0)}
\end{align}
Owing to the fact that the eigenvalues are distinct and $\bfw$ is analytic, after possibly restricting to a sufficiently small neighborhood, we may choose the normalized right eigenvectors $r_i(x) (i=1,2,3)$ corresponding to the eigenvalues $\lambda_i(i=1,2,3)$, respectively. 

Define the function
\begin{equation}
    F(x):=\left\langle
    \nabla\lambda_1(x),
    r_1(x)
    \right\rangle.
\end{equation} 
Since we have $\nabla\lambda_1(x_0)=0,$
it follows that
\begin{equation}
    \nabla F(x_0)
    =D^2\lambda_1(x_0)r_1(x_0), 
\end{equation} 
which does not vanish due to the generic condition ({\bfA2}). Hence, by the implicit function theorem,  the level set  
\begin{equation}
    \Gamma^*
    :=
    \left\{
    x\in\mathcal D:
    F(x)=0
    \right\}\label{g*}
\end{equation} 
is locally indeed an analytic surface. We then seek to parametrize this surface. To this end, we construct two analytic tangent vector fields. Let 
\begin{equation*}
    \tilde r_2 := r_2 - \frac{\nabla F \cdot r_2}
    {\nabla F \cdot r_1}r_1,\quad
  \tilde r_3 := r_3 - \frac{\nabla F \cdot r_3}
    {\nabla F \cdot r_1}r_1,\quad
r_2^* := \frac{\tilde r_2}{\left\| \tilde r_2 \right\|},\quad
 r_3^* := \frac{\tilde r_3}{\left\| \tilde r_3 \right\|}.
\end{equation*} 
By construction, $r_2^*(x)$, $r_3^*(x)$ belong to the tangent space $T_x \Gamma^*$ for any $x\in \Gamma^*$ near $x_0$, and they can be viewed as a basis of  $T_x \Gamma^*$. Now we may choose the following  reparametrization 
\begin{equation*}
    \gamma^*:U\subset\mathbb R^2\to \Gamma^*,
\end{equation*} 
such that
$$
    \gamma^*(0)=x_0,
    \quad
    D_\beta\gamma^*(0)
    =
    [r_2^*(x_0),r_3^*(x_0)],\quad \beta\doteq(\beta_1,\beta_2)^T.
$$

Define 
$$
    R_0(\beta)
    :=
    D_\beta\gamma^*(\beta)
    =
    [\partial_{\beta_1}\gamma^*(\beta),
    \partial_{\beta_2}\gamma^*(\beta)].
$$
Obviously, $R_0(\beta)$ is a basis of tangent space $T_x\Gamma^*$.
Then, there exists an analytic invertible matrix $B(\beta)\in\mathbb R^{2\times2}$ satisfying 
$$
    R_0(\beta)
    =[r_2^*(\gamma^*(\beta)),
    r_3^*(\gamma^*(\beta))]B(\beta), 
$$
where by construction, $B(0)=I $ and $I$ is an identity matrix. 

Next, for each $\beta\in U$, we define
$$
    \sigma
    \longmapsto
    x(\sigma,\beta)
$$
satisfying the ODE
$$
    \partial_\sigma x(\sigma,\beta)
    =
    r_1(x(\sigma,\beta)),
    \qquad
    x(0,\beta)
    =
    \gamma^*(\beta).
$$
Then, let 
$$
    \theta
    :=(\sigma,\;\beta_1,\;\beta_2)^{\top}, 
$$
and 
$$
    J(\sigma,\beta):=
    D_\theta x(\sigma,\beta)
    =
   ( x_\sigma,\;
    x_{\beta_1},\;
    x_{\beta_2}).
$$
Note that $J(\sigma,\beta)$ now forms a basis of three dimensional space $\mathbb{R}^3$. 
It is easy to see that 
$$
    J(0,\beta)=(r_1(\gamma^*(\beta)),\;R_0(\beta)),
$$ 
and 
$$
    J(0,0)=(
    r_1(x_0),\;
    r_2^*(x_0),\;
    r_3^*(x_0)),
$$
which is invertible and can be viewed a basis of $\mathbb{R}^3$ at $x_0$. Therefore, $J$ is invertible over some neighborhood of $\theta=0$. Hence, the map 
$\theta\mapsto x(\theta) 
$ 
is an analytic local diffeomorphism. 

To simplify notations, for any function $h=h(x)$, we denote it by  
$h(\theta):=h(x(\theta)).
$
Moreover, over the surface $\Gamma^*$  we shall denote 
$ h(\beta):=h(0,\beta).$
In particular,
$$    \lambda_i(\beta):=\lambda_i(0,\beta),\;
r_i(\beta):=r_i(0,\beta),\;
    \tau(\beta):=\tau(0,\beta),\;
    \bar\rho(\beta):=\bar\rho(0,\beta),\;
   \mathbf w(\beta):=\mathbf w(0,\beta).
$$
Finally, as in the 2D case, we define the set 
$$
    \mathcal R
    :=
    \left\{
    (t,\theta):
    t>\tau(\theta)
    \right\},
 $$ and its boundary 
$$    \Sigma:=
    \left\{
    (t,\theta):
    t=\tau(\theta)
    \right\}. 
$$

\begin{remark}Note that  the set of vectors 
$\{r_1(\theta),\;\tilde{r}_2(\theta),\;\tilde{r}_3(\theta)\}$ is still a basis of $\mathbb{R}^3$, but $\tilde{r}_2(\theta),\;\tilde{r}_3(\theta)$ are not orthogonal to each other, and $r_1$ is not necessarily perpendicular to the tangent plane at $x_0$.\end{remark}

\begin{figure}[ht]
	\centerline{\hbox{\includegraphics[width=13cm] {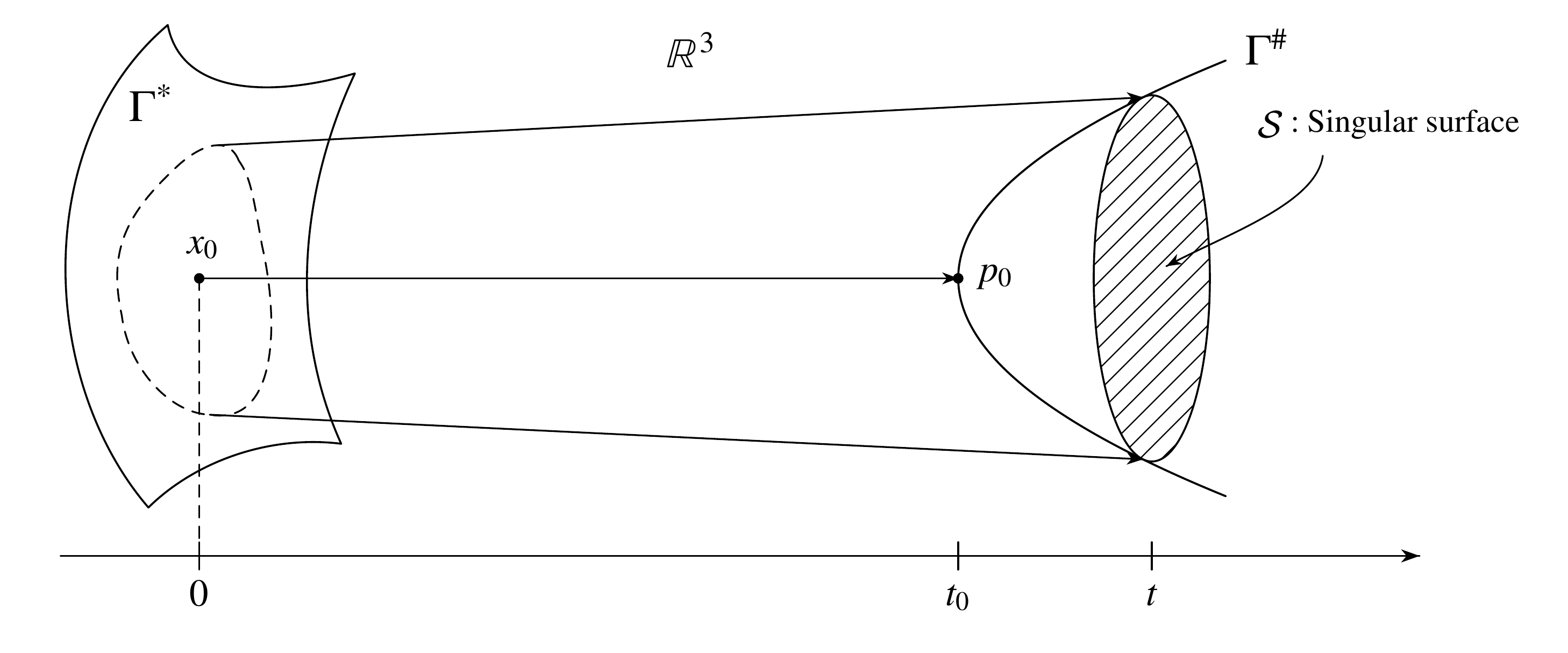}}}
	\caption{\small  Here $\Gamma^*$ is the surface in  $\mathbb{R}^3$ space, constructed by (\ref{g*}).  The image of $\Gamma^*$
		under the map $x \mapsto \bigl(\tau(x), x+ \tau(x) \bfw(x)\bigr)$ is a surface $\Gamma^\sharp$
		in four-dimensional $(t,x)$-space.
		At each time $t>t_0 = \tau (x_0)$,  the singular surface $\mathcal{S}$ (where a positive amount of mass is concentrated) is parameterized by $\zeta=(\zeta_1,\zeta_2)\mapsto y(t,\zeta)$, where $\tau(\zeta)\leq t$.
		Here $y(t,\zeta)$ is the position at time $t$ of the particle which was located at
		$\Gamma^*(\zeta)$
		at the initial time $t=0$.
		The boundary of $\mathcal{S}$ lies on the surface  $\Gamma^\sharp$.}
	\label{fig}
\end{figure}

\subsection{Constructing $X_1,X_2$} 

As in the 1-D and 2-D cases, for every $(t,\sigma,\beta)\in \mathcal R$ sufficiently close to $(t_0,0,0)$, there exist two additional pre-images, $X_1, X_2$ satisfying Eq. (\ref{Eqn:FlowMapBifurcation}),  where
$$
    X_i=x(\theta^i), \quad  \theta^i=(\sigma^i,\beta_1^i,\beta_2^i)^\top, \quad i=1,2,
$$
and $(t,\sigma^i,\beta^i)\notin \mathcal{R}. 
$

Therefore, given $(t,\sigma,\beta) \in \mathcal R$, we first seek to determine $X_1,X_2$, as this will allow us to determine the location of particles along the singular surface $\mathcal{S}$.  To this end, we rewrite Eq. (\ref{Eqn:FlowMapBifurcation}) in the following equivalent vector equation
\begin{equation}
    \Phi(t,\theta,\theta^i):=Q^{\top}(\theta) 
    \left(\xi(t,x(\theta^i))-\xi(t,x(\theta))\right)
    =0,
    \label{Eqn:BigPhiTotal}
\end{equation}
where $Q(\theta):=(r_1(\theta),h_2(\theta),h_3(\theta))$, with $h_2(\theta),h_3(\theta)$ defined by
$$h_2 := \frac{r_1\times r_2^*}{\left\| r_1\times r_2^* \right\|}, \quad
  h_3 := \frac{r_1\times r_3^*}{\left\| r_1\times r_3^* \right\|}.
$$
\begin{remark} The matrix $Q$ is invertible because $r_1$, $h_2$, and $h_3$ are linearly independent. In addition, $r_1$ is orthogonal to $h_2$, $h_3$, while $h_2$ and $h_3$ need not be orthogonal.\end{remark} 

Prior to singularity formation, this mapping admits a unique solution, given by $\theta^i=\theta$. Indeed, we compute
$$
    \left.D_{\theta^i}\Phi\right|_{\theta^i=\theta} 
    =Q^\top(\theta) \left(I+tD_x\mathbf{w}(x(\theta)) \right)J(\theta).  
$$
Since $Q$ and $J$ are invertible, we observe that 
$$
    \det\left(\left.D_{\theta^i}\Phi\right|
    _{\theta^i=\theta}\right)\neq0
$$
for all $t<\tau(\theta)$. At $t=\tau(\theta)$, this determinant vanishes, and the mapping loses invertibility. We note the vector equations (\ref{Eqn:BigPhiTotal}) are equivalent to the following set of two scalar equations. \begin{align}
    &\Psi_1(t,\theta,\theta^i):=r_1(\theta)^\top \left( x(\theta^i) 
    +t\mathbf{w}(x(\theta^i)) 
    -x(\theta)-t\mathbf{w}
    (x(\theta)) \right) = 0 ,
    \label{Eqn:Psi_1}
    \\
    &\Psi_2(t,\theta,\theta^i)
    :=H^\top(\theta)\left(x(\theta^i) 
    +t\mathbf{w}(x(\theta^i)) 
    -x(\theta)-t\mathbf{w}
    (x(\theta)) \right) = 0,
    \label{Eqn:Psi_2}
\end{align}
where $H:=[h_2,h_3]$. We first consider Eq. (\ref{Eqn:Psi_2}) and compute its Jacobian matrix w.r.t. $\beta$,
$$
    P(\beta) :=\left.
    \frac{\partial \Psi_2}{\partial\beta^i}
    \right|_{t=\tau(\beta),\sigma=0,\sigma^i=0,
    \beta^i=\beta} 
    =H^\top(\beta)(I+\tau(\beta)D_x\mathbf{w}(\beta)) 
    R_0(\beta).
$$ 

Note that the matrix $M(\beta):=I+\tau(\beta) D_x\mathbf{w}(\beta)$ admits the eigenvectors and eigenvalues given by   
\begin{align*}
    v_M^1(\beta) = r_1(\beta) ,
    \qquad &\lambda_M^1(\beta) = 0 ,
    \\
    v_M^2(\beta) = r_2(\beta) ,
    \qquad &\lambda_M^2(\beta) = 1 - 
    \frac{\lambda_2(\beta)}{\lambda_1(\beta)}, 
    \\
    v_M^3(\beta) = r_3(\beta), 
    \qquad &\lambda_M^3(\beta) = 1 - 
    \frac{\lambda_3(\beta)}{\lambda_1(\beta)}.
\end{align*} 
Thus, we have $\text{ker}(M(\beta)) =\operatorname{span}\{ r_1(\beta) \}$. Now,  by contradiction we show that the determinant of $P$ does not vanish. Indeed, if this is not true, there exists some nonzero vector $v \in \mathbb{R}^2$, such that 
$$
    H^\top(\beta) M(\beta) R_0(\beta) v = 0.
$$
Introduce $u(\beta):=R_0(\beta)v$ and we have 
$$
H^\top(\beta) M(\beta) u(\beta) = 0.
$$
Note that $u \neq 0$ since the columns of $R_0(\beta)$ are linearly independent.
By construction, $\ker(H^\top(\beta))= \operatorname{span}\{r_1(\beta)\}$, and 
it follows that 
\begin{equation}
    M(\beta) u(\beta) \in \operatorname{span}\{r_1(\beta)\}. 
    \label{Eqn:MuBetaR1} 
\end{equation} 

On the other hand, we may write $u(\beta)$ in terms of the eigenbasis
$$
    u(\beta) = c_1(\beta) r_1(\beta) + c_2(\beta) r_2(\beta) 
    + c_3(\beta) r_3(\beta). 
$$
Since 
$$
    u(\beta)\in\operatorname{span}\{r_2^*(\beta),
    r_3^*(\beta)\}, 
$$
and 
$r_1(\beta),r_2^*(\beta),r_3^*(\beta) $ 
are linearly independent, we find that $c_2(\beta),c_3(\beta)$ do not vanish simultaneously. A simple computation  gives  
$$
    M(\beta)u(\beta) = c_2(\beta)\left(1 - \frac{\lambda_2(\beta)}
    {\lambda_1(\beta)} \right)r_2(\beta) 
    + c_3(\beta)\left(1 - \frac{\lambda_3(\beta)}{\lambda_1(\beta)} \right)r_3(\beta),
$$
which  implies that 
$
    M(\beta)u(\beta) \notin \operatorname{span}
    \{r_1(\beta) \}.
$This is a
contradiction to (\ref{Eqn:MuBetaR1}). 
Therefore, by the implicit function theorem, we can solve out $\beta^i$ by Eq. (\ref{Eqn:Psi_2}), say
$\beta^i=\phi(t,\theta,\sigma^i), 
$ for some analytic function 
satisfying    
$\phi(t,\theta,\sigma) = \beta.$

Next we shall consider the leading-order expansion of $\phi(t,\theta,\sigma)$. Let
$$
    \tilde\Phi(t,\theta,\theta^i):= 
    \xi(t,x(\theta^i))-\xi(t,x(\theta))
$$
and we compute its derivative with respect to $\sigma^i$ for $t\approx\tau(\beta)$,
\begin{align}
    \left.\frac{\partial\tilde\Phi}
    {\partial\sigma^i}\right|_{\beta^i=\beta} 
    &=(I+tD_x\mathbf{w}(x(\theta^i)))r_1(\theta^i) 
    \nonumber\\
    &=(1+t\lambda_1(\theta^i))r_1(\theta^i) 
    \nonumber\\
    &=\left((t-\tau(\beta))\lambda_1(\theta^i) 
    +\tau(\beta)(\lambda_1(\theta^i)
    -\lambda_1(\beta))\right)r_1(\theta^i). 
    \label{Eqn:PartialPhiTildePartialSigma}
\end{align} 

Since the map 
$\sigma\mapsto\tau(\sigma,\beta)=-1/\lambda_1(\sigma,\beta)$
has a positive local minimum at $\sigma=0$,
a Taylor expansion yields
\begin{equation}
    \lambda_1(\theta) 
    = -\omega_1(\beta) + 
    \sum_{k=3}^\infty \omega_k(\beta) \frac{\sigma^{k-1}}{(k-1)!}, \label{Eqn:Lambda_1_Taylor_Exp}
\end{equation} 
where $\omega_1(\beta):=-\lambda_1(0,\beta)>0,\; \omega_3(\beta)>0,\; \omega_k(\beta):=\partial^{k-1}_\sigma 
\lambda_1(0,\beta)$. Then, this allows us to compute the expansion of $\tau(x)$ in the following form 
\begin{equation}
    \tau(\theta)=-\frac{1}{\lambda_1(\theta)} 
    =\frac{1}{\omega_1(\beta)} 
    \left(1+\frac{\omega_3(\beta)}
    {2\omega_1(\beta)}\sigma^2+\O(\sigma^3)
    \right). \label{Eqn:Tau_Taylor_Exp}
\end{equation} 

We now expand the following vector product quantities. Note that $r_1(\theta)$ is a unit vector and is therefore orthogonal to its first derivative.  Hence, we have 
\begin{align}
    \langle r_1(\beta),r_1(\theta)
    -r_1(\beta)\rangle&=q_1(\beta)\sigma^2 
    +q_2(\beta)\sigma^3+\O(\sigma^4) ,
    \label{Eqn:r1T_r1_Taylor_Exp}\\
     H^\top(\beta)\left(r_1(\theta)-r_1(\beta)
    \right)&=q_3(\beta)\sigma+q_4(\beta)\sigma^2 
    +\O(\sigma^3), \label{Eqn:HT_r1_Taylor_Exp}
\end{align}
where 
$$
    q_1(\beta) := \frac{1}{2}r_1(\beta)^\top  
    \frac{\partial^2 r_1}{\partial\sigma^2}(\beta) \in\mathbb{R},
    \quad
    q_3(\beta) := H^\top(\beta)\frac{\partial r_1}
    {\partial\sigma}(\beta)
    \in\mathbb{R}^{2}, 
$$
and $q_2,q_4$ are analytic coefficients whose forms need not be computed explicitly. Combining Eqs.
(\ref{Eqn:Lambda_1_Taylor_Exp}), 
(\ref{Eqn:Tau_Taylor_Exp}), 
(\ref{Eqn:HT_r1_Taylor_Exp}) and Eq.
(\ref{Eqn:PartialPhiTildePartialSigma}) 
gives the following expansion,
\begin{align*}
    &\Psi_2(t,\sigma,\beta,Z,\beta)=H^\top(\beta)
    \left(x(Z,\beta)+t\bfw(x(Z,\beta))-x(\sigma,\beta)
    -t\mathbf{w}(x(\sigma,\beta))\right) 
    \nonumber\\
    &=\int_\sigma^Z \left((t-\tau(\beta))
    \lambda_1(\sigma^i,\beta)+\tau(\beta)
    (\lambda_1(\sigma^i,\beta)-\lambda_1(\beta))
    \right)H^\top(\beta)r_1(\sigma^i,\beta)d\sigma^i 
    \nonumber\\
    &=(t-\tau(\beta))\lambda_1(\beta) 
    \left(\frac{1}{2}q_3(\beta)(Z^2-\sigma^2) 
    +\frac{1}{3}q_4(\beta)(Z^3-\sigma^3)\right) 
    +\frac{1}{8}\omega_3(\beta)\tau(\beta)q_3(\beta) 
    (Z^4-\sigma^4) \nonumber\\
    &+\O\left((t-\tau(\beta))\left|Z^4-\sigma^4\right| 
    +\left|Z^5-\sigma^5 \right|\right).
\end{align*}

Introduce the following notations 
$$
q_5(\beta):= -\frac{1}{2}\lambda_1(\beta)
    P^{-1}(\beta)q_3(\beta),
    \quad
q_6(\beta):= -\frac{1}{3}\lambda_1(\beta)
    P^{-1}(\beta)q_4(\beta),
    \quad q_7(\beta):= -\frac{1}{8}\omega_3(\beta) 
    \tau(\beta) P^{-1}(\beta) q_3(\beta),
$$
by the implicit function theorem, we obtain 
$$
    \phi(t,\theta,Z)-\beta=-P^{-1}\Psi_2(t, 
    \theta,(Z,\beta))
    +\O(
    ((t-\tau(\beta))+|\sigma|+|Z|)
    |\Psi_2(t, 
    \theta,(Z,\beta))|
    +|\Psi_2(t,\theta,(Z,\beta))|^2). 
$$
Therefore, 
\begin{align}
    \phi(t,\theta,Z) - \beta 
    = &(t-\tau(\beta))\left(q_5(\beta)(Z^2-\sigma^2) 
    + q_6(\beta) (Z^3-\sigma^3) \right) 
    + q_7(\beta)(Z^4-\sigma^4) 
    \nonumber\\
    &+\O\left((t-\tau(\beta))\left|Z^4-\sigma^4\right| 
    +\left|Z^5-\sigma^5 \right|\right) := \phi_0(t,\sigma,\beta,Z),  
    \label{Eqn:phiExp}
\end{align} 
which gives a formula for $\beta^i$ in the following form 
$$
    \beta^i = \beta + \phi_0(t,\sigma,\beta,Z) .
$$

Now we move to consider Eq. (\ref{Eqn:Psi_1}). Calling 
$$
    \tilde\Psi := \Psi_1(t,\theta,Z, 
    \phi(t,\theta,Z))
$$
and we split this into two terms 
$\tilde\Psi= A+B$, where 
\begin{align}
    &A:= \left\langle 
    r_1(\beta),x(Z,\beta)+t\mathbf{w}(x(Z,\beta)) 
    -x(\sigma,\beta)-t\mathbf{w}(x(\sigma,\beta))
    \right\rangle ,\label{Eqn:A_section4}
    \\
    &B:= \left\langle 
    r_1(\beta),x(Z,\phi(t,\theta,Z)) + 
    t\mathbf{w}(x(Z,\phi(t,\theta,Z))) 
    -x(Z,\beta)-t\mathbf{w}(x(Z,\beta))
    \right\rangle,\label{Eqn:B_section4}
\end{align} 
in which $A$ is the difference w.r.t.  $\sigma$, while  $B$ denotes the remaining difference w.r.t. $\beta$.

Note that  $A$ can be rewritten as  
$$A=\left\langle r_1(\beta), \int_\sigma^Z (1+t 
\lambda_1(s,\beta))r_1(s,\beta)ds \right\rangle.$$

By using (\ref{Eqn:r1T_r1_Taylor_Exp}), we can deduce that 
\begin{align}
    A = &-\omega_1(\beta)(t-\tau(\beta))(Z-\sigma) 
    + \frac{t\omega_3(\beta)}{6}(Z^3-\sigma^3) 
    + \frac{t\omega_4(\beta)}{24}(Z^4-\sigma^4) 
    \nonumber\\
    &-\omega_1(\beta)(t-\tau(\beta))\left( 
    \frac{q_1(\beta)}{3}(Z^3-\sigma^3)
    +\frac{q_2(\beta)}{4}(Z^4-\sigma^4) 
    \right) 
    \nonumber\\
    &+\O((t-\tau(\beta))\left|Z^4-\sigma^4 
    \right|+\left|Z^5-\sigma^5\right|).
    \label{Eqn:A_Expanded}
\end{align} 
At the same time, we can obtain the expansion for $B$ in the form 
\begin{align}
    B = r_1(\beta)^\top\left(I+\tau(\beta)
    D_x\mathbf{w}(\beta)\right) 
    R_0(\beta)\phi_0(t,\sigma,\beta,Z) 
    + \O\biggl(
    |\phi_0|\bigl(|\phi_0|+ 
    |t-\tau(\beta)|  
    + |Z|\bigr) 
    \biggr). 
    \label{Eqn:B_Expanded}
\end{align} 

We then seek the two nontrivial solutions to $\tilde\Psi = 0$. To remove the trivial solution $Z=\sigma,$  we divide it by $Z-\sigma$ and substitute Eqs. (\ref{Eqn:A_Expanded})--(\ref{Eqn:B_Expanded}) into the resulting equation. Then, we have
\begin{align*}
    0 &= \frac{\tilde\Psi}{Z-\sigma} 
    \nonumber\\
    &= -\omega_1(\beta)(t-\tau(\beta)) 
    + \frac{1}{6}t\omega_3(\beta)(Z^2+Z\sigma+\sigma^2) 
    + \frac{1}{24}t\omega_4(\beta)(Z+\sigma)(Z^2+\sigma^2) 
    \nonumber 
    \\
    &-\omega_1(\beta)(t-\tau(\beta))\left( 
    \frac{1}{3}q_1(\beta)(Z^2+Z\sigma+\sigma^2) 
    +\frac{1}{4}q_2(\beta)(Z+\sigma)(Z^2+\sigma^2)
    \right) \nonumber 
    \\
    &+q_8(\beta)\left( 
    (t-\tau(\beta))\left( 
    q_5(\beta)(Z+\sigma) + q_6(\beta) 
    (Z^2+Z\sigma+\sigma^2)
    \right) 
    +q_7(\beta)(Z+\sigma)(Z^2+\sigma^2)
    \right) \nonumber 
    \\
    &+\O\left((t-\tau(\beta))(|Z|^3+|\sigma|^3) 
    +(|Z^4|+|\sigma|^4)\right), 
\end{align*} 
where 
$$
    q_8(\beta):= r_1(\beta)^\top\left(I+\tau(\beta)D_x\mathbf{w}
    (\beta)\right)R_0(\beta) \in \mathbb{R}^{1\times2}.
$$

As we seek solutions after the singularity formation, we define 
\begin{equation}
   \tilde\sigma := \frac{\sigma}{\delta_\beta}
    \label{Eqn:sigmaHat} 
\end{equation} 
and $$\delta_\beta:=t-\tau(\beta),\quad
U:=\frac{Z}{\sqrt{\delta_\beta}}. $$
Then, we divide the previous equation by $\frac{t\omega_3(\beta)}{6}$ and obtain   
\begin{align}
    0&= (Z^2+Z\sigma+\sigma^2)- 
    \frac{6 \omega_1(\beta) \delta_\beta}{t \omega_3(\beta)}
    +\frac{\omega_4(\beta)}{4\omega_3(\beta)}(Z+\sigma)(Z^2+\sigma^2)\nonumber\\ 
    &-\frac{6\omega_1(\beta)\delta_\beta}
    {t\omega_3(\beta)}
    \biggl[\frac{q_1(\beta)}{3}(Z^2+Z\sigma+\sigma^2)+
    \frac{q_2(\beta)}{4}(Z+\sigma)(Z^2+\sigma^2)\biggr]\nonumber\\ 
    &+\frac{6q_8(\beta)}{t\omega_3(\beta)} \left\{ 
    \delta_\beta\left[ q_5(\beta)(Z+\sigma)+q_6(\beta)
    (Z^2+\sigma Z+\sigma^2)\right]
    +q_7(\beta)(Z+\sigma)\bigl(Z^2+\sigma^2\bigr)\right\}\nonumber\\
    &+\O\left(\delta_\beta(|Z|^3+|\sigma|^3)
    +(|Z^4|+|\sigma|^4)\right). \label{Eqn:Phi_Tilde_Div}
\end{align} 

Using the expansion 
$$
    \frac{1}{t}=\frac{1}{\tau(\beta)+\delta_\beta}
    =\frac{1}{\tau(\beta)}\sum_{k=0}^\infty 
    \left(-\frac{\delta_\beta}{\tau(\beta)}\right)^k 
    =\omega_1(\beta)\sum_{k=0}^\infty 
    (-\omega_1(\beta)\delta_\beta)^k,
$$
we can rewrite (\ref{Eqn:Phi_Tilde_Div}) in a more compact form 
$$
    0=U^2-c_1(\beta)+\sqrt{\delta_\beta} 
    \left(c_2(\tilde\sigma,\beta)U 
    +c_3(\beta)U^3\right) 
    +\delta_\beta\mathfrak{R}\left(\sqrt{\delta_\beta}, 
    \tilde\sigma, \beta, U\right), 
$$
where $\mathfrak{R}$ collects higher-order terms and is analytic in its arguments, and 
\begin{align}
    &c_{1}(\beta) := {6 \omega_1^2(\beta) \over \omega_3(\beta)} , 
    \label{Eqn:c_1}
    \\
    &c_2(\tilde\sigma,\beta) := \tilde\sigma + \frac{6  q_8(\beta)q_5(\beta)\omega_1(\beta)}{\omega_3(\beta)},
    \label{Eqn:c_2}
    \\
    &c_3(\beta):= \frac{\omega_4(\beta)}{4\omega_3(\beta)}+\frac{6q_8(\beta)q_7(\beta)\omega_1(\beta)}{\omega_3(\beta)}.
    \label{Eqn:c_3}
\end{align} 

We then construct the two nontrivial solutions for $Z$. 
\begin{theorem}
\label{Thm:ZpmForm}
There exist a neighborhood
$
\mathcal N
\subset
\mathbb R\times\mathbb R^2
$
containing $(0,0)$ and a constant $\delta_0>0$ such that, for every
$
(\tilde\sigma,\beta)\in\mathcal N,
0\leq\delta_\beta<\delta_0,
$
Eq. (\ref{Eqn:Phi_Tilde_Div}) admits two local solutions, 
$
    Z^\pm=Z^\pm(t,\tilde\sigma,\beta) 
$
in the form
\begin{equation}
    Z^\pm(t,\tilde\sigma,\beta)
    =
    \delta_\beta\alpha_0(\delta_\beta,\tilde\sigma,\beta)
    \pm\sqrt{\delta_\beta}
    \alpha_1(\delta_\beta,\tilde\sigma,\beta), 
    \label{Eqn:ZpmFinal}
\end{equation} 
where $\alpha_0$ and $\alpha_1$ are analytic in
$(\delta_\beta,\tilde\sigma,\beta)$.
Moreover, 
$$\alpha_0(\delta_\beta,\tilde\sigma,\beta)
    =-\frac{1}{2}\left[c_2(\tilde\sigma,\beta)+c_1(\beta)c_3(\beta)\right]+\O(\delta_\beta),$$
    $$
    \alpha_1(\delta_\beta,\tilde\sigma,\beta)
    =\sqrt{c_1(\beta)}+\O(\delta_\beta).
$$
\end{theorem}  
\begin{proof}
    Define 
$
        \varepsilon_\beta:=\sqrt{\delta_\beta}. 
$
    We then seek solutions to 
$
        \mathfrak{F}(\varepsilon_\beta,\tilde\sigma, 
        \beta,U)=0, 
$
    where 
$$
        \mathfrak{F}(\varepsilon_\beta,\tilde\sigma, 
        \beta,U):=U^2-c_1(\beta)+\varepsilon_\beta 
        \left(c_2(\tilde\sigma,\beta)U
        +c_3(\beta)U^3\right) 
        +\varepsilon_\beta^2\mathfrak{R} 
        (\varepsilon_\beta,\tilde\sigma,\beta,U). 
$$
    When $\varepsilon_\beta=0$, it follows that 
$
        U^2-c_1(\beta)=0.   
 $ 
By (\ref{Eqn:c_1}), we know  $c_1(\beta)>0$. Therefore, we have  two distinct real roots $U=\pm\sqrt{c_1(\beta)}. $

   Now we  seek the solution  for $\varepsilon_\beta>0$. To this end, we differentiate 
$$
        \frac{\partial}{\partial U} 
        \mathfrak{F}\left(0,\tilde\sigma, 
        \beta,\pm\sqrt{c_1(\beta)}\right) 
        =\pm2\sqrt{c_1(\beta)} \neq 0. 
$$
    By the implicit function theorem, this implies the local existence of two analytic solution branches, 
$$
   U^+(\varepsilon_\beta,\tilde\sigma,\beta, \quad 
       U^-(\varepsilon_\beta,\tilde\sigma,\beta), 
$$
    which satisfy 
$$
        U^\pm(0,\tilde\sigma,\beta)=\pm 
        \sqrt{c_1(\beta)}. 
$$
In order to derive the formulas, we expand $U^+$ in $\varepsilon_\beta$ as follows
$$
        U^+(\varepsilon_\beta,\tilde\sigma, 
        \beta)=\sum_{k=0}^\infty a_k(\tilde\sigma,\beta) 
        \varepsilon^k_\beta. 
$$ 
    Collecting even and odd terms yields 
$$
        U^+(\varepsilon_\beta,\tilde\sigma,\beta) 
        =\varepsilon_\beta\alpha_0(
        \varepsilon_\beta^2, 
        \tilde\sigma,\beta)+\alpha_1( 
        \varepsilon_\beta^2, 
        \tilde\sigma,\beta),
$$where  $\alpha_0,\alpha_1$ are given as in (\ref{Eqn:ZpmFinal}).

Note the symmetry 
$$
        \mathfrak{F}(\varepsilon_\beta,\tilde\sigma, 
        \beta,U)=\mathfrak{F}(-\varepsilon_\beta,\tilde\sigma, 
        \beta,-U), 
$$
    which implies 
$$
        U^-(\varepsilon_\beta,\tilde\sigma, 
        \beta)=-U^+(-\varepsilon_\beta,
        \tilde\sigma,\beta). 
$$
    Therefore, we have
$$
        U^-(\varepsilon_\beta,\tilde\sigma,\beta) 
        =\varepsilon_\beta\alpha_0(
        \varepsilon_\beta^2, 
        \tilde\sigma,\beta)-\alpha_1( 
        \varepsilon_\beta^2, 
        \tilde\sigma,\beta). 
$$
 This gives the desired form of $Z^\pm$, 
$$
        Z^\pm(t,\tilde\sigma,\beta) 
        =\delta_\beta\alpha_0(\delta_\beta,
        \tilde\sigma,\beta) 
        \pm \sqrt{\delta_\beta}\alpha_1(\delta_\beta, 
        \tilde\sigma,\beta),
$$as stated in (\ref{Eqn:ZpmFinal}).

It remains to show the asymptotic expansion of $\alpha_0, \alpha_1$. We differentiate the equation $\mathfrak{F}(\varepsilon_\beta,\tilde\sigma, 
\beta,U^+(\varepsilon_\beta,\tilde\sigma, 
\beta))=0  $ and obtain 
$$\frac{\partial}{\partial\varepsilon_\beta} 
        \mathfrak{F}(\varepsilon_\beta,\tilde\sigma, 
        \beta,U^+(\varepsilon_\beta,\tilde\sigma, 
        \beta))+\frac{\partial}{\partial U} 
        \mathfrak{F}(\varepsilon_\beta,\tilde\sigma, 
        \beta,U^+(\varepsilon_\beta,\tilde\sigma, 
        \beta))\frac{\partial U^+
        (\varepsilon_\beta,\tilde\sigma,\beta)}
        {\partial\varepsilon_\beta}=0. 
$$
    At $\varepsilon_\beta=0$, it follows that
$$
        c_2(\tilde\sigma,\beta)\sqrt{c_1(\beta)} 
        +c_3(\beta)c_1(\beta)^{3/2} 
        +2\sqrt{c_1(\beta)}\frac{\partial U^+}
        {\partial\varepsilon_\beta}
        (0,\tilde\sigma,\beta)=0,
$$
and
$$
        \frac{\partial U^+}{\partial\varepsilon_\beta}
        (0,\tilde\sigma,\beta)= 
        -\frac{1}{2}\left(c_2(\tilde\sigma,\beta) 
        +c_3(\beta)c_1(\beta)\right). 
$$
Thus, this gives us the desired asymptotic expansions, 
$$
\alpha_0(\delta_\beta,\tilde\sigma,\beta)=
        -\frac{1}{2}\left(c_2(\tilde\sigma,\beta) 
        +c_3(\beta)c_1(\beta)\right) 
        +\O(\delta_\beta), 
$$
$$
      \alpha_1(\delta_\beta,\tilde\sigma,\beta)=
        \sqrt{c_1(\beta)}+\O(\delta_\beta). 
$$The proof is completed. 
\end{proof} 

Furthermore, by using the above results we can  determine the velocity $\mathbf{v}^\pm$ and the density $\rho^\pm$, in a
neighborhood of a point where the singularity is formed. 

\begin{theorem}\label{Theorem:vpm}
Fix $\beta_0\in\mathbb R^2$ and $\tilde\sigma_0\in\mathbb R$. There exists a neighborhood $\mathcal N\subset\mathbb R\times\mathbb R\times\mathbb R^2$ of $(\tau(\beta_0),\tilde\sigma_0,\beta_0)$ such that, for every $(t,\tilde\sigma,\beta)\in\mathcal N$ satisfying
$
    t>\tau(\beta), 
$
the velocity and density associated with $X^\pm$ are given by 
\begin{align}
\mathbf v^\pm(t,\sigma,\beta)-\mathbf w(\beta)
&=(t-\tau(\beta))A_1(t,\tilde\sigma,\beta)\pm\sqrt{t-\tau(\beta)}A_2(t,\tilde\sigma,\beta),
\label{Eqn:vpmForm}\\
\rho^\pm(t,\sigma,\beta)
&=\frac{a_1(t,\tilde\sigma,\beta)\pm\sqrt{t-\tau(\beta)}a_2(t,\tilde\sigma,\beta)}{t-\tau(\beta)}.
\label{Eqn:rhopmForm}
\end{align}
Furthermore, $a_1,a_2,A_1,A_2$ are all analytic functions satisfying 
\begin{align}
A_2(\tau(\beta),\tilde\sigma,\beta)
&=\sqrt{c_1(\beta)}\lambda_1(\beta)r_1(\beta)
=-\omega_1^2(\beta)\sqrt{\frac{6}{\omega_3(\beta)}}r_1(\beta),
\label{Eqn:A2Bdry}\\
a_1(\tau(\beta),\tilde\sigma,\beta)
&=\frac{\omega_1(\beta)\bar\rho(\beta)}
{2(\lambda_2(\beta)-\lambda_1(\beta))(\lambda_3(\beta)-\lambda_1(\beta))}>0. 
\label{Eqn:a1Bdry}
\end{align} 
\end{theorem}

\begin{proof}    
    We recall 
$$
        \mathbf{v}^\pm(t,\sigma,\beta) 
        = \mathbf{w}(X^\pm(t,\sigma,\beta)).
$$
Note that $$
        X^\pm=x(Z^\pm,\phi(t,\sigma,\beta,Z^\pm)) 
$$
and $Z^\pm$ is given (\ref{Eqn:ZpmFinal}), where $\alpha_0,\alpha_1$ are analytic,  these may be composed into the form (\ref{Eqn:vpmForm}) with analytic functions $A_1,A_2$. We then analyze the function $A_2$. Indeed, at $t=\tau(\beta)$, 
$$
        A_2(\tau(\beta),\tilde\sigma,\beta) 
        =\alpha_1(0,\tilde\sigma, 
        \beta)D_x\mathbf{w}(\beta)r_1(\beta) 
        =\sqrt{c_1(\beta)}\lambda_1(\beta)r_1(\beta), 
$$
which implies the desired estimate on $A_2$. 

Now we move to determine the desired estimates on the density. Recall  
$$
        \rho^\pm (t,\sigma,\beta)=
        \frac{\bar \rho(X^\pm)}
        {\det\left(I + tD_x\bfw(X^\pm)\right)}
        ={\bar \rho(X^\pm)\over
        \left(1+t\lambda_2(X^\pm)\right)
        \left(1+t\lambda_3(X^\pm)\right)}
        \cdot {1\over 1+t\lambda_1\bigl(X^\pm)}
 $$
and note that 
    \begin{align}
        1+t\lambda_1(X^\pm) &= -\lambda_1(\beta)\tau(\beta) 
        +t\lambda_1(X^\pm) ,
        \nonumber\\
        &=\delta_\beta\lambda_1(\beta) 
        +t\left(\lambda_1(X^\pm)-\lambda_1(\beta) \right).
        \label{Eqn:1PlusTLambda_1} 
    \end{align} 
By the similar arguments in the above discussions, there exist analytic functions $\theta_i (i=1,\cdots,4)$ satisfying  
    \begin{align}
        &{\bar \rho(X^\pm)\over
         \left(1+t\lambda_2(X^\pm)\right)
        \left(1+t\lambda_3(X^\pm)\right) }=
        \theta_1(t,\tilde\sigma,\beta)\pm 
        \sqrt{\delta_\beta}
        \theta_2(t,\tilde\sigma,\beta) , 
        \label{Eqn:rhoBarOverLambda23}
        \\
        &1+t\lambda_1\bigl(X^\pm)=
        \delta_\beta\left(\theta_3(t,\tilde\sigma,\beta)
        \pm \sqrt{\delta_\beta} 
        \theta_4(t,\tilde\sigma,\beta)\right).
        \label{Eqn:lam_1}
    \end{align} 
    Then, by using Eqs. \eqref{Eqn:ZpmFinal} and \eqref{Eqn:phiExp}, we have 
$$
        \lambda_1(X^\pm)-\lambda_1(\beta) 
        =\frac{\omega_3(\beta)}{2}(Z^\pm)^2 
        +\O\left(\delta_\beta^{3/2}\right).
$$
With the aid of expansions for $Z^\pm$, we can obtain the following
    \begin{align}
        \delta_\beta\lambda_1(\beta)+ t \left( \lambda_1(X^\pm) -\lambda_1(\beta)\right)
        &=-\omega_1(\beta)\delta_\beta  
        +{\tau(\beta)\omega_3(\beta)\over 2}
        c_1(\beta)\delta_\beta+\O(\delta_\beta^{3/2})
        \nonumber\\
        &=2 \omega_1(\beta)\delta_\beta
        +\O\left(\delta_\beta^{3/2}\right) .
        \label{Eqn:tLam1}
    \end{align} 
    This allows us to expand the blowup of the denominator in the following manner
    \begin{equation}
        \frac{1}{1+t\lambda_1(X^\pm)} 
        = \frac{1}{2\omega_1(\beta)(t-\tau(\beta))} 
        \left( 
        1+\sum_{k\geq1}g_k(t-\tau(\beta))^{k/2}
        \right) \label{Eqn:tLam1Recip}, 
    \end{equation} 
    which, after collecting even and odd terms appropriately, gives us the desired form of $\rho^\pm$ posterior to singularity formation. To complete the proof, we need to show the form of $a_1$ and it is strictly positive as $t\rightarrow\tau(\beta)+$. Indeed,  
$$
        a_1(\tau(\beta),\tilde\sigma,\beta) 
        =\frac{\omega_1(\beta)\bar\rho(\beta)}
        {2\left(\lambda_2(\beta)-\lambda_1(\beta)\right)
        	\left(\lambda_3(\beta)-\lambda_1(\beta)\right)}>0,
$$ where by assumptions we have $\lambda_3>\lambda_1,\lambda_2>\lambda_1$ and $\bar{\rho}(\beta)>0$.
\end{proof}

\section{Local construction of singular hypersurfaces}
\label{s:5}
\setcounter{equation}{0} 

By utilizing the results of Section 4, we aim to  seek a local analytic solution $  (t,\zeta)\longmapsto (y,\eta,\bfv) $ 
to the system of equations 
\begin{align} 
    \left\{ 
    \begin{aligned}
        &y_t = \bfv ,
        \\
        &\eta_t = \langle \bfv-\bfv^+,p\times 
        q\rangle\rho^+-\langle \bfv-
        \bfv^-,p\times q\rangle\rho^- ,
        \\
        &\bfv_t=\langle \bfv-
        \bfv^+,p\times q\rangle\rho^+
        \frac{\bfv^+-\bfv}{\eta}
        -\langle \bfv-\bfv^-,p\times q
        \rangle\rho^-\frac{\bfv^-
        -\bfv}{\eta} ,
        \\
        &p_t-\bfv_{\zeta_1} = 0 ,
        \\
        &q_t-\bfv_{\zeta_2} = 0 , 
    \end{aligned}
    \right. \label{Eqn:SystemEvolution}
\end{align} 
where $  p=y_{\zeta_1},
q=y_{\zeta_2}$. This system is coupled with boundary data
\begin{align}
    \left\{ 
    \begin{aligned}
        & y(\tau(\zeta),\zeta)=\gamma^*(\zeta) 
        +\tau(\zeta)\bfw(\gamma^*(\zeta)), 
        \\
        &\eta(\tau(\zeta),\zeta)=0, 
        \\
        &\bfv(\tau(\zeta),\zeta)=\bfw(\gamma^*(\zeta)). 
    \end{aligned}
    \right. 
    \label{Eqn:BdryConditions} 
\end{align} 

In order to state the main result, we need to introduce 
a suitable space of analytic vector fields. More precisely, we fix a closed ball ${\mathbb{B}}= \{ x\in \R^3\,;~|x|\leq r_0\} $
and a constant
$\varrho>0$.   We call $X_{{\mathbb{B}}, \varrho}$ the space of all
analytic functions $\phi:\R^3\mapsto\R$ such that
\bel{avf}   \|\phi\|_{{\mathbb{B}}, \varrho}~\doteq~\sum_{n=0}^\infty {\varrho^n\over n!}
\max_{|\alpha|=n} \sup_{x\in {\mathbb{B}}} \bigl| D^\alpha \phi(x)\bigr|~<~+\infty.\eeq
Here $\alpha=(\alpha_1,\alpha_2,\alpha_3)$ is a multi-index, and 
$D^\alpha \phi\doteq \partial_{x_1}^{\alpha_1}\partial_{x_2}^{\alpha_2}\partial_{x_3}^{\alpha_3}\phi$.
We remark that this is the set of all functions whose Taylor series at every point $x\in {\mathbb{B}}$ converges on a disc
of radius $\varrho/2$.

In the following theorem, we consider the space of all initial data
$(\bfw,\bar\rho)= (w_1, w_2, w_3, \bar \rho)$ such that $w_1, w_2, w_3, \bar \rho\in X_{{\mathbb{B}}, \varrho}$, for some given ${\mathbb {B}}, \varrho$.   Within this space, we shall construct
piecewise analytic solutions to the equations of pressureless gases (\ref{SP})-(\ref{ic})
beyond singularity formation, for a generic set of initial data.

\begin{theorem}\label{thm5.1} For the Cauchy problem (\ref{SP})-(\ref{ic}),  we assume the initial data
	are analytic functions and satisfy {\bf (A1)-(A2)}.
Then, there exists an analytic solution up to some time $t_0$ where the $\C^1$ norm of the velocity $\bfv$ blows up.
	This solution can be extended to a larger interval  $[0, t_0+\delta_0]$. For $t_0<t\leq t_0+\delta_0$,
	the solution is piecewise analytic, and contains a singular surface $\Gamma^t$ where a positive amount of mass is concentrated.
\end{theorem}

We begin with some notations. We denote $\mathcal{B}$ as our basis and $\mathcal{B}^*$ as its dual basis, which are given by 
$$
    \mathcal{B}(\sigma,\beta):= \{r_1(\sigma,\beta), 
    r_2^*(\sigma,\beta),r_3^*(\sigma,\beta)\} ,
$$$$
    \mathcal{B}^*(\sigma,\beta):= \{l_1^*(\sigma,\beta), 
    l_2^*(\sigma,\beta),l_3^*(\sigma,\beta)\}.
$$
Here we claim that $l_2^*,l_3^*$ are left eigenvectors of $D\bfw$ corresponding to eigenvalues $\lambda_2,\lambda_3$, respectively. Indeed, let $l_1,l_2,l_3$ denote the dual left eigenvectors relative to $r_1,r_2,r_3$. Then, we have the following results 
$$
    l_2(\sigma,\beta)^\top r_1(\sigma,\beta) = 0,\quad
    l_2(\sigma,\beta)^\top r_2^*(\sigma,\beta) 
    =l_2^\top\left(b_1r_1+b_2r_2\right)=b_2,\quad
    l_2(\sigma,\beta)^\top r_3^*(\sigma,\beta) 
    =l_2^\top\left(b_3r_1+b_4r_3\right)=0 
$$for some constants $b_i$, where $b_2,b_4\neq0$. 
Analogous computations can hold for $l_3(\sigma,\beta)$. Therefore, by the uniqueness of the dual basis, we observe that $l_2^*(\sigma,\beta),l_3^*(\sigma,\beta)$ are left eigenvectors of $D_x\mathbf{w}$. Note that $l_1^*(\sigma,\beta)$ is not necessarily a left eigenvector. 

Note that $y$ denotes the path of a particle along the surface. For $\tau(x(\theta))<t$,   
\begin{equation}
    y=x(\theta)+t\bfw(x(\theta)), 
    \label{Eqn:yTransport}
\end{equation} 
where 
$
    \theta=[\sigma,\beta_1,\beta_2]^\top. 
$
Define the Jacobian of $x$ with respect to $\theta$ as$$
    J(\theta)=x_\theta(\theta) 
    =[x_\sigma,x_{\beta_1},x_{\beta_2}]. 
$$
By construction, $x_\sigma=r_1$, and 
$
    J(0,\beta)=[r_1(\beta),R_0(\beta)].
$

Differentiating (\ref{Eqn:yTransport}) gives the following identities. 
\begin{align}
    y_t&=\mathbf{w}\bigl(x(\theta)\bigr)
    +\left(I+tD_x\mathbf{w}\bigl(x(\theta)\bigr)\right)
    J(\theta)\theta_t, \label{Eqn:yt}
    \\
    y_{\zeta_1}&=\left(I+tD_x\mathbf{w}
    \bigl(x(\theta)\bigr)\right)
    J(\theta)\theta_{\zeta_1},\label{Eqn:yzeta1}
    \\
    y_{\zeta_2}&=\left(I+tD_x\mathbf{w}
    \bigl(x(\theta)\bigr)\right)J(\theta)
    \theta_{\zeta_2},\label{Eqn:yzeta2}
    \\
    y_{tt}&=2D_x\mathbf{w}\bigl(x(\theta)\bigr)
    J(\theta)\theta_t+tD_x^2\mathbf{w}
    \bigl(x(\theta)\bigr)\left[J(\theta)\theta_t,
    J(\theta)\theta_t\right]
    \nonumber\\
    &\quad+\left(I+tD_x\mathbf{w}\bigl(x(\theta)\bigr)
    \right)\left(J(\theta)\theta_{tt}+\Lambda_J
    (\theta_t)\right), 
    \label{Eqn:ytt}
\end{align}
where $\Lambda_J:\mathbb{R}^3\rightarrow\mathbb{R}^3$ represents the quadratic map
$$
    \Lambda_J(v)=\sum_{i,j=1}^3 
    x_{\theta_i\theta_j}v_iv_j.
$$

Substituting Eq. (\ref{Eqn:yt}) into the second equation of Eq. \eqref{Eqn:SystemEvolution} gives  
\begin{align}
    \eta_t&=(\rho^+(\bfv-\bfv^+)-\rho^-
    (\bfv-\bfv^-))^\top p\times q 
    \nonumber\\
    &=(\rho^+(\bfw+(I+tD_x\bfw)
    J\theta_t-\bfv^+)-\rho^-(\bfw+
    (I+tD_x\bfw)J\theta_t-\bfv^-))
    ^\top p\times q \nonumber\\
    &=(\rho^+(\bfw-\bfv^+)-\rho^-(\bfw-\bfv^-) 
    +(\rho^+-\rho^-)(I+tD_x\bfw)J\theta_t)^\top 
    p\times q \nonumber\\
    &:=H(t,\theta,\theta_t, \theta_\zeta) .
    \label{Eqn:HDefinition}
\end{align} 
The third equation of (\ref{Eqn:SystemEvolution}) can be rewritten in the form
\begin{align}
    \bfv_t=\frac{1}{\eta}(\rho^+\left((y_t
    -\bfv^+)^\top (p\times q)\right)(\bfv^+-y_t)
    -\rho^-\left((y_t-\bfv^-)^\top (p\times q)
    \right)(\bfv^--y_t)). \label{Eqn:vtInt}
\end{align}

By comparing Eqs. \eqref{Eqn:vtInt} and  \eqref{Eqn:ytt}, we obtain the following equation for $\theta_{tt}$, 
\begin{align}
    (I+tD_x\mathbf{w})J\theta_{tt}=&-(I+tD_x\bfw)
    \Lambda_J(\theta_t)+(-2D_x\bfw J\theta_t-
    tD_x^2\bfw[J\theta_t,J\theta_t])
    \nonumber\\
    &+\frac{1}{\eta}\biggl(\rho^+\left((y_t
    -\bfv^+)^\top p\times q\right)(\bfv^+-y_t)
    \nonumber\\
    &-\rho^-\left((y_t-\bfv^-)^\top p\times q\right)
    (\bfv^--y_t)\biggr).
    \label{Eqn:Q_theta_tt_complicated}
\end{align} 
Introduce
\begin{align}
    &F:=-(I+tD_x\bfw)\Lambda_J(\theta_t) 
    +(-2D_x\bfw J\theta_t
    -tD_x^2\bfw[J\theta_t,J\theta_t]) 
    \label{Eqn:F_Definition}, 
    \\
    &G:=\rho^+\left((y_t
    -\bfv^+)^\top p\times q\right)(\bfv^+-y_t)
    -\rho^-\left((y_t-\bfv^-)^\top p\times q\right)
    (\bfv^--y_t) 
    \label{Eqn:G_Definition}, 
\end{align} 
then Eq.
\eqref{Eqn:Q_theta_tt_complicated} can be written as 
$$
    (I+tD_x\mathbf{w})J\theta_{tt}=F(t,\theta,
    \theta_t) 
    +\frac{1}{\eta}G(t,\theta, 
    \theta_t,\theta_\zeta). 
$$
Let
$$
    L_*^\top:=\begin{bmatrix}
        l_2^* \\
        l_3^*
    \end{bmatrix}, 
    \quad
    \mathcal{P}:=L_*^\top 
    (I+tD\bfw)J_\beta, \quad J_\beta:=[x_{\beta_1},x_{\beta_2}].
$$
Note that, at $\sigma=\beta=0$ and $t=\tau(0)$,  
$$
    \P_0=\begin{bmatrix}
        1+\tau\lambda_2&0\\
        0&1+\tau\lambda_3 
    \end{bmatrix}.
$$
We see that $\P$ is locally invertible. Thus, this gives rise to an equivalent form of our system
\begin{align}
    &\eta_t = H \label{Eqn:eta_t},
    \\
    &\sigma_{tt}=\frac{l_1^\top(F+\eta^{-1}G)}
    {1+t\lambda_1}-l_1^\top J_\beta \beta_{tt}, 
    \label{Eqn:Sigma_tt}
    \\
    &\beta_{tt}=\mathcal{P}^{-1} 
    L_*^\top (F+\eta^{-1}G).  
    \label{Eqn:beta_tt}
\end{align} 

\subsection{Rewriting the System of Equations} 

We seek to rewrite Eqs. (\ref{Eqn:eta_t})-(\ref{Eqn:beta_tt}) in a more convenient form. Indeed, we note from (\ref{Eqn:BdryConditions}) that
$
    1/\eta$ approaches $\infty 
$
along $t=\tau(\zeta)$. In this section, we shall first define new variables $\hat\eta$, $\hat\theta$ and rewrite (\ref{Eqn:eta_t})--(\ref{Eqn:beta_tt}) as a first order system of equations. 

To this end, we define  
$$
    S:=\frac{\partial\sigma}{\partial t} ,
    \quad
    Z:=\frac{\partial\beta}{\partial t},
$$
and the following rescaled quantities 
$$
    \hat\eta:=\frac{\eta}{\sqrt{t-\tau(\zeta)}} ,\quad
    \hat\beta:=\frac{\beta-\zeta}
    {t-\tau(\zeta)} ,
    \quad
    \hat\sigma:=\frac{\sigma}{t-\tau(\zeta)}, 
    \quad
    \hat\theta:=[\hat\sigma,\hat\beta]^\top,$$
which will be shown to be regular in the limit as $t\rightarrow\tau(\zeta)$. In addition, we set 
$\delta_\zeta:=t-\tau(\zeta). $

Now it is ready to rewrite (\ref{Eqn:HDefinition}) and (\ref{Eqn:F_Definition})--(\ref{Eqn:G_Definition}) in terms of the new variables. Note that from Theorem \ref{Theorem:vpm} we find that $\bfv^\pm(t,\sigma,\beta)$  relates to $\bfw(x(0,\beta))$, whereas (\ref{Eqn:HDefinition}) depends on $\bfw(x(\sigma,\beta))$. Therefore, we need to expand $\bfw$ w.r.t. $\sigma$. Indeed, we have
$$
    \left.\partial_\sigma(\bfw(x(\sigma,\beta)))
    \right|_{\sigma=0}
    =\left.D_x\bfw(x(\sigma,\beta))
    r_1(\sigma,\beta)\right|_{\sigma=0} 
    =\lambda_1(\beta)r_1(\beta). 
$$
Therefore, 
$$
    \bfw(x(\sigma,\beta))=\bfw(x(0,\beta)) 
    +\sigma\lambda_1(\beta)r_1(\beta) 
    +\sigma^2\mathcal{I}(\sigma,\beta), 
$$
where $\mathcal{I}$ is analytic in its arguments. Substituting $
    \sigma=\delta_\beta\tilde\sigma $
into the above formula 
yields 
\begin{equation}
    \bfw(x(\sigma,\beta))-\bfw(x(0,\beta)) 
    =\delta_\beta\tilde\sigma\lambda_1(\beta) 
    r_1(\beta)+\delta_\beta^2\tilde\sigma^2 
    \mathcal{I}(\delta_\beta\tilde\sigma, 
    \beta)=:\delta_\beta A_3(t,\tilde\sigma,\beta). 
    \label{Eqn:wExpanded}
\end{equation} 

 By defining
\begin{equation}
    \hat A_1:=\delta_\beta (A_1-A_3),
    \label{Eqn:A1HatDefn} 
\end{equation} 
and inserting Equations (\ref{Eqn:vpmForm})--(\ref{Eqn:rhopmForm}) into (\ref{Eqn:HDefinition}), we can obtain the formula for $H$
\begin{align}
    H&=\frac{2}{\sqrt{\delta_\beta}}\left(
    -A_2a_1-\hat A_1a_2
    +a_2(I+tD_x\mathbf{w}
    (x(\theta)))
    J\theta_t\right)^\top p\times q \nonumber
    \\
    &:=\frac{1}{\sqrt{\delta_\zeta}} 
    H_1(t,\zeta,\hat\theta,
    \theta_t,\theta_\zeta) . 
    \label{Eqn:H_in_detail}
\end{align} 
The details of  above calculation are shown in Appendix \ref{App:H_Computation} and are not listed here for brevity.  

Next, we seek to rewrite the left hand side of Equation (\ref{Eqn:HDefinition}) in terms of $\hat\eta$.  Since 
$$
    \hat\eta_t=\frac{\eta_t}{\sqrt{t-\tau(\zeta)}} 
    -\frac{\hat\eta}{2(t-\tau(\zeta))},
$$
we have 
$$
    \hat\eta_t=\frac{H_2}{t-\tau(\zeta)} ,
    \quad
    H_2:=H_1-\frac{1}{2}\hat\eta .
$$
Then, we turn to rewrite Eqs. (\ref{Eqn:F_Definition})--(\ref{Eqn:G_Definition}). To this end, we summarize the previously derived relations. 
\begin{align}
    &y_t-\bfw=(I+tD_x\bfw)J\theta_t=:B ,
    \label{Eqn:yt_minus_w}
    \\
    &\textbf{v}^\pm-\bfw=
    \hat A_1\pm A_2 \delta_\beta^{1/2},
    \label{Eqn:vpmMinusw}
    \\
    &y_{\zeta_1}=(I+tD_x\bfw)
    J\theta_{\zeta_1} ,
    \\
    &y_{\zeta_2}=(I+tD_x\bfw)J
    \theta_{\zeta_2} ,
    \\
    &\rho^\pm=a_1\delta_\beta^{-1}\pm 
    a_2\delta_\beta^{-1/2}.
\end{align} 
Here we claim that $Z$ vanishes along the boundary $t=\tau(\zeta)$. Indeed, inserting the third equation of (\ref{Eqn:BdryConditions}) into Eq. (\ref{Eqn:yt_minus_w}) gives that at $t=\tau(\zeta)$, 
$
    0= (I+\tau(\zeta)D_x\bfw)J\theta_t .
$
Furthermore, at $\sigma=0$ we have
$
    J\theta_t=r_1S+R_0Z, 
$
and thus, we obtain 
$$
    0= (I+\tau(\zeta)D_x\bfw)(r_1S+R_0Z) .
$$
Since at $t=\tau(\zeta)$,
$
    (I+\tau(\zeta)D_x\bfw)r_1=0, 
$
the $S$ dependence vanishes, giving 
$$
    0= (I+\tau(\zeta)D_x\bfw)R_0Z. 
$$
Multiplying on both sides of the above equation by $L_*^\top$ and noting the local invertibility of 
$
    L_*^\top(I+\tau(\zeta)D_x\bfw)R_0, 
$
we arrive at 
\begin{align}
    Z(\tau(\zeta),\zeta)=0. 
    \label{Eqn:ZBoundaryCond}
\end{align}

Furthermore, as in the 1-D and 2-D cases, in order for the integrability conditions to hold, we require that $\sigma(\tau(\zeta),\zeta)=0$. Therefore, our solutions may be expanded as  
\begin{align}
    &\sigma(t,\zeta)=S_0(\zeta)\delta_\zeta
    +\O(\delta_\zeta^2) ,\label{Eqn:sigmaForm}
    \\
    &\beta(t,\zeta)=\zeta+\O(\delta_\zeta^2) .
    \label{Eqn:betaForm}
\end{align} 
By using Eqs. (\ref{Eqn:sigmaForm})--(\ref{Eqn:betaForm}), we can expand Eqs. (\ref{Eqn:yt_minus_w})--(\ref{Eqn:vpmMinusw}) to derive the leading order terms. Due to the vanishing of $Z$ along the boundary, we find that
$$
    \beta(t,\zeta)-\zeta=\O(\delta_\zeta^2).$$
Therefore, 
$$
    \delta_\beta=\delta_\zeta+\O(\delta_\zeta^2).  
$$

We then compute 
$$
    y_t-\bfw=(I+tD_x\bfw)J\theta_t  
   =(I+tD_x\bfw)(r_1S+J_\beta Z), 
$$
where 
$
    J_\beta(\sigma,\beta)=R_0(\zeta)
    +\O(\delta_\zeta). 
$
Noting (\ref{Eqn:sigmaForm})--(\ref{Eqn:betaForm}),
we have $$
    \sigma(t,\zeta)=S_0(\zeta)\delta_\zeta+\O
    (\delta_\zeta^2), \quad 
    Z(t,\zeta)=Z_1(\zeta)\delta_\zeta+\O
    (\delta_\zeta^2). 
$$
It follows that 
\begin{align}
    B=\delta_\zeta(-\omega_1 S_0 r_1+(I+\tau(\zeta)D_x\bfw)R_0Z_1) 
    +\O(\delta_\zeta^2). 
    \label{Eqn:B_Comp}
\end{align} 
Multiplying on the left side by $l_1^\top$ gives 
$$
    l_1^\top B=-\omega_1S_0\delta_\zeta 
    +\O(\delta_\zeta^2). 
$$

Now we consider the quantity  $\bfv^\pm-\bfw$. We first recall the form of $Z^\pm$ from (\ref{Eqn:ZpmFinal}), 
$$
    Z^\pm(t,\tilde\sigma,\beta)=\pm
    \sqrt{c_1\delta_\beta}-\frac{1}{2}(c_2(\tilde\sigma,\beta)+c_1(\beta)c_3(\beta))\delta_\beta 
    +\O(\delta_\beta^{3/2}). 
$$ 
By (\ref{Eqn:c_2}), we note that
$$
    c_2(\tilde\sigma,\beta)=\tilde\sigma 
    +\frac{6q_8(\beta)q_5(\beta)\omega_1(\beta)}
    {\omega_3(\beta)}=: \tilde\sigma 
    +c_2^0(\beta). 
$$In addition, 
 $\tilde\sigma=S_0$ along $t=\tau(\zeta)$. As a result, we have 
$$
    -\frac{1}{2}(c_2(\tilde\sigma,\beta) 
    +c_1(\beta)c_3(\beta))=-\frac{1}{2}\tilde\sigma 
    +b^*(\beta)=-\frac{1}{2}S_0+b^*(\beta) 
    +\O(\delta_\zeta),
$$
 where $b^*(\beta)$ is independent of $S_0$.
It further follows that 
\begin{align}
    Z^\pm-\sigma&=\pm\sqrt{c_1(\beta)}
    \delta_\beta^{1/2}+\left(-\frac{1}{2}S_0+b^*(\beta) 
    -S_0\right)\delta_\beta+\O(\delta_\beta^{3/2}) 
    \nonumber\\
    &=\pm\sqrt{c_1(\beta)}
    \delta_\beta^{1/2}+\left(-\frac{3}{2}S_0+b^*(\zeta) 
    \right)\delta_\beta+\O(\delta_\beta^{3/2}). 
    \label{Eqn:ZpmSigma}
\end{align}

Recalling the expansion for $\phi_0$,  we have
\begin{align}
    \beta^\pm-\beta = &(t-\tau(\beta))
    \left(q_5(\beta)({Z^\pm}^2-\sigma^2) 
    + q_6(\beta) ({Z^\pm}^3-\sigma^3) \right) 
    + q_7(\beta)({Z^\pm}^4-\sigma^4) 
    \nonumber\\
    &+\O\left((t-\tau(\beta))\left|{Z^\pm}^4
    -\sigma^4\right| 
    +\left|{Z^\pm}^5-\sigma^5 \right|\right), 
    \label{Eqn:phiBetaExp}
\end{align} 
which 
implies that  
$$
    \beta^\pm-\beta=\O(\delta_\zeta^2). 
$$
Note that the Taylor expansion for  $x$ is given by
\begin{align}
    x(Z^\pm,\beta^\pm)-x(\sigma,\beta)&= 
    x_\sigma(\sigma,\beta)(Z^\pm-\sigma) 
    +\frac{1}{2}x_{\sigma\sigma}(Z^\pm-\sigma)^2
    +\O(|Z^\pm-\sigma|^3+|\beta^\pm-\beta|) 
    \nonumber\\
    &=r_1(Z^\pm-\sigma)+\frac{1}{2}\partial_\sigma 
    r_1(Z^\pm-\sigma)^2+\O(\delta_\zeta^{3/2}), 
    \label{Eqn:xPmMinusx}
\end{align} we can derive  the following expansion 
\begin{align}
    &\bfv^\pm-\bfw=\bfw(x(Z^\pm,\beta^\pm)) 
    -\bfw(x(\sigma,\beta)) 
    \nonumber\\
    &=D_x\bfw(x(Z^\pm,\beta^\pm)-x(\sigma,\beta)) 
    +\frac{1}{2}D_x^2 \bfw 
    [x(Z^\pm,\beta^\pm) 
    -x(\sigma,\beta),x(Z^\pm,\beta^\pm) 
    -x(\sigma,\beta)]
    \nonumber\\
    &\quad+\O(|x(Z^\pm,\beta^\pm) 
    -x(\sigma,\beta)|^3) 
    \nonumber\\
    &=D_x\bfw r_1(Z^\pm-\sigma)+\frac{1}{2} 
    D_x\bfw \partial_\sigma r_1 (Z^\pm-\sigma)^2 
    +\frac{1}{2} D_x^2\bfw [r_1,r_1] 
    (Z^\pm-\sigma)^2+\O(\delta_\zeta^{3/2}) 
    \nonumber\\
    &=\lambda_1 r_1(Z^\pm-\sigma)+\frac{1}{2} 
    D_x\bfw \partial_\sigma r_1 (Z^\pm-\sigma)^2 
    +\frac{1}{2}D_x^2\bfw [r_1,r_1] 
    (Z^\pm-\sigma)^2+\O(\delta_\zeta^{3/2}).
    \label{Eqn:vpmMinusW}
\end{align} 
Multiplying by $l_1^\top$ on the both sides of (\ref{Eqn:vpmMinusW}) reduces to a scalar equation  
$$
    l_1^\top(\bfv^\pm-\bfw)=\lambda_1(Z^\pm 
    -\sigma)+\frac{1}{2}l_1^\top\left( 
    D_x\bfw \partial_\sigma r_1
    +D_x^2\bfw [r_1,r_1]\right)(Z^\pm-\sigma)^2 
    +\O(\delta_\zeta^{3/2}). 
$$
 Note that the second term
$$
    \frac{1}{2}l_1^\top\left( 
    D_x\bfw \partial_\sigma r_1
    +D_x^2\bfw [r_1,r_1]\right)(Z^\pm-\sigma)^2
$$
is independent of $S_0$ at order of  $\delta_\zeta$, we have 
$$
    l_1^\top (\bfv^\pm-\bfw)=\mp\omega_1
    \sqrt{c_1}\delta_\zeta^{1/2}+\omega_1\left( 
    \frac{3}{2}S_0+\tilde\psi(\zeta)\right) 
    \delta_\zeta + \O(\delta_\zeta^{3/2}), 
$$
where $\tilde\psi$ is independent of $S_0$. This can be simplified as 
$$
    l_1^\top (\bfv^\pm-\bfw)=\mp\omega_1^2 
    \sqrt{\frac{6\delta_\zeta}{\omega_3}}
    +\omega_1\left(\frac{3}{2}S_0 
    +\tilde\psi(\zeta)\right) 
    \delta_\zeta + \O(\delta_\zeta^{3/2}).   
$$
In turn, it implies that   
\begin{align}
    l_1^\top(\bfv^\pm-y_t)=\mp\omega_1^2 
    \sqrt{\frac{6\delta_\zeta}{\omega_3}} 
    +\omega_1\left(\frac{5}{2}S_0+\tilde\psi(\zeta)
    \right)\delta_\zeta+\O(\delta_\zeta^{3/2}). 
    \label{Eqn:lTvpmyt}
\end{align}

Our goal is to obtain the order of functions $F, G$.  Note that $\hat A_1=\O(\delta_\zeta), A_2=\O(1)$,
We can analyze  the cross product
\begin{align*}
    p\times q&=(I+tD_x\bfw)J\theta_{\zeta_1}
    \times(I+tD_x\bfw)J\theta_{\zeta_2} 
    \nonumber\\
    &=\det\left(I+tD_x\bfw\right) 
    \left(\left(I+tD_x\bfw\right)^{-1}
    \right)^\top
    J\theta_{\zeta_1}\times J\theta_{\zeta_2} 
    \nonumber\\
    &=\O(\delta_\zeta)
    \left(\left(I+tD_x\bfw \right)^{-1}\right)^\top
    J\theta_{\zeta_1}\times J\theta_{\zeta_2}.
\end{align*} 
In turn, we obtain the scalar triple product
\begin{align}
    \mu:&=(p\times q)^\top r_1= 
    (1+t\lambda_2)(1+t\lambda_3)\left\langle 
    J\theta_{\zeta_1}\times J\theta_{\zeta_2}, 
    r_1\right\rangle. 
    \label{Eqn:muOrder}
\end{align} 
Since $\beta(\tau(\zeta),\zeta)=\zeta$,  we conclude at $t=\tau(\zeta)$, 
\begin{align}
    &\mu_0(\zeta)=(1+\tau\lambda_2)(1+\tau\lambda_3) 
    \left\langle 
    R_{0,1}(\zeta)\times R_{0,2}(\zeta), 
    r_1\right\rangle>0. 
    \label{Eqn:mu0_ineq}
\end{align} 
The sign of the triple product is known explicitly within some neighborhood of $x_0$ owing to the chosen orientation of $p\times q$. Indeed, by our parametrization, $
    p\times q = |p\times q|\bfn, 
$ where $\bfn$ is the unit normal pointing from the negative branch to the positive branch (and, thus, the higher velocity side to the lower velocity side). We further note that 
$$
    x(Z^+,\beta^+)-x(Z^-,\beta^-)=r_1(Z^+-Z^-) 
    +\O(\delta_\zeta) 
    =2\sqrt{c_1\delta_\zeta}r_1+\O(\delta_\zeta), 
$$
which implies that, for $\delta_\zeta\ll1$, $r_1$ also points from the minus branch to the plus branch. This leads to the strict positivity 
$$
    \left\langle r_2^*\times r_3^*,r_1 \right\rangle>0  \quad\text{at} \quad x_0.
$$ Thus,  the strict positivity of (\ref{Eqn:mu0_ineq}) also holds in some neighborhood of the singularity. 

We then move to the expansion of $G$. To this end, recalling (\ref{Eqn:lTvpmyt}) gives  
$$
    l_1^\top G=\frac{4\mu\omega_1^3}{\omega_3} 
    \delta_\zeta^{1/2}\left(a_1\sqrt{6\omega_3} 
    \left(\tilde\psi+\frac{5}{2}S_0 \right)
    -3\omega_1a_2\right)+\O(\delta_\zeta) 
$$
By (\ref{Eqn:G_Order}) and the definition of $F$, we have 
$
G=\O(\delta_\beta
    ^{1/2}), 
F=\O(1),
$ which combined with Eq. (\ref{Eqn:Sigma_tt}), allow us to rewrite the evolution equation  of $\sigma$, that is, 
$$
    \sigma_{tt}=\frac{F_\sigma}{\delta_\zeta}, 
$$
where 
\begin{equation*}
   F_\sigma:=\delta_\zeta\left[ 
    \frac{l_1^\top F}{1+t\lambda_1}
    -l_1^\top J_\beta \mathcal{P}^{-1} 
    L_*^\top F \right] 
    +\frac{1}{\hat\eta}\sqrt{\frac{\delta_\zeta}
    {\delta_\beta}}\left[
    \frac{l_1^\top G_0}{1+t\lambda_1} 
    -l_1^\top J_\beta \mathcal{P}^{-1} 
    L_*^\top G_0 
    \right],
    \quad G_0:=\sqrt{\delta_\beta}G.
\end{equation*}

Furthermore, by introducing 

$$
    F_\beta:=\mathcal{P}^{-1}L_*^\top F,\quad 
    G_\beta:=\frac{1}{\hat\eta}\sqrt{\frac{t-\tau(\zeta)}
    {t-\tau(\beta)}}\mathcal{P}^{-1}L_*^\top G_0, 
$$
we have the equation for $Z$,
$$
    Z_t=\beta_{tt}=F_\beta+\frac{G_\beta}{\delta_\zeta}. 
$$
Thus, we obtain the following first order system of equations,  
\begin{align} 
    \left\{
    \begin{aligned}
        &\hat\eta_t=\frac{H_2(t,\zeta,\hat\theta,
        \theta_t,\theta_\zeta)}{t-\tau(\zeta)} , 
        \\
        &\hat\sigma_t=\frac{S-\hat\sigma}{t-\tau(\zeta)}, 
        \\
        &\hat\beta_t=\frac{Z-\hat\beta}
        {t-\tau(\zeta)} ,
        \\
        &S_t=\frac{F_\sigma(t,\zeta,\hat\theta,
        \theta_t, \theta_\zeta)}
        {t-\tau(\zeta)} ,
        \\
        &Z_t=F_\beta+\frac{G_\beta}{t-\tau(\zeta)}.
    \end{aligned}
    \right. \label{Eqn:FirstOrderSystem}
\end{align}
By the previous discussions, we know that at $t=\tau(\zeta)$,
\begin{align}
    &\eta(\tau(\zeta),\zeta)
    =\sigma(\tau(\zeta),\zeta)=0 ,
    \\
    &\beta(\tau(\zeta),\zeta)=\zeta ,
    \label{Eqn:BetaBoundary} 
    \\
    &\bfv(\tau(\zeta),\zeta)=\bfw
    (\gamma^*(\zeta)) , 
    \\
    &Z(\tau(\zeta),\zeta)=0 . 
\end{align} 
Thus, it remains to determine $\hat\eta(\tau(\zeta),\zeta),
S(\tau(\zeta),\zeta)$. As in \cite{BCH2}, these will be uniquely determined by requiring  that $H_2(t,\zeta,
\hat\theta,\theta_t,\theta_\zeta), F_\sigma(t,\zeta,\hat\theta,
\theta_t,\theta_\zeta)$ vanish at $t=\tau(\zeta)$.

\subsection{Determining Boundary Conditions}

We first define boundary data, 
$$
    \hat\eta(\tau(\zeta),\zeta)=\rho_0(\zeta),\quad
 \hat\sigma(\tau(\zeta),\zeta) 
    =S(\tau(\zeta),\zeta)=S_0(\zeta),
    \quad \hat\beta(\tau(\zeta),\zeta) 
    =Z(\tau(\zeta),\zeta)=Z_0(\zeta) .
$$
As demonstrated in Eq. (\ref{Eqn:ZBoundaryCond}), $Z$ vanishes along the boundary $t=\tau(\zeta)$. Based on this, we are able to  determine $\rho_0(\zeta)$. We first observe that at $t=\tau(\zeta)$, $H_2=0$ implies $\rho_0(\zeta)=2H_1$. Furthermore, since 
$
    \beta_t(\tau(\zeta),\zeta)=Z_0(\zeta)=0,$
we obtain 
$
    J\theta_t=r_1S_0,  
$
along the boundary. Therefore, 
\begin{equation}
    \left.(I+\tau D_x\bfw)J\theta_t\right|
    _{t=\tau(\zeta)}=0. 
    \label{Eqn:a2QTheta_t}
\end{equation}
By  vanishing of $\hat A_1$ at $t=\tau(\zeta)$, we obtain 
\begin{align}
    \rho_0(\zeta)&=-4a_1A_2^\top p\times q 
    \nonumber\\
    &=4a_1\omega_1(\zeta)^2\sqrt{\frac{6}{\omega_3(\zeta)}}r_1^\top(\zeta) p\times q 
    \nonumber\\
    &=4a_1\mu_0\omega_1(\zeta)^2\sqrt{\frac{6}{\omega_3(\zeta)}} 
    > 0 . 
    \label{Eqn:rho0}
\end{align} 

\begin{remark} Under the planar reduction, this reduces precisely to the two-dimensional formula, as shown in (5.36) of \cite{BCH2}. \end{remark} 

We finally seek to determine $S_0(\zeta)$. Due to our imposition that $F_\sigma$ vanishes along the boundary, we have 
\begin{align}
    &\left.F_\sigma\right|_{\delta_\zeta=0}=
    \left.\left[\delta_\zeta\left( 
    \frac{l_1^\top F}{1+t\lambda_1}-l_1^\top 
    J_\beta\P^{-1}L_*^\top F
    \right) 
    +\frac{1}{\hat\eta}\sqrt{\frac{\delta_\zeta}
    {\delta_\beta}}\left( 
    \frac{l_1^\top G_0}{1+t\lambda_1} 
    -l_1^\top J_\beta \P^{-1}L_*^\top 
    G_0 
    \right)\right]\right|_{\delta_\zeta=0}=0,  
    \label{Eqn:FsigmaComp}
\end{align} 
where our aim (as in 2D case) is to demonstrate that (\ref{Eqn:FsigmaComp}) allows us to recover a unique value of $S_0(\zeta)$. To this end, we need to examine each term separately. 

First, recalling  
$F=\O(1), G_0=\O\left(\delta_\zeta\right)$, 
the following two terms vanish in the limit $t\rightarrow\tau(\zeta)$. 
$$
    \delta_\zeta l_1^\top J_\beta \P^{-1} 
    L_*^\top F \rightarrow 0, 
    \quad
    \frac{1}{\hat\eta}\sqrt{\frac{\delta_\zeta}
    {\delta_\beta}}l_1^\top J_\beta \P^{-1} 
    L_*^\top G_0 \rightarrow 0. 
$$
We now examine the first term in \eqref{Eqn:FsigmaComp}. Recall  
$$
    F=-(I+tD_x\bfw)\Lambda_J(\theta_t) 
    -2D_x\bfw J\theta_t-tD^2_x\bfw 
    [J\theta_t,J\theta_t], 
$$ 
and
$$
    J\theta_t=S_0r_1+\O(\delta_\zeta) 
    \quad
    1+t\lambda_1=-\omega_1\delta_\zeta 
    +\O(\delta_\zeta^2). 
$$ 
We observe that 
$$
    l_1^\top(I+tD_x\bfw)=(1+t\lambda_1)l_1^\top, 
$$
and as $t$ tends to $\tau(\zeta)$, 
\begin{align*}
    \delta_\zeta\frac{l_1^\top(-(I+tD_x\bfw)
    \Lambda_J)}{1+t\lambda_1}=-\delta_\zeta 
    l_1^\top \Lambda_J \longrightarrow 0. 
\end{align*} 

Furthermore, noting that 
\begin{align*}
    l_1^\top D_x\bfw J\theta_t 
    =\lambda_1l_1^\top J\theta_t 
    =\lambda_1S_0+\O(\delta_\zeta), 
\end{align*} 
we obtain the limit 
\begin{align*}
    \delta_\zeta\frac{-2l_1^\top D_x\bfw J\theta_t}
    {1+t\lambda_1}\longrightarrow -2S_0,\quad t\rightarrow \tau(\zeta). 
\end{align*} 

We finally claim that the quadratic part vanishes in \eqref{Eqn:FsigmaComp}, as in \cite{BCH2}. First, observe that along the boundary, 
$J\theta_t\rightarrow S_0r_1 
$,
which means we only need to show 
\begin{align*}
    l_1^\top D_x^2\bfw[r_1,r_1] 
\end{align*} 
vanishes in the limit as $\delta_\zeta\rightarrow0$. To this end, differentiating  both sides of 
$$
    D_x\bfw(x) r_1(x) =\lambda_1(x)r_1(x)
$$
in the direction of $r_1$ yields  
$$
    (D_{r_1}D_x\bfw(x))r_1(x)+D_x\bfw(x) 
    (D_{r_1}r_1(x))=(D_{r_1}\lambda_1(x))r_1(x) 
    +\lambda_1(x)D_{r_1}r_1(x). 
$$
Then, noting  $D_x^2\bfw[r_1,r_1]=(D_{r_1}D_x\bfw(x))r_1(x)$, we have 
\begin{align*}
    D_x^2\bfw(x)[r_1(x),r_1(x)]+D_x\bfw(x) 
    (D_{r_1}r_1(x))=(D_{r_1}\lambda_1(x))r_1(x) 
    +\lambda_1(x)D_{r_1}r_1(x). 
\end{align*} 
Multiplying both sides by $l_1^\top$ gives 
\begin{align*}
    l_1^\top(x)D_x^2\bfw(x)[r_1(x),r_1(x)]+\lambda_1 
    l_1^\top(x)(D_{r_1}r_1(x))
    =D_{r_1}\lambda_1(x) 
    +\lambda_1(x)l_1^\top(x)D_{r_1}r_1(x). 
\end{align*} 
which follows that
\begin{align*}
    l_1^\top(x)D_x^2\bfw(x)[r_1(x),r_1(x)]
    =D_{r_1}\lambda_1(x). 
\end{align*} 

Noting the fact that our singular surface $\Gamma^*$ is defined according to the vanishing of 
$
    D_{r_1}\lambda_1(x)=(\nabla\lambda_1(x))\cdot 
    r_1(x), 
$
we can obtain the desired limit, 
\begin{align*}
    l_1^\top(x)D_x^2\bfw(x)[r_1(x),r_1(x)]=0. 
\end{align*} 
This gives the total contribution from $F$ as 
\begin{align*}
    \delta_\zeta\frac{l_1^\top F}
    {1+t\lambda_1}\rightarrow -2S_0, 
\end{align*}
and allows us to recover the following  equation in $S_0(\zeta)$ for $F_\sigma$. 
\begin{align}
    \lim_{t\rightarrow\tau(\zeta)}
    F_\sigma=-2S_0(\zeta)
    +\lim_{t\rightarrow\tau(\zeta)}
    \frac{l_1^\top G_0}
    {\rho_0(\zeta)(1+t\lambda_1)} = 0 . 
    \label{Eqn:Fsigma_int}
\end{align} 

It remains to show that the coefficient for $S_0$ remains nonzero in the limit as $\delta_\zeta\rightarrow0$ and thus determines a unique initial value for $S_0$. Therefore, we now examine the second term in the summation \eqref{Eqn:Fsigma_int}. 

Recall the form of $G_0$ and we have   
$$
    l_1^\top G_0=\frac{4\mu_0\omega_1^3}{\omega_3} 
    \delta_\zeta \left( 
    a_1\sqrt{6\omega_3}\left( 
    \tilde\psi+\frac{5}{2}S_0
    \right)-3\omega_1a_2
    \right)+\O(\delta_\zeta^{3/2}). 
$$
By the expansion of $1/\lambda_1$ given in (\ref{Eqn:Tau_Taylor_Exp}) and the computation of $\rho_0$ in (\ref{Eqn:rho0}), it follows that 
$$
    \frac{l_1^\top G_0}{\rho_0(1+t\lambda_1)} 
    \rightarrow 
    -\left(\tilde\psi+\frac{5}{2}S_0\right) 
    +\frac{a_2}{2a_1}\sqrt{c_1(\zeta)}, 
$$
where $c_1$ is defined in (\ref{Eqn:c_1}). We  claim that when $t\rightarrow\tau(\zeta)$, 
$$
    \frac{a_2}{2a_1}\sqrt{c_1(\zeta)} 
    =\frac{3}{4}S_0+\mathcal{C}, 
$$
where $\mathcal{C}$ does not depend on $S_0$. Indeed, the coefficients $a_1,a_2$, as defined in Equation (\ref{Eqn:rhopmForm}), may be determined through evaluating Equation (\ref{Eqn:densityTransport}) at $X^\pm$. Note that the expansion of $\lambda_1$ is given by  
$$
    \lambda_1(X^\pm)-\lambda_1(\beta) 
    =\frac{1}{2}\omega_3(Z^\pm)^2 
    +\frac{1}{6}\omega_4(Z^\pm)^3 
    +O(1)(Z^\pm)^4. 
$$
Then,
combining this with (\ref{Eqn:ZpmFinal}) and (\ref{Eqn:phiBetaExp}) gives 
\begin{equation}
    \lambda_1(X^\pm)-\lambda_1(\beta) 
    =\delta_\beta\left(\frac{1}{2}\omega_3c_1(\beta) 
    \pm\sqrt{\delta_\beta}L\right)
    +\O(\delta_\beta^{2}), 
    \label{Eqn:Lambda1Zpm}
    \end{equation}
    where $$
    L:=-\frac{1}{2}\omega_3\sqrt{c_1(\beta)}\left( 
    c_2(\tilde\sigma,\beta)+c_1(\beta)c_3(\beta)
    \right)+\frac{1}{6}\omega_4c_1(\beta)^\frac{3}{2} .
$$ 
For convenience of future discussions, by the definitions of $c_1(\beta),c_2(\tilde\sigma,\beta)$ and  $c_3(\beta)$ in Equations (\ref{Eqn:c_1})--(\ref{Eqn:c_3}), we sometimes decompose $L$ into two terms
$$
    L(\beta,\tilde\sigma)=-\frac{1}{2}\omega_3\sqrt{c_1(\beta)}\tilde\sigma 
    +L_1(\beta).
$$
Combining this with (\ref{Eqn:1PlusTLambda_1}) gives 
$$
    1+t\lambda_1(X^\pm)= 
    \frac{2\delta_\beta}{\tau(\beta)} 
    \left(1\pm\frac{1}{2}\tau(\beta)^2
    \sqrt{t-\tau(\beta)}L 
    +\O(t-\tau(\beta))
    \right), 
$$
and
$$
    \frac{1}{1+t\lambda_1(X^\pm)} 
    =\frac{\tau(\beta)}{2(t-\tau(\beta))} 
    \left(1\pm f_0\sqrt{t-\tau(\beta)} 
    +\O(\delta_\beta)\right), 
$$
where $f_0$ is analytic and at $t=\tau(\beta)$, it satisfies
\begin{align*}
    f_0=-\frac{1}{2}\tau(\beta)^2L. 
\end{align*} 
 Then, by noting (\ref{Eqn:rhoBarOverLambda23}), we obtain the following formulas for $a_1,a_2$ defined in (\ref{Eqn:rhopmForm}), as $\delta_\zeta\rightarrow0$,  
$$
    a_1(\tau(\beta),\tilde\sigma,\beta)
    =\frac{\tau(\beta)}{2}\theta_1 ,
    \quad
    a_2(\tau(\beta),\tilde\sigma,\beta)
    =\frac{\tau(\beta)}{2}(\theta_1f_0+\theta_2) .
$$
Therefore, letting $\delta_\zeta\rightarrow0$ gives the claimed identity
\begin{align*}
    \frac{a_2}{2a_1}\sqrt{c_1(\beta)} 
    &=\frac{1}{2}\sqrt{c_1(\zeta)}\left(f_0
    +\frac{\theta_2}{\theta_1}\right) 
    \nonumber\\
    &=-\frac{\tau(\zeta)^2}{4}\sqrt{c_1(\zeta)}L
    +\frac{1}{2}\sqrt{c_1(\zeta)} 
    \frac{\theta_2}{\theta_1} 
    \nonumber\\
    &=\frac{3}{4}S_0-\frac{\tau(\zeta)^2}{4} 
    \sqrt{c_1(\zeta)}L_1(\zeta) 
    +\frac{1}{2}\sqrt{c_1(\zeta)}
    \frac{\theta_2}{\theta_1} 
    \nonumber\\
    :&=\frac{3}{4}S_0+L_2(\zeta), 
\end{align*}
where, owing to the fact that at the boundary, $\theta_1,\theta_2$ only depend on $\beta$, we observe  that $L_2(\zeta)$ does not depend on $S_0$. Thus,  $S_0$ is determined uniquely via 
$$
    S_0(\zeta)=\frac{4}{15}\left(L_2(\zeta) 
    -\tilde\psi(\zeta)\right).
$$

\subsection{Evolution in Time of the Singular Surface} 

With our initial conditions determined, we now seek a solution of the form $(\zeta,t)\rightarrow(\eta,\theta)$. Define the following vector as the solution to the first order system of equations (\ref{Eqn:FirstOrderSystem}). 
\begin{align}
    \mathbf{u}:= \begin{bmatrix}
        \hat\eta\\\hat\sigma\\\hat\beta_1 
        \\\hat\beta_2\\S\\Z_1\\Z_2 
    \end{bmatrix} \in\mathbb{R}^7 . 
    \label{Eqn:SolutionVector}
\end{align} 

Furthermore, we introduce a further renaming of variables,
$
    \tilde t := t-\tau(\zeta),
    \tilde\bfu(\tilde t,\zeta):=
    \bfu(\tilde t+\tau(\zeta),\zeta). 
$
By the chain rule, we have 
$$
    \partial_{\tilde t}\tilde\bfu
    (\tilde t,\zeta)
    =\partial_t\bfu(\tilde t+\tau(\zeta), 
    \zeta), \quad
\partial_{\zeta_i}\tilde \bfu(\tilde t, 
    \zeta)=
    \partial_{\zeta_i}\bfu(\tilde t+\tau(\zeta), 
    \zeta)
    +\tau_{\zeta_i}(\zeta)\partial_t 
    \bfu(\tilde t+\tau(\zeta),\zeta). 
$$
In addition, by noting  that 
$$\sigma=\tilde t\hat\sigma, 
    \quad \beta_j=\zeta_j+\tilde t\hat\beta_j,
$$
we have 
\begin{align}
    &\partial_{\zeta_i}\sigma= 
    \tilde t\partial_{\zeta_i}\hat\sigma 
    -\tau_{\zeta_i}(\zeta)S 
    \label{Eqn:sigma_zetai}\\
    &\partial_{\zeta_i}\beta_j= 
    \delta_{ij}+
    \tilde t\partial_{\zeta_i}\hat\beta_j 
    -\tau_{\zeta_i}(\zeta)Z_j. 
    \label{Eqn:betaj_zetai}
\end{align}
As a consequence,  $\tilde\bfu$ evolves in time according to  the following system of equations 
$$
    \tilde\bfu_{\tilde t}=\frac{1}{\tilde t} 
    \Omega(\tilde t,\zeta,\tilde\bfu, 
    \tilde\bfu_\zeta)+\Upsilon(\tilde t,
    \zeta,\tilde\bfu, \tilde\bfu_\zeta), 
$$
where 
\begin{align*}
    \Omega:= \begin{bmatrix}
    H_2
    \\S-\hat\sigma
    \\Z_1-\hat\beta_1
    \\Z_2-\hat\beta_2
    \\F_\sigma
    \\G_{\beta_1}
    \\G_{\beta_2}
    \end{bmatrix}, \qquad 
    \Upsilon:=\begin{bmatrix}
        0\\0\\0\\0\\0
        \\F_{\beta_1}\\F_{\beta_2} 
    \end{bmatrix}, 
\end{align*} 
and the initial data are given by  
\begin{align*}
    \tilde\bfu(0,\zeta):=\tilde\bfu_0(\zeta 
    =\begin{bmatrix}
        \rho_0(\zeta)\\S_0(\zeta)\\0\\0
        \\S_0(\zeta)\\0\\0 
    \end{bmatrix}.
\end{align*} 

Recalling at $\tilde{t}=0$, 
$$
    H_2=0,\qquad S-\hat\sigma=0, 
    \qquad Z-\hat\beta=0, \qquad 
    F_\sigma=0, \qquad G_0=0, 
$$
and noting that 
 $G_0=0$ implies $G_\beta=0,$
we obtain that at $\tilde{t}=0$,  $\Omega$ satisfies 
$$
    \Omega(0,\zeta,\tilde\bfu_0(\zeta), 
    \tilde\bfu_{0,\zeta}(\zeta))=0 
$$ 
for all $\zeta$.

 Then, we introduce the new unknows  $\vartheta:=\tilde\bfu-\tilde\bfu_0(\zeta)$ and arrive at the new system of equations   
\begin{align}
    &\vartheta_{\tilde t}=\frac{1}{\tilde t} 
    \mathbf{F}(\tilde t,\zeta,\vartheta, 
    \vartheta_\zeta)+\mathbf{G}(\tilde 
    t,\zeta,\vartheta,\vartheta_\zeta),\label{Eqn:vartheta}
    \\   
    &\vartheta(0,\zeta)=0,\label{Eqn:varthetaData}  
\end{align} 
where $\mathbf{F},\mathbf{G}$ are  analytic functions satisfing  $\mathbf{F}(0,\zeta,0,0)=0$. Equivalently, Eq. (\ref{Eqn:vartheta}) can be written by 
\begin{align}
    &\tilde t\vartheta_{\tilde t}= 
    \mathbf{H}(\tilde t,\zeta,\vartheta, 
    \vartheta_\zeta) , 
    \label{Eqn:BriotBouquet}
\end{align} 
where $\mathbf{H}(0,\zeta,0,0)=0$. Next we want to  demonstrate that $\mathbf{H}_{\vartheta_\zeta}
(0,\zeta,0,0)=0$. Since $\tilde t=0$, this is equivalent to showing 
$
    \frac{\partial\Omega}{\partial 
    \vartheta_\zeta}=0.
$
However, by construction, 
$$
    \frac{\partial (\partial_{\zeta_i}\beta_j)}{\partial 
    (\partial_{\zeta_k}\hat\beta_m)} = 
    \tilde t \delta_{ik}\delta_{jm}=0,
    \quad
    \frac{\partial(\partial_{\zeta_i}\sigma)}
    {\partial(\partial_{\zeta_k}\hat\sigma)}=
    \tilde t\delta_{ik}=0, 
$$
at $\tilde t=0$. Thus, 
$
    \mathbf{H}_{\vartheta_\zeta}
    (0,\zeta,0,0)=0. 
$

Therefore, we conclude that the system in Equation (\ref{Eqn:BriotBouquet}) is of Briot-Bouquet type. Following the results of \cite{li}, it remains to show that none of the eigenvalues of $\Xi:=
\frac{\partial\mathbf{H}}
{\partial\vartheta}(0,\zeta,0,0)$ are positive integers at $\tilde t=0$. Based on  the calculations in Appendix \ref{App:Partials}, this matrix takes the following form, 
\begin{align} 
    \Xi = 
    \begin{bmatrix}
        -\frac{1}{2}&0&*&*&0&*&* 
        \\
        0&-1&0&0&1&0&0 
        \\
        0&0&-1&0&0&1&0 
        \\
        0&0&0&-1&0&0&1 
        \\
        *&-\frac{3}{4}&*&*&-3&*&*
        \\
        0&0&0&0&0&-\frac{1}{2}&0 
        \\
        0&0&0&0&0&0&-\frac{1}{2} 
    \end{bmatrix}, 
    \label{Eqn:XI_Final} 
\end{align} 
which has the following eigenvalues 
\begin{align}\label{last}
    \left\{ 
    -\frac{1}{2}, -\frac{1}{2}, -\frac{1}{2}, 
    -1, -1, -\frac{3}{2},-\frac{5}{2}
    \right\}.  
\end{align} 
Obviously, 
none of these eigenvalues is a positive integer. Hence, the results of \cite{li} yield a unique local analytic solution to the system of equations  (\ref{Eqn:BriotBouquet}) with initial data (\ref{Eqn:varthetaData}). Finally, the proof of Theorem 5.1 is completed.

\begin{appendices}

\section{Algebraic Simplification of Terms}
\label{App:AlgebraicSimplification} 

\subsection{Computation of $H$}
\setcounter{equation}{0} 
\label{App:H_Computation} 

We compute the form of (\ref{Eqn:HDefinition}), namely, 
\begin{align}
    H=(\rho^+(\bfw-\bfv^+)-\rho^-(\bfw-\bfv^-) 
    +(\rho^+-\rho^-)(I+tD_x\mathbf{w})J\theta_t)^\top 
    p\times q. 
    \label{Eqn:H_Appendix_Unsimplified}
\end{align} 
Firstly, we compute the difference between $\bfv^\pm$ and $\bfw$. By using Theorem \ref{Theorem:vpm} and Eq. (\ref{Eqn:wExpanded}), we obtain 
\begin{align*}
    &\bfv^\pm-\bfw(0,\beta)=\delta_\beta A_1 
    \pm \sqrt{\delta_\beta}A_2, 
    \\
    &\bfw(\sigma,\beta)-\bfw(0,\beta) 
    =\delta_\beta A_3, 
    \\
    &\bfv^\pm-\bfw(\sigma,\beta)=\hat A_1 
    \pm\sqrt{\delta_\beta}A_2. 
\end{align*} 

Recall the form of $\rho^\pm$
$$
    \rho^\pm=a_1\delta_\beta^{-1}\pm 
    a_2\delta_\beta^{-1/2}. 
$$
Based on this, we compute the form of (\ref{Eqn:H_Appendix_Unsimplified}) term by term.   Indeed, 
\begin{align*}
    \rho^+(\bfw-\bfv^+)&=(a_1\delta_\beta^{-1} 
    +a_2\delta_\beta^{-1/2})\left( 
    -\hat A_1-\sqrt{\delta_\beta}A_2 
    \right) 
    \nonumber\\
    &=-\delta_\beta^{-1}\left( 
    a_1\hat A_1+a_1A_2\delta_\beta^{1/2} 
    +a_2\hat A_1\delta_\beta^{1/2} 
    +a_2A_2\delta_\beta 
    \right).
\end{align*}
Then, we have
\begin{align*}
    \rho^-(\bfw-\bfv^-)&=(a_1\delta_\beta^{-1} 
    -a_2\delta_\beta^{-1/2})\left( 
    -\hat A_1+A_2\delta_\beta^{1/2} 
    \right) 
    \nonumber\\
    &=\delta_\beta^{-1}\left( 
    -a_1\hat A_1+a_1A_2\delta_\beta^{1/2} 
    +a_2\hat A_1\delta_\beta^{1/2} 
    -a_2A_2\delta_\beta 
    \right). 
\end{align*} 

Subtracting the above two formulas gives 

\begin{align*}
    \rho^+(\bfw-\bfv^+)-\rho^-(\bfw-\bfv^-)=
    -2\delta_\beta^{-1/2}\left( 
    a_1A_2+a_2\hat A_1
    \right). 
\end{align*} 

Further noting that 
\begin{align*}
    \rho^+-\rho^-=2a_2\delta_\beta^{-1/2} 
\end{align*} 
yields the desired form for (\ref{Eqn:H_Appendix_Unsimplified}), that is,  
\begin{align}
    H=2\delta_\beta^{-1/2}\left(
    -a_1A_2-a_2\hat A_1 
    +a_2(I+tD_x\bfw)J\theta_t
    \right)^\top p\times q. 
    \label{Eqn:H_simp_Appendix}
\end{align}

\subsection{Computation of $G$}
\label{App:G_Computation} 

We compute the leading-order behavior of function   
$$
    G=\rho^+\left((y_t
    -\bfv^+)^\top p\times q\right)(\bfv^+-y_t)
    -\rho^-\left((y_t-\bfv^-)^\top p\times q\right)
    (\bfv^--y_t). 
$$

By (\ref{Eqn:vpmMinusW}), we have 
$$
    \bfv^\pm-\bfw=\lambda_1 r_1(Z^\pm-\sigma)+\frac{1}{2} 
    D_x\bfw \partial_\sigma r_1 (Z^\pm-\sigma)^2 
    +\frac{1}{2}D_x^2\bfw [r_1,r_1] 
    (Z^\pm-\sigma)^2+\O(\delta_\zeta^{3/2}).$$

Furthermore, we recall that 
\begin{align*}
    &Z^\pm-\sigma=\pm\sqrt{c_1(\beta)}
    \delta_\beta^{1/2}+\left(-\frac{3}{2}S_0+b^*(\zeta) 
    \right)\delta_\beta+\O(\delta_\beta^{3/2}), 
    \\
    &(Z^\pm-\sigma)^2=c_1\delta_\zeta
    +\O(\delta_\zeta^{3/2}). 
\end{align*} 

Using this and the expansion of $\lambda_1$ in $\omega_k$, we obtain 
\begin{align*}
    \bfv^\pm-\bfw=&\mp\omega_1 r_1
    \sqrt{c_1(\beta)}\delta_\zeta^{1/2}
    +\Bigg(\omega_1\left(\frac{3}{2}S_0-b^*(\zeta) 
    \right)r_1 
    \nonumber\\
    &+\frac{c_1}{2}\left(D_x\bfw \partial_\sigma 
    r_1+D_x^2\bfw [r_1,r_1] 
    \right)\Bigg)\delta_\zeta
    +\O(\delta_\zeta^{3/2}).
\end{align*} 

Furthermore, Eq. (\ref{Eqn:B_Comp}) implies that 
$$
    y_t-\bfw=\left(-\omega_1 S_0 r_1+(I+\tau(\zeta)D_x\bfw)R_0Z_1\right)
    \delta_\zeta 
    +\O(\delta_\zeta^2). 
$$
Subtracting the above two equations gives 
\begin{align*}
    &\bfv^\pm-y_t=\mp\omega_1r_1\sqrt{c_1(\beta)} 
    \delta_\zeta^{1/2}+\Bigg(\omega_1
    \left(\frac{5}{2}S_0-b^*(\zeta) 
    \right)r_1 
    \nonumber\\
    &+\frac{c_1}{2}\left(D_x\bfw \partial_\sigma 
    r_1+D_x^2\bfw [r_1,r_1] 
    \right)
    -(I+\tau(\zeta)D_x\bfw)R_0Z_1
    \Bigg)\delta_\zeta 
    +\O(\delta_\zeta^{3/2}), 
\end{align*} 
which implies that 
\begin{align*}
    \bfv^\pm-y_t=\mp\omega_1^2\sqrt{
    \frac{6}{\omega_3}}r_1\delta_\zeta^{1/2}
    +\O(\delta_\zeta),
\end{align*} Here we used the identity 
$
	c_1=\frac{6\omega_1^2}{\omega_3}. 
$

We then rearrange the terms within $G$ in the following form  
\begin{align*}
    G=-\rho^+\left((\bfv^+-y_t)^\top 
    p\times q\right)(\bfv^+-y_t)
    +\rho^-\left((\bfv^--y_t)^\top p\times q\right)
    (\bfv^--y_t), 
\end{align*} 
and expand each part separately. 
Since
\begin{align*}
    (\bfv^+-y_t)^\top p\times q=
    -\mu\omega_1^2\sqrt{\frac{6}{\omega_3}} 
    \delta_\zeta^{1/2}+\O(\delta_\zeta),
\end{align*}
we have 
\begin{align*}
    \left((\bfv^+-y_t)^\top 
    p\times q\right)(\bfv^+-y_t)= 
    \frac{6\omega_1^4}{\omega_3}\mu r_1
    \delta_\zeta+\O(\delta_\zeta^{3/2}). 
\end{align*} 

A similar computation also holds for the second term
\begin{align*}
    \left((\bfv^--y_t)^\top 
    p\times q\right)(\bfv^--y_t)= 
    \frac{6\omega_1^4}{\omega_3}\mu r_1
    \delta_\zeta+\O(\delta_\zeta^{3/2}). 
\end{align*} 
Inserting this into our equation for $G$ yields 
\begin{equation}
    G=(\rho^--\rho^+)\frac{6\omega_1^4}{\omega_3}\mu r_1
    \delta_\zeta+\O((|\rho^-|+|\rho^+|)
    \delta_\zeta^{3/2}) =\O(\delta_\zeta^{1/2}). 
    \label{Eqn:G_Order}
\end{equation}

\section{Computation of Partial Derivatives} 
\label{App:Partials} 

In this section, we compute the partial derivatives for entries of $\Xi$. Throughout this section, we make use of the fact that all transformed coefficient functions are analytic functions of $(\tilde t, \zeta,\tilde\bfu,\tilde\bfu_\zeta)$. Therefore, their derivatives along the boundary may be equivalently computed after taking their restriction to $\tilde t=0$.

\subsection{Row 1 Derivatives} 
\setcounter{equation}{0} 
\label{App:H_Derivatives} 

We seek to compute the derivatives of 
$$
    H_2=H_1-\frac{1}{2}\hat\eta, 
\quad
    H_1:=2\sqrt{\frac{\delta_\zeta}{\delta_\beta}} 
    \left(-a_1A_2-a_2\hat A_1+a_2B\right)^\top 
    p\times q.
$$
Next we shall compute the derivatives of each term individually. 

\subsubsection{Term 1}

Note that 
$$
    \delta_\beta=t-\tau(\beta) =\tilde t+\tau(\zeta) 
    -\tau(\zeta+\tilde t\hat\beta)=\tilde t\left(1-
    \nabla_\zeta\tau(\zeta)\cdot\hat\beta 
    \right)+\O(\tilde t^2). 
$$
Therefore, 
$$
    \sqrt{\frac{\delta_\zeta}{\delta_\beta}}
    =(1-\nabla_\zeta\tau(\zeta)
    \cdot\hat\beta)^{-1/2}+\O(\tilde t).  
$$
Since at $\tilde t=0$,  $\hat\beta=0$. We then have 
$
    \sqrt{\frac{\delta_\zeta}{\delta_\beta}}
    =1. 
$

Using this, it follows from the chain rule that 
$$
    \left.\frac{\partial}{\partial\hat\beta_j} 
    \sqrt{\frac{\delta_\zeta}{\delta_\beta}}
    \right|_{\tilde t=0}=\frac{1}{2}\tau_{\zeta_j}
    (\zeta). 
$$
Furthermore, since there is no dependence on other variables, we obtain that  
\begin{align*}
    \left.\frac{\partial}{\partial\hat\eta} 
    \sqrt{\frac{\delta_\zeta}{\delta_\beta}}
    \right|_{\tilde t=0}=\left.\frac{\partial}
    {\partial\hat\sigma}\sqrt{\frac{\delta_\zeta}
    {\delta_\beta}}\right|_{\tilde t=0}
    =\left.\frac{\partial}{\partial S}
    \sqrt{\frac{\delta_\zeta}{\delta_\beta}}
    \right|_{\tilde t=0} 
    =\left.\frac{\partial}{\partial Z_j}
    \sqrt{\frac{\delta_\zeta}{\delta_\beta}}
    \right|_{\tilde t=0}=0. 
\end{align*} 

\subsubsection{Term 2} 

At $\tilde t=0$, the following identities hold. 
$$(I+\tau(\zeta)D_x\bfw)r_1=0,  \quad
    J=[r_1,R_0], 
    \quad
   J\theta_t=r_1S+R_0Z.  
$$
Therefore, at $\tilde t=0$
$$
    B=(I+\tau(\zeta)D_x\bfw)J\theta_t, 
     =(I+\tau(\zeta)D_x\bfw)R_0Z. 
$$
Due to the vanishing of $Z$ along the boundary and the lack of dependence on $S$, we conclude that 
\begin{align}
    \left.B\right|_{\tilde t=0}=0, 
    \qquad 
    \left.\frac{\partial B}{\partial S}
    \right|_{\tilde t=0}=0, 
    \qquad 
    \frac{\partial B}{\partial Z}= 
    (I+\tau(\zeta)D_x\bfw)R_0. 
    \label{Eqn:B_bdry}
\end{align} 

Furthermore, from Eqs. (\ref{Eqn:sigma_zetai})--(\ref{Eqn:betaj_zetai}), at $\tilde t=0$ we have 
$$
    \sigma_{\zeta_i}=-\tau_{\zeta_i}S, 
    \quad
\beta_{\zeta_i}=e_i-\tau_{\zeta_i}Z
$$
By using these, we compute 
\begin{align*}
    p&=(I+\tau(\zeta)D_x\bfw)J\theta_{\zeta_1} 
    \nonumber\\
    &=(I+\tau(\zeta)D_x\bfw)(-\tau_{\zeta_1}(\zeta) 
    Sr_1+R_0(e_1-\tau_{\zeta_1}(\zeta)Z)), 
    \\
    q&=(I+\tau(\zeta)D_x\bfw)J\theta_{\zeta_2} 
    \nonumber\\
    &=(I+\tau(\zeta)D_x\bfw)(-\tau_{\zeta_2}(\zeta) 
    Sr_1+R_0(e_2-\tau_{\zeta_2}(\zeta)Z)). 
\end{align*} 
Therefore,  at $\tilde t=0$ we obtain 
\begin{align}
    &p=(I+\tau(\zeta)D_x\bfw)R_0e_1, 
    \label{Eqn:pComputed}\\
    &q=(I+\tau(\zeta)D_x\bfw)R_0e_2, 
    \label{Eqn:qComputed}\\
    &\left.\frac{\partial (p\times q)}
    {\partial S}\right|_{\tilde t=0} 
    =0, 
    \label{Eqn:pqS}
    \\
    &\left.\frac{\partial (p\times q)}
    {\partial\hat\sigma}\right|_{\tilde t=0}=0. 
    \label{Eqn:pqSigmaHat}
\end{align} 
Furthermore, 
\begin{align*}
    &\frac{\partial p}{\partial Z}=-\tau_{\zeta_1}
    (\zeta)(I+\tau(\zeta)D_x\bfw)R_0, 
    \\
    &\frac{\partial q}{\partial Z}=-\tau_{\zeta_2}
    (\zeta)(I+\tau(\zeta)D_x\bfw)R_0. 
\end{align*} 

Using this, we can compute 
\begin{align*}
    \frac{\partial (p\times q)}{\partial Z_j} 
    &=(-\tau_{\zeta_1}(I+\tau(\zeta)D_x\bfw)R_0e_j)
    \times ((I+\tau(\zeta)D_x\bfw)R_0e_2)
    \nonumber\\
    &+((I+\tau(\zeta)D_x\bfw)R_0e_1)
    \times(-\tau_{\zeta_2} 
    (I+\tau(\zeta)D_x\bfw)R_0e_j). 
\end{align*} 
For each $j=1,2$, one of these cross products vanishes. Therefore, 
$$
    \left.\frac{\partial (p\times q)}{\partial Z_j}
    \right|_{\tilde t=0}=
    \left.-\tau_{\zeta_j}(\zeta)(p\times q) 
    \right|_{\tilde t=0}. 
$$

\subsubsection{Term 3} 

At $\tilde t=0$, we note that 
$
    a_2\hat A_1=a_2B=0. 
$
Furthermore, at the boundary
$H_2=0$ implies $H_1=\frac{1}{2}\rho_0. 
$
Therefore, 
$$
    \left.-a_1A_2^\top(p\times q)
    \right|_{\tilde t=0}=\frac{1}{4}\rho_0. 
$$
Next we compute 
$$
    \left.\frac{\partial}{\partial Z} 
    \left(-a_1A_2-a_2\hat A_1+a_2B\right)
    \right|_{\tilde t=0} 
    =a_2(I+\tau(\zeta)D_x\bfw)R_0. 
$$
From (\ref{Eqn:pComputed})--(\ref{Eqn:qComputed}), we derive that 
$$
    (p\times q)^\top\left.\left(\frac{\partial}{\partial Z} 
    \left(-a_1A_2-a_2\hat A_1+a_2B\right)
    \right) 
    \right|_{\tilde t=0}=0. 
$$

It remains to compute their derivative with respect to $\hat\sigma$. First, note that (\ref{Eqn:A2Bdry})--(\ref{Eqn:a1Bdry}) implies 
\begin{align}
    \left.\frac{\partial a_1}{\partial\hat\sigma}
    \right|_{\tilde t=0}=\left.\frac{\partial A_2}
    {\partial\hat\sigma}\right|_{\tilde t=0}=
    \left.\frac{\partial (a_1A_2)}
    {\partial\hat\sigma}\right|_{\tilde t=0}=0. 
    \label{Eqn:da1A2}
\end{align} 

Furthermore, $\hat A_1=\O(\delta_\zeta)$. Then, it follows that 
\begin{align}
    \left. \frac{\partial (a_2\hat A_1)}
    {\partial\hat\sigma}\right|_{\tilde t=0} 
    =0. 
    \label{Eqn:da2A1}
\end{align}
Finally, (\ref{Eqn:B_bdry}) implies that $B$ vanishes along the boundary. Furthermore, since $\sigma=\tilde t\hat\sigma$, we have 
\begin{align}
    \left.\frac{\partial B}{\partial\hat\sigma} 
    \right|_{\tilde t=0}=0, 
    \label{Eqn:partialSigmaHatB}
\end{align} 
which means  that 
\begin{align}
    \left.\frac{\partial (a_2B)}
    {\partial\hat\sigma}\right|_{\tilde t=0} 
    =0. 
    \label{Eqn:da2B}
\end{align} 
Using (\ref{Eqn:da1A2})--(\ref{Eqn:da2A1}) and (\ref{Eqn:da2B}), we obtain 
\begin{align}
    \left.\frac{\partial}{\partial\hat\sigma} 
    \left(-a_1A_2-a_2\hat A_1+a_2B\right)
    \right|_{\tilde t=0}=0.
    \label{Eqn:d2Sigmahat}
\end{align}

\subsubsection{Derivative with respect to $\hat\eta$ and $\hat\sigma$} 

Since $H_1$ does not depend on $\hat\eta$, we obtain 
$$
    \left.\frac{\partial H_2}{\partial\hat\eta} 
    \right|_{\tilde t=0}=-\frac{1}{2}. 
$$ At the same time, 
by using Eqs. (\ref{Eqn:pqSigmaHat}) and (\ref{Eqn:d2Sigmahat}), we obtain 
$$
    \left.\frac{\partial H_2}{\partial\hat\sigma} 
    \right|_{\tilde t=0}=0. 
$$

\subsubsection{Derivative with respect to $S$} 

We first note that 
\begin{align}
    \left.\frac{\partial a_1}{\partial S}
    \right|_{\tilde t=0}=\left.\frac{\partial a_2}
    {\partial S}\right|_{\tilde t=0}=\left.\frac{\partial 
    A_2}{\partial S}\right|_{\tilde t=0}
    =\left.\frac{\partial \hat A_1}{\partial S}
    \right|_{\tilde t=0}=0. 
    \label{Eqn:aCoeffsPartialS}
\end{align} 

Furthermore,  we have 
\begin{align}
    &\left.\frac{\partial B}{\partial S}
    \right|_{\tilde t=0}=(I+\tau(\zeta)D_x\bfw) 
    r_1=0, 
    \label{Eqn:partialS_B}
    \\
    &\left.\frac{\partial (p\times q)}{\partial S}
    \right|_{\tilde t=0}=0. 
    \label{Eqn:partialS_pq}
\end{align} 
Thus, it follows that
$$
    \left.\frac{\partial H_2}{\partial S}
    \right|_{\tilde t=0}=0. 
$$

\subsection{Row 2-4 Derivatives}

By construction, we easily have 
\begin{align*}
    &\frac{\partial}{\partial\hat\sigma} 
    (S-\hat\sigma)=\frac{\partial}{\partial \hat\beta_1} 
    (Z_1-\hat\beta_1)=\frac{\partial}{\partial \hat\beta_2} 
    (Z_2-\hat\beta_2)=-1, 
    \\
    &\frac{\partial}{\partial S} 
    (S-\hat\sigma)=\frac{\partial}{\partial Z_1} 
    (Z_1-\hat\beta_1)=\frac{\partial}{\partial Z_2} 
    (Z_2-\hat\beta_2)=1, 
\end{align*} 
and all other relevant partial derivatives should vanish. 

\subsection{Row 5 Derivatives} 

We now seek to compute the derivatives 
$
    \frac{\partial F_\sigma}
    {\partial\hat\sigma },   
    \frac{\partial F_\sigma}{\partial S}
$
at the boundary. We first recall the form of $F_\sigma$
$$
    F_\sigma=\delta_\zeta\left( 
    \frac{l_1^\top F}{1+t\lambda_1} 
    -l_1^\top J_\beta \P^{-1} 
    L_*^\top F 
    \right) 
    +\frac{1}{\hat\eta}\sqrt{\frac{\delta_\zeta}
    {\delta_\beta}}\left( 
    \frac{l_1^\top G_0}{1+t\lambda_1} 
    -l_1^\top J_\beta \P^{-1} 
    L_*^\top G_0 
    \right),
$$
which  consists of four terms
\begin{align*}
    &\F_1:=\delta_\zeta\frac{l_1^\top F}
    {1+t\lambda_1} , 
    \quad
   \F_2:=\delta_\zeta l_1^\top J_\beta 
    \P^{-1}L_*^\top F , 
    \quad \F_3:=\frac{1}{\hat\eta}\sqrt{\frac{\delta_\zeta}
    {\delta_\beta}}\frac{l_1^\top G_0}{1+t\lambda_1},  
    \\
    &\F_4:=\frac{1}{\hat\eta}\sqrt{\frac{\delta_\zeta}
    {\delta_\beta}}l_1^\top J_\beta \P^{-1} 
    L_*^\top G_0, 
    \quad
    F_\sigma=\F_1-\F_2+\F_3-\F_4. 
\end{align*} 

\paragraph{Term 1} We first examine $\F_1$. 

Recall
\begin{align*}
    F=-(I+tD_x\bfw)\Lambda_J(\theta_t) 
    -2D_x\bfw J\theta_t 
    -tD^2_x\bfw [J\theta_t,J\theta_t]. 
\end{align*} 
Multiplying by $l_1^\top$ on both sides  gives 
\begin{align*}
    \left.l_1^\top F\right|_{Z=0}=
    -(1+t\lambda_1)l_1^\top 
    \Lambda_J(\theta_t)-2\lambda_1 S 
    -tS^2l_1^\top D^2\bfw [r_1,r_1]. 
\end{align*} 
As noted before, $l_1^\top D^2\bfw [r_1,r_1]$ vanishes at the boundary. Furthermore, we know 
$
    1+t\lambda_1=\O(\delta_\zeta).
$
Thus,
we find that 
$$
    \F_1=-2S+\O(\delta_\zeta), 
$$which 
implies  that 
$$
    \left.\frac{\partial \F_1}{\partial S}
    \right|_{\tilde t=0}=-2, \quad  \left.\frac{\partial\F_1}{\partial\hat\sigma}
    \right|_{\tilde t=0}=0. 
$$

\paragraph{Term 2} We then examine $\F_2$. 

We note that $\delta_\zeta$ does not depend on $S$ or $\hat\sigma$. Furthermore, the latter terms in $\F_2$ are bounded in the limit as $t\rightarrow\tau(\zeta)$. Therefore, 
$$
    \left.\frac{\partial \F_2}{\partial S}
    \right|_{\tilde t=0}=0, \quad  \left.\frac{\partial\F_2}{\partial\hat\sigma}
    \right|_{\tilde t=0}=0.  
$$

\paragraph{Term 3} We then examine $\F_3$. 

This term must be examined carefully because both the numerator and denominator approach zero as $t\rightarrow\tau(\zeta)$. We first note that 
$$
    Z^\pm-\sigma=\pm\sqrt{c_1}\delta_\beta^{1/2} 
    +\left(-\frac{3}{2}\hat\sigma+b^*\right)
    \delta_\beta+\O(\delta_\beta^{3/2}). 
$$
In addition, we have
$$
    l_1^\top B=-\omega_1S\delta_\zeta+\O(\delta_\zeta^2).  
$$
Then, it follows that 
$$
    l_1^\top (\bfv^\pm-y_t)=\mp\omega_1^2 
    \sqrt{\frac{6\delta_\beta}{\omega_3}} 
    +\omega_1\left(\tilde\psi+\frac{3}{2}\hat \sigma 
    +S\right)\delta_\beta +\O(\delta_\beta^{3/2}), 
$$
and
$$
    l_1^\top G=\frac{4\mu_0\omega_1^3}{\omega_3} 
    \delta_\beta^{1/2}\left( 
    a_1\sqrt{6\omega_3}\left( 
    \tilde\psi +\frac{3}{2}\hat\sigma+S
    \right)-3\omega_1a_2
    \right)+\O(\delta_\beta). 
$$
Therefore, we obtain 
\begin{align*}
    \F_3=-\left(\tilde\psi +\frac{3}{2}\hat\sigma+S
    \right)+\frac{a_2}{2a_1}\sqrt{c_1} 
    +\O(\delta_\beta). 
\end{align*} 
By utilizing our previous derivation, we further obtain 
$$
    \F_3=-S-\frac{3}{4}\hat\sigma +L_2(\zeta) 
    -\tilde\psi(\zeta)+\O(\delta_\beta). 
$$
Finally, we conclude that  
$$
    \left.\frac{\partial \F_3}{\partial S}
    \right|_{\tilde t=0}=-1, \quad
    \left.\frac{\partial\F_3}{\partial\hat\sigma}
    \right|_{\tilde t=0}=-\frac{3}{4}. 
$$

\paragraph{Term 4} We then examine $\F_4$. 

We first note that 
$
    G_0=0 
$
at the boundary. Furthermore, since  
$
    \left.\frac{\partial X}{\partial S}
    \right|_{\tilde t=0}=0, 
$
we obtain 
$
    \left.\frac{\partial G_0}{\partial S}
    \right|_{\tilde t=0}=0. 
$
Then, due to the fact that 
$
    \left.\frac{\partial\sigma}{\partial\hat\sigma}
    \right|_{\tilde t=0}=0, 
$
we obtain 
$
    \left.\frac{\partial\F_4}{\partial\hat\sigma}
    \right|_{\tilde t=0}=0 . 
$
Combining all of this gives 
$$
    \left.\frac{\partial F_\sigma}{\partial S}
    \right|_{\tilde t=0}=-2-1=-3, 
    \quad
    \left.\frac{\partial F_\sigma}{\partial\hat\sigma} 
    \right|_{\tilde t=0}=-\frac{3}{4}. 
$$

\subsection{Row 6-7 Derivatives} 

We first note that 
$
    \tilde t Z_t=G_\beta+\tilde tF_\beta. 
$
At time $\tilde t=0$, the last term vanishes. We thus can compute the derivatives of $G_\beta$.

\subsubsection{$\partial_Z G_\beta$} 

At $\tilde t=0$, $\bfu=\bfu_0$. We recall the form of $G_\beta$. 
$$
    G_\beta=\frac{1}{\hat\eta}\sqrt{\frac{\delta_\zeta}
    {\delta_\beta}}\P^{-1}L_*^\top G_0. 
$$
We seek to differentiate this with respect to $Z_1$ and evaluate it at $\tilde t=0$. We note that at the boundary, the derivative of 
$
    \frac{1}{\hat\eta}\sqrt{\frac{\delta_\zeta}
    {\delta_\beta}}\P^{-1}L_*^\top 
$
is bounded, whereas $G_0=0$. We therefore compute 
$$
    \frac{\partial G_\beta}{\partial Z_1} 
    =\frac{1}{\rho_0}\P_0^{-1}L_*^\top 
    \frac{\partial G_0}{\partial Z_1}. 
$$
Recall the form of $G$,
$$
    G:=\rho^+\left((y_t
    -\textbf{v}^+)^\top p\times q\right)(\textbf{v}^+-y_t)
    -\rho^-\left((y_t-\textbf{v}^-)^\top p\times q\right)
    (\textbf{v}^--y_t). 
$$ 

Defining 
$
    X:=B-\hat A_1, 
$
and recalling the definition of $A_2$ gives us 
$$
    G_0=-2a_2(X^\top p\times q)X
    +2a_1(X^\top p\times q)A_2 
    +2a_1(A_2^\top p\times q)X 
    -2a_2\delta_\beta(A_2^\top p\times q)A_2. 
$$
Note that the terms $\delta_\beta,a_1,a_2,A_1,A_2$ do not depend on $Z_1,Z_2$, while the term $p\times q$ does depend on $Z_1,Z_2$. However, it will be not relevant since it is multiplied by terms that vanish at $\tilde{t}=0$. Bearing in mind that all computations below should be understood along the boundary $t=\tau(\zeta)$. Hence, it holds that
$$
    \frac{\partial X}{\partial Z_i}=\frac{\partial B}
    {\partial Z_i}. \qquad i=1,2.
$$

Then,
we consider the quantity
$$
    B=(I+tD_x\bfw)(r_1S+x_{\beta_1}Z_1+x_{\beta_2}Z_2). 
$$
A direct computation gives 
$$
    \frac{\partial B}{\partial Z_i}= 
    (I+\tau(\zeta)D_x\bfw)R_{0,i}. 
$$
A similar computation also holds for $Z_2$. Therefore, 
$$
    \frac{\partial B}{\partial Z}= 
    (I+\tau(\zeta)D_x\bfw)R_{0}. 
$$

We then show that $X$ vanishes at the boundary. Indeed, at $t=\tau(\zeta)$, $Z=0$. Therefore, 
$$
    B=(I+\tau(\zeta)D_x\bfw)r_1S_0=0. 
$$
Due to this,  the derivative of the first term  vanishes, as it is quadratic. We then examine the second term. We first compute 
$
    p=y_{\zeta_1},\; q=y_{\zeta_2}
$
along the boundary. By their definitions, we have
\begin{align*}
    p=(I+\tau(\zeta)D_x\bfw)(r_1\sigma_{\zeta_1} 
    +J_{\beta}\beta_{\zeta_1}) , 
    \\
    q=(I+\tau(\zeta)D_x\bfw)(r_1\sigma_{\zeta_2} 
    +J_{\beta}\beta_{\zeta_2}) . 
\end{align*}
Furthermore, along the boundary, we differentiate 
\begin{align*}
    &\beta(\tau(\zeta),\zeta)=\zeta 
    \\
    &\beta_t\tau_{\zeta_1}+\beta_{\zeta_1}=e_1 
    \\
    &\beta_t\tau_{\zeta_2}+\beta_{\zeta_2}=e_2  
\end{align*} 
At the boundary, $\beta_t=0$. Therefore, 
$
    \beta_{\zeta_1}=e_1, 
    \beta_{\zeta_2}=e_2. 
$ 
By using this, and the fact that $r_1$ is a right eigenvector, we obtain 
\begin{align}
    &p=(I+\tau D_x\bfw)R_0e_1=\frac{\partial B}
    {\partial Z_1}, 
    \\
    &q=(I+\tau D_x\bfw)R_0e_2=\frac{\partial B}
    {\partial Z_2}. 
\end{align}

Due to the vanishing of $X$ at the boundary, we obtain 
$$
    \frac{\partial}{\partial Z_1}
    (a_1(X^\top p\times q)A_2)= 
    a_1\left(\left(\frac{\partial X}{\partial 
    Z_1}\right)^\top p\times q\right)A_2 
    =a_1(p^\top p\times q)A_2 =0 . 
$$
Similarly, 
$$
    \frac{\partial}{\partial Z_2}
    (a_1(X^\top p\times q)A_2)= 
    a_1\left(\left(\frac{\partial X}{\partial 
    Z_2}\right)^\top p\times q\right)A_2 
   =a_1(q^\top p\times q)A_2 =0. 
$$
We then notice that the derivative of the fourth term also vanishes. Indeed, by the vanishing of $\delta_\beta$, we see that 
$$
    \frac{\partial}{\partial Z_i}\left( 
    a_2\delta_\beta(A_2^\top p\times q)A_2
    \right)=0. 
$$

Thus, the only nonzero contribution comes from the third term. Its derivative is calculated as 
$$
    \frac{\partial}{\partial Z_i}(2a_1(A_2^\top 
    p\times q)X)=2a_1(A_2^\top p\times q)(I+\tau D_x\bfw)R_{0,i}.  
$$
This gives us 
\begin{align*}
    \frac{\partial G_\beta}{\partial Z_i} 
    &=\frac{2a_1(A_2^\top p\times q)}{\rho_0}
    \P_0^{-1}L_*^\top 
    (I+\tau D_x\bfw)R_{0,i} 
    \\
    &=\frac{2a_1(A_2^\top p\times q)}{-4a_1A_2^\top p\times q}
    (L_*^\top (I+\tau D_x\bfw)R_0)^{-1}L_*^\top 
    (I+\tau D_x\bfw)R_{0,i} 
    \\
    &=-\frac{1}{2}e_i 
\end{align*} 
and the computation of the lower block 
$$
    \frac{\partial G_\beta}{\partial Z} 
    =-\frac{1}{2}I_2.
$$

\subsubsection{Other Derivative Terms} 

We then claim that 
$$
    \left.\frac{\partial G_\beta}
    {\partial\hat\eta}\right|_{\tilde t=0} 
    =\left.\frac{\partial G_\beta}
    {\partial\hat\sigma}\right|_{\tilde t=0} 
    =\left.\frac{\partial G_\beta}
    {\partial\hat\beta}\right|_{\tilde t=0} 
    =\left.\frac{\partial G_\beta}
    {\partial S}\right|_{\tilde t=0}=0. 
$$
Indeed, noting that $G_0=0$ at $\tilde t=0$, this reduces to show that 
$$
    \left.\frac{\partial G_0}
    {\partial\hat\eta}\right|_{\tilde t=0} 
    =\left.\frac{\partial G_0}
    {\partial\hat\sigma}\right|_{\tilde t=0} 
    =\left.\frac{\partial G_0}
    {\partial\hat\beta}\right|_{\tilde t=0} 
    =\left.\frac{\partial G_0}
    {\partial S}\right|_{\tilde t=0}=0. 
$$
Noting the form of $G_0$, 
\begin{align*}
    G_0=-2a_2(X^\top p\times q)X
    +2a_1(X^\top p\times q)A_2 
    +2a_1(A_2^\top p\times q)X 
    -2a_2\delta_\beta(A_2^\top p\times q)A_2. 
\end{align*} 
it is immediately clear that 
$$
    \left.\frac{\partial G_0}
    {\partial\hat\eta}\right|_{\tilde t=0}=0. 
$$
Furthermore, (\ref{Eqn:aCoeffsPartialS})--(\ref{Eqn:partialS_pq}) imply that 
$$
    \left.\frac{\partial G_0}{\partial S} 
    \right|_{\tilde t=0}=0. 
$$

Note that  
$
    \tilde\sigma=\frac{\delta_\zeta}
    {\delta_\beta}\hat\sigma, 
$
we obtain the relation 
$$
    \frac{\partial\tilde\sigma}{\partial\hat\sigma} 
    =\frac{\delta_\zeta}{\delta_\beta}. 
$$
Therefore, 
$$
    \frac{\partial\hat A_1}{\partial\hat\sigma} 
    =\delta_\beta\frac{\partial}{\partial\hat\sigma} 
    (A_1-A_3) 
    =\delta_\zeta\frac{\partial}{\partial\tilde\sigma} 
    (A_1-A_3), 
$$
which is $\delta_\zeta$ times the derivative of an analytic function. Therefore, 
$$
    \left.\frac{\partial\hat A_1}{\partial\hat\sigma} 
    \right|_{\tilde t=0}=0. 
$$
Combining this with (\ref{Eqn:partialSigmaHatB}) gives 
$$
    \left.\frac{\partial X}{\partial\hat\sigma} 
    \right|_{\tilde t=0}=0, 
$$
which, in conjunction with the fact that 
$$
    X=B-\delta_\beta(A_1-A_3)=\O(\delta_\zeta), 
$$
implies 
$$
    \left.\frac{\partial G_\beta}
    {\partial\hat\sigma}\right|_{\tilde t=0}=0. 
$$

We finally compute the derivative with respect to $\hat\beta$. Note that 
$\beta=\zeta+\tilde t\hat\beta,
    \frac{\partial\beta}{\partial\hat\beta_j} 
    =\tilde t e_j. 
$
Therefore, since $B$ depends on $\beta$, instead of $\hat\beta$, its derivative vanishes, 
$$
    \left.\frac{\partial B}{\partial\hat\beta_j} 
    \right|_{\tilde t=0}=\left.\frac{\partial B}{\partial\beta_j}\tilde te_j 
    \right|_{\tilde t=0}=0. 
$$
Furthermore, we have
$$
    \frac{\partial \delta_\beta}
    {\partial\hat\beta_j}=-(\nabla_\beta\tau(\beta)) 
    \cdot \frac{\partial\beta}{\partial\hat\beta_j} =-\tilde t(\nabla_\beta\tau(\beta)) 
    \cdot e_j,  
$$
which also vanishes in the limit as $\tilde t\rightarrow0$. Thus, we have the following  results  
$$
    \left.\frac{\partial X}{\partial\hat\beta} 
    \right|_{\tilde t=0}=\left.\frac{\partial}
    {\partial\hat\beta}\left( 
    B-\hat A_1
    \right)\right|_{\tilde t=0}=0, 
$$which further
implies that 
$$
    \left.\frac{\partial G_0}{\partial\hat\beta}
    \right|_{\tilde t=0}=0.  
$$
Finally, this completes the calculations. 

\end{appendices}

\v
{\bf Acknowledgments.} 
The authors would like to thank Professor Alberto Bressan for valuable suggestions.
The work by the first and second authors was partially supported by the NSF with grants 
DMS-2306258 and DMS-2605028. The work of third author was in part supported by Zhejiang Normal University with
grant YS304222929. 

{\bf Data Availability.} Data sharing is not applicable to this article as no new data were created or analyzed in this
study

\v

{\bf Declarations
Conflict of interest.} The authors declare  no conflict of interest.

\v


\begin{thebibliography}{999}



\bibitem{A} S.~Alinhac, {\it Blowup for Nonlinear Hyperbolic Equations.}
Birkh\"auser, Boston, 1995.

\bibitem{Al}S. Alinhac,  Blowup of small data solutions for a class of quasilinear wave equations in two
space dimensions. II {\it Acta Math.} {\bf 182} (1999) 1–23.

\bibitem{Ali}S. Alinhac, Blowup of small data solutions for a quasilinear wave equation in two space
dimensions, {\it Ann. Math.} {\bf 149} (1999) 97–127.




%

\bibitem{Bloom} J.~M.~Bloom, The local structure of smooth maps of manifolds.
B.A.~thesis, Harvard 2004.
www.math.harvard.edu/theses/phd/bloom/ThesisXFinal.pdf

\bibitem{BD} S.~Bianchini and S.~Daneri, On the sticky particle solutions to the multi-dimensional pressureless Euler equations, {\it J. Differ. Equations} {\bf 368} (2023) 173-202

\bibitem{Brenier} Y. Brenier and E. Grenier, Sticky particles and scalar conservation laws [J]. SIAM Journal on Numerical Analysis, 1998, 35: 2317–2328.
%
%

 \bibitem{BC} A.~Bressan and G.~Chen,
Generic regularity of conservative solutions
to a nonlinear  wave equation,
{\it Ann. Inst. Poincar\'e Anal. Nonlin.} {\bf 34} (2017) 335--354.


 \bibitem{BCH} A.~Bressan, G.~Chen, and S.~Huang, Generic singularities of 2D
 pressureless flow. {\it Science China Math.},  {\bf68} (2025) 559-576. 

 \bibitem{BCH2} A.~Bressan, G.~Chen, and S.~Huang, The Initial Stages of a Generic Singularity
for a 2D Pressureless Gas. {\it Commun. Math. Phys.},  {\bf406} (2025). 

 \bibitem{BHY} A.~Bressan, T.~Huang, and F.~Yu,  Structurally stable singularities
for a nonlinear   wave equation. {\it Bull. Inst. Math. Acad. Sinica},
{\bf 10} (2015), 449--478.

 \bibitem{BN} A.~Bressan and T.~V.~Nguyen, Non-existence and non-uniqueness for multidimensional
 sticky particle systems.   {\it Kinetic Rel. Models} {\bf 7} (2014),
 205--218.


\bibitem{Vicol0} T. Buckmaster, T. Drivas, S. Shkoller and V. Vicol, Simultaneous development of
    shocks and cusps for 2D Euler with azimuthal symmetry from smooth data, \emph{Ann. PDE}, {\bf 8} (2022), no. 2, pp. 199.

\bibitem{Vicol1} T. Buckmaster, S. Shkoller and V. Vicol, Formation of shocks for 2d isentropic
    compressible Euler, \emph{Commun. Pure Appl. Math.}, {\bf 75} (2022), no. 9, 2069--2120.

\bibitem{BSV} T. Buckmaster, S. Shkoller and V. Vicol, Shock formation and vorticity creation for
    3d Euler, \emph{Commun. Pure Appl. Math.}, {\bf 76} (2023), no. 9, 1965--2072.

\bibitem{Vicol3} T. Buckmaster, S. Shkoller and V. Vicol, Formation of point shocks for 3D
    compressible Euler, \emph{Commun. Pure Appl. Math.}, {\bf 76} (2023), no. 9, 2073--2191.


\bibitem{CCD} H.~Cai, G.~Chen and Y.~Du,
Uniqueness and regularity of conservative solution to a wave system modeling nematic liquid crystal, {\it J. Math. Pures Appl.} {\bf  117} (2018), 185--220.

\bibitem{CCW} H. Cai, G. Chen and T. Wang, Singularity formation for radially symmetric
    expanding wave of Compressible Euler Equations, {\it SIAM J. Math. Anal.}, {\bf 55} (2023), no.
    4, 2917--2947.

\bibitem{G9} G. Chen, Optimal time-dependent lower bound on density for
classical solutions of 1-D compressible Euler equations, {\it Indiana Univ. Math. J.}, {\bf 66}
(2017), no. 3, 725--740.




\bibitem{CEHS} G. Chen, F. A. El-Katri, Y. Hu and Y. Shen, Global solution and singularity formation for the supersonic expanding wave of compressible Euler equations with radial symmetry,
arXiv:2404.07830v2.

\bibitem{CHS}	G. Chen, Y. Hu and Y. Shen, Structurally stable singularities and Lipschitz stable optimal transport metrics for the compressible Euler equations, \emph{submitted},
available at  arXiv:2609.13943.





%

\bibitem{Ch1} D. Christodoulou, {\it The Shock Development Problem}, EMS Monographs in Mathematics,
European Mathematical Society (EMS), Z\"urich, 2019.



\bibitem{Damon} J.~Damon, Generic properties of solutions to partial differential equations
{\it Arch. Rational Mech. Anal.} {\bf 140} (1997), 353--403.

\bibitem{D} V.~G.~Danilov, Interaction of $\delta$-shock waves in a system of pressureless gas dynamics equations.
{\it Russ. J. Math. Phys.} {\bf 26} (2019),  306--319.


\bibitem{DM} V.~G.~Danilov and  D.~Mitrovic,
Shock wave formation process for a multidimensional scalar conservation law.
{\it Quart. Appl. Math.} {\bf 69} (2011), 613--634.

%


\bibitem{ERS} W.~E, Yu.~G.~Rykov, and  Ya.~G.~Sinai,
Generalized variational principles, global weak solutions and behavior with random initial data for systems of conservation laws arising in adhesion particle dynamics.
{\it Comm. Math. Phys.} {\bf 177} (1996), 349--380.


%

%
%


\bibitem{GG}  M.~Golubitsky and  V.~Guillemin,
{\it Stable Mappings and Their Singularities.} Springer-Verlag, New York, 1973.

\bibitem{GS} M.~Golubitsky and D.~Schaeffer,
 Stability of shock waves for a single conservation law.
{\it Adv. Math.} {\bf 15} (1975),  65-71.



\bibitem{GSS} S.~N.~Gurbatov, A.~I.~Saichev, and S.~F.~Shandarin, Large-scale structure of the Universe. The Zeldovich
approximation and the adhesion model. {\it  Phys. Usp.}  {\bf 55} (2012),  223--249.

\bibitem{Guck}  J.~Guckenheimer, Catastrophes and Partial Differential Equations.
{\it Ann. Inst. Fourier} {\bf 23} (1973), 31--59.


\bibitem{Gk2}
J.~Guckenheimer, Solving a single conservation law. In
``Dynamical systems - Warwick 1974'', pp. 108--134.
Springer Lecture Notes in Math, {\bf 468}.

\bibitem{hw}F. Huang, Z. Wang, Well posedness for pressureless flow, {\it Commun. Math. Phys.}, {\bf 222} (2001),  117 -146.


\bibitem{Kong}
D.X.~Kong,
Formation and propagation of singularities for $2\times 2$ quasilinear hyperbolic systems.
{\it Trans. Amer. Math. Soc.} {\bf  354}  (2002), 3155--3179.
%
%


\bibitem{li}F.B. ~Li, On systems of partial differential equations of
Briot-Bouquet type. {\it Bull. Aust. Math. Soc.} {\bf  98}  (2018), 122--133.

    
    
    

    

    



\bibitem{S} D.~Schaeffer, A regularity theorem for conservation laws,
{\it Adv. in Math.} {\bf  11}
(1973), 368--386.

\bibitem{Sever}
M.~Sever, An existence theorem in the large for zero-pressure gas dynamics.  {\it Diff. Integral Equat.}  \textbf{14}  (2001), 1077--1092.

%






\bibitem{Z} Ya.~B.~Zeldovich, Gravitational instability: An approximate theory for large density perturbations, {\it Astron. Astrophys.} {\bf 5} (1970), 84--89.


\end{thebibliography}
\end{document}